\documentclass[12pt,a4paper]{amsart}

\usepackage{enumerate}
\usepackage{amsmath,amsthm,verbatim,amssymb,amsfonts,amscd,graphicx}
\usepackage{mathtools}
\usepackage{graphics}
\usepackage{hyperref}
\usepackage{mathrsfs}
\usepackage[left=15mm,top=15mm,right=15mm,bottom=15mm]{geometry}
\usepackage[latin2]{inputenc}
\usepackage{tikz-cd}

\theoremstyle{plain}
\newtheorem{theorem}{Theorem}[section]
\newtheorem{corollary}[theorem]{Corollary}
\newtheorem*{corollary*}{Corollary}
\newtheorem{lemma}[theorem]{Lemma}
\newtheorem{proposition}[theorem]{Proposition}

\newtheorem{conjecture}[theorem]{Conjecture}

\newtheorem{example}[theorem]{Example}
\newtheorem*{example*}{Example}
\newtheorem{remark}[theorem]{Remark}

\theoremstyle{definition}
\newtheorem{definition}[theorem]{Definition}
\newtheorem*{definition*}{Definition}
\numberwithin{equation}{theorem}

\newcommand{\mm}{\mathfrak{m}}
\newcommand{\NN}{\mathbb{N}}

\begin{document}
\title{Density functions for filtrations of graded ideals}
\author{Suprajo Das}
\address{Mathematical and Physical Sciences Division, School of Arts and Sciences, Ahmedabad University, Navrangpura, Ahmedabad, Gujarat 380009, India}
\email{suprajo.das@ahduni.edu.in}
\author{Hoang Le Truong}
\address{University of Engineering and Technology, Vietnam National University, Hanoi, Vietnam}
\email{hltruong@vnu.edu.vn}
\maketitle

\begin{abstract}
We study density functions associated to filtrations of graded ideals in standard graded domains. Extending earlier work, we develop a general theory for arbitrary filtrations and establish their fundamental properties, including existence, continuity, log-concavity, and integral representations. We show that every density function is locally uniformly approximable by piecewise polynomial density functions arising from truncated Noetherian filtrations. We prove differentiability properties of saturation density functions and give a numerical characterization of when the symbolic powers of a graded radical ideal in a polynomial ring coincide with the integral closures of its ordinary powers. Finally, we establish several inequalities for the mixed multiplicities of equigenerated ideals.
\end{abstract}

\section{Introduction}

An important theme in commutative algebra and algebraic geometry is the extraction of continuous invariants from discrete algebraic data. A useful instance of this idea is the notion of a \emph{density function}: a measurable function on $\mathbb{R}$ that encodes an algebraic invariant on an $\mathbb{R}$-scale, typically retaining far more information than the invariant alone. In commutative algebra, density functions were introduced by Trivedi \cite{Tri18} to study the Hilbert-Kunz multiplicity of graded ideals of finite colength, and have since emerged to be a useful and versatile tool. Some recent papers on this topic include \cite{MT19}, \cite{MT20}, \cite{Tri20}, \cite{TW21}, \cite{TW23}, \cite{CSTZ24}, \cite{MM25}, \cite{DRT25}, and \cite{DRT26}. Constructions of this nature play an equally prominent role in algebraic geometry, particularly in the theory of $K$-stability of Fano varieties, see \cite{Li17}, \cite{Fuj18}, and \cite{BJ20}.

Throughout this article, let $R=\oplus_{m\geq 0}R_m$ be a $d$-dimensional standard graded domain over an algebraically closed field $R_0=k$, and let $\mm_R = \oplus_{m\geq 1}R_m$ denote its graded maximal ideal. Given a filtration $\mathbb{I}=\{I_n\}_{n\in\mathbb N}$ of nonzero graded ideals in $R$, its density function $f_{\mathbb{I}} \colon \mathbb{R} \to \mathbb{R}$ is defined by
\[
f_{\mathbb{I}}(x) = \limsup_{n\to\infty} \dfrac{\dim_k {(I_n)}_{\lfloor xn\rfloor}}{n^{d-1}/d!}.
\]
Informally, this function captures the asymptotic growth of the graded components ${\left(I_n\right)}_m$ that are concentrated near rays of slope $x$ in the $(m,n)$-plane.

The first objective of this manuscript is to develop a general theory of density functions for arbitrary (possibly non-Noetherian) filtrations. This extends the work of \cite{DRT25}, where the case of Noetherian filtrations and their saturations was studied.

\subsection*{Limit existence and the Newton-Okounkov body framework}
Our main existence result is Theorem \ref{mainthm}, which forms the core of Section \ref{section 4}. We prove that the density function $f_{\mathbb{I}}$ exists as a genuine limit at every point other than the \emph{Waldschmidt constant} $$\alpha_{\mathbb{I}} = \lim_{n\to\infty}\frac{\min\{m\in\mathbb{N}\mid {(I_n)}_m\neq 0\}}{n}.$$ Moreover, it vanishes identically below $\alpha_{\mathbb{I}}$, and on the open half-line $(\alpha_{\mathbb{I}},\infty)$, it is positive, nondecreasing, log-concave, and continuous, with differentiability outside a countable subset. It also admits an integral representation $$\int_{a}^b f_{\mathbb{I}}(x)dx = \lim_{n\to\infty}\dfrac{\sum_{m=\lfloor an\rfloor}^{\lfloor bn\rfloor}\dim_k {\left(I_n\right)}_m}{n^d/d!},$$ for every compact interval $[a,b]\subset \mathbb{R}$.

The key principle behind these results is the translation of asymptotic algebraic data into convex geometry via \emph{Newton-Okounkov bodies}. Following the foundational work of Okounkov \cite{Oko96}, Lazarsfeld and Musta\c{t}\u{a} \cite{LM09}, Kaveh and Khovanskii \cite{KK12}, Boucksom \cite{Bou14}, and Cutkosky \cite{Cut14}, we associate to each rational slope $x=p/q>\alpha_{\mathbb{I}}$, a strongly nonnegative semigroup $\Gamma^{\mathbb{I}}_{(p,q)} \subseteq \mathbb{N}^{d+1}$, constructed from a rank one monomial valuation dominating $R_{\mathfrak{m}_R}$. The graded components ${(I_{qn})}_{pn}$ are then encoded by the lattice points in the slice $\Gamma^{\mathbb{I}}_{(p,q)} \cap \left(\mathbb{N}^d \times \{qn\}\right)$, and general results on valuation semigroups show that the asymptotics of these lattice-point counts are governed by the Euclidean volumes of the associated Newton-Okounkov bodies $\Delta_x(\mathbb{I})$. An important outcome is the existence of a \emph{global Newton-Okounkov body} $\Delta(\mathbb{I}) \subseteq \mathbb{R}^{d+1}$, whose horizontal slices $\Delta_x(\mathbb{I})$ are the compact convex bodies governing the density function at each slope $x$, and whose truncated volume controls the integrated density function, see Remark \ref{global_okounkov}. From this interpretation, the log-concavity of $f_{\mathbb{I}}$ is a consequence of the Brunn-Minkowski inequality applied to these slices, continuity follows from log-concavity, and the integral representation is deduced from the Lebesgue's dominated convergence theorem. Related ideas have appeared previously in \cite{CMM24}.

These constructions can be related to the notion of the \emph{concave transform} introduced by Boucksom and Chen \cite{BC11} in the context of filtered graded linear series. We also propose a notion of density functions for filtrations of ideals in certain local rings; see Section \ref{local_density}. The idea is to reduce to the graded setting by passing to the associated graded ring, after which Theorem \ref{mainthm} applies.

\subsection*{Noetherian approximation and polynomial behavior}
The passage from Noetherian to arbitrary filtrations is mediated by an approximation theorem. For every integer $a\ge1$, we associate to $\mathbb{I}$ its \emph{$a$-th truncated filtration} $\mathbb{I}^{[a]}$, the Noetherian filtration generated by the ideals $I_1,\ldots,I_a$. Theorem \ref{approx_noeth_filt} establishes that the density functions $\{f_{\mathbb{I}^{[a]}}\}_{a=1}^{\infty}$ converge locally uniformly to $f_{\mathbb{I}}$. Since each $f_{\mathbb{I}^{[a]}}$ is piecewise polynomial by \cite{DRT25}, this realizes $f_{\mathbb{I}}$ as a locally uniform limit of piecewise polynomial functions. The idea of employing truncated filtrations is adapted from \cite{CSS19}.

We present two applications of this approximation result. The first is Theorem \ref{even_linear}, which shows that if $\mathbb{I}$ has linearly bounded generating degrees, then its density function $f_{\mathbb{I}}$ eventually coincides with a real polynomial of degree $d-1$. Every Noetherian filtration satisfies this hypothesis, as do certain non-Noetherian ones; see \cite{HHT02}. The second is Theorem \ref{density_integral}, which shows that density functions are invariant under taking integral closures. For Noetherian filtrations, the converse implication is known by \cite{DRT26}. More general statements for $\mathfrak{m}_R$-primary filtrations in local rings can be found in \cite{Cut21} and \cite{BLQ24}.

\subsection*{Saturation filtrations and differentiability of density functions}

For an arbitrary filtration $\mathbb{I}=\{I_n\}_{n\in\mathbb N}$, the density function $f_{\mathbb{I}}$ has at most countably many points of non-differentiability, though this set can still be dense. This motivates the study of filtrations whose density functions are differentiable away from a finite set of points. The property holds automatically for Noetherian filtration, since in that case $f_{\mathbb{I}}$ is a piecewise polynomial function with only finitely many pieces. We suspect that it persists for saturated filtrations, even though they need not be Noetherian. In fact, there are examples of graded ideals $I$ in a polynomial ring $R$ for which the saturation filtration $\left\{I^n \colon \mathfrak{m}_R^{\infty}\right\}_{n\in\mathbb{N}}$ fails to be Noetherian, see \cite{Nagata59} and \cite{CHST05}. Based on our results we are led to the following conjecture:

\begin{conjecture}\label{conj1}
Let $\mathbb{I}=\{I_n\}_{n\in\mathbb{N}}$ be a filtration of graded ideals in $R$, and $J\subseteq \mathfrak{m}_R$ be a graded ideal. Consider the saturation filtration $\mathbb{I}^{\mathrm{sat}}_J = \left\{I_n \colon J^{\infty}\right\}_{n\in\mathbb{N}}$. Then the corresponding saturation density function $f_{\mathbb{I}^{\mathrm{sat}}_J}$ is continuous on $\mathbb{R}$, and continuously differentiable away from finitely many points.
\end{conjecture}

We are able to prove this conjecture under any of the following hypotheses:
\begin{enumerate}
 \item[$(i)$] $R$ is a polynomial ring over a field (Theorem \ref{saturated_density_diff}).
 \item[$(ii)$] $\mathbb{I}=\{I_n\}_{n\in\mathbb{N}}$ is a Noetherian filtration of graded ideals in $R$ (Theorem \ref{symb_geom}).
 \item[$(iii)$] $\dim R = 2$ (Theorem \ref{dimension_two}).
\end{enumerate}

For polynomial rings $R$, the argument proceeds by passing to generic initial ideals with respect to the reverse lexicographic order, which preserve Hilbert functions and also commute with saturations. The latter property generally fails for ordinary initial ideals. This transforms the problem into analyzing a filtration of monomial ideals in one fewer variable, allowing us to represent the saturation density function as the integral of a lower-dimensional density function. The continuous differentiability of $f_{\mathbb{I}^{\mathrm{sat}}_J}$ on $(\alpha_{\mathbb{I}^{\mathrm{sat}}_J},\infty)$ then follows. This approach is inspired by \cite{May14} and \cite{DST15}.

\subsection*{Generalized symbolic powers and birational geometry} Conjecture \ref{conj1} is also valid for arbitrary standard graded domains $R$ provided $\mathbb{I}=\{I_n\}_{n\in\mathbb{N}}$ is a Noetherian filtration. After passing to a suitable Veronese, the filtration may be reduced to the adic filtration $\mathbb{I}=\{I^n\}_{n\in\mathbb{N}}$ associated to a single graded ideal $I\subseteq R$. In this case, Theorem \ref{symb_geom} gives a geometric realization of the saturation density function $f_{\mathbb{I}^{\mathrm{sat}}_J}$. More precisely, there exist a normal projective variety $X$, a projective birational morphism $\varphi \colon X \to \mathrm{Proj}\;R$ dominating the blow-up of $\mathrm{Proj}\;R$ along the ideal sheaf associated to $I$, and an effective Cartier divisor $F$ on $X$ with exceptional support, such that $$f_{\mathbb{I}^{\mathrm{sat}}_J}(x) = d\cdot \mathrm{vol}_X(xH-F)\quad \forall x\in \mathbb{R}.$$ Here, $H$ is the pullback of a hyperplane section to $X$, and $\mathrm{vol}_X$ is the volume function on $N^1(X)_{\mathbb{R}}$ in the sense of Lazarsfeld \cite{Laz04a}. This description identifies the Waldschmidt constant with the pseudoeffective threshold: $$\alpha_{\mathbb{I}^{\mathrm{sat}}_J} = \min\left\{x\in \mathbb{R} \mid xH-F\;\text{is pseudoeffective}\right\}.$$ The continuity of $f_{\mathbb{I}^{\mathrm{sat}}_J}$ then follows from the continuity of the volume function established by Lazarsfeld \cite{Laz04a}. Furthermore, the continuous differentiability of $f_{\mathbb{I}^{\mathrm{sat}}_J}$ on $(\alpha_{\mathbb{I}^{\mathrm{sat}}_J},\infty)$ follows from the $C^1$-differentiability of the volume function on the big cone, proved independently by Lazarsfeld and Musta\c{t}\u{a} \cite{LM09}, and Boucksom, Favre, and Jonsson \cite{BFJ09}.

The geometric formulation suggests several natural questions and potential applications. For instance, one may ask whether the saturation density function $f_{\mathbb{I}^{\mathrm{sat}}_J}$ is a piecewise polynomial. In dimension two, the answer is affirmative under more general hypotheses, see Theorem \ref{dimension_two}. In dimension three, this is established in Theorem \ref{symb_geom} though with potentially infinitely many pieces. The proof relies crucially on the locally polynomial behavior of volume functions on nonsingular projective surfaces \cite{BKS04}. In higher dimensions, the situation is less clear, owing to the pathological behavior of volume functions, as illustrated in \cite{Cut86}, \cite{BKS04}, and \cite{KLM13}.

The special case $J=\mathfrak{m}_R$ was treated earlier in \cite{DRT25}. Extending the analysis to arbitrary ideals $J$ involves additional technical difficulties.

\subsection*{Symbolic powers and integral closures}
For a given ideal $I$, the symbolic powers $\{I^{(n)}\}_{n\in\mathbb{N}}$ of $I$ can be recovered as a saturation filtration $\left\{I^n \colon J^{\infty}\right\}_{n\in\mathbb{N}}$ for a suitable choice of an ideal $J$. The problem of characterizing ideals whose symbolic and ordinary powers coincide is classical and notoriously difficult. As an application, we obtain a Rees-type numerical characterization for the equality of the symbolic powers of a graded radical ideal in a polynomial ring with the integral closures of its ordinary powers; see Corollary \ref{sym_reg}. Our main ingredient is a result from \cite{FKL16} concerning the equality of volume functions of divisors.

\subsection*{Mixed multiplicities of filtrations of equigenerated ideals}
The study of mixed multiplicities of $\mathfrak{m}_R$-primary ideals in a Noetherian local ring $(R,\mathfrak{m}_R)$ goes back to the foundational work of Bhattacharya \cite{Bha57}, Rees \cite{Ree61b}, and Risler and Teissier \cite{Tei73}. A generalization was obtained by Katz and Verma \cite{KV89}, who extended the notion to collections of ideals that are not necessarily $\mathfrak{m}_R$-primary. Trung and Verma \cite{TV07} gave a polyhedral description of mixed multiplicities for monomial ideals in terms of mixed volumes of suitable polytopes. These invariants have also played a role in combinatorics: Huh \cite{Huh12} employed mixed multiplicities in his analysis of the coefficients of chromatic polynomials of graphs.

In the final section \ref{section 8}, we study the (generalized) mixed multiplicities $\{e_i(\mathfrak{m}_R\vert I)\}_{i=1}^{d-1}$, where $I\subseteq R$ is a nonzero graded ideal generated in a single degree. The key ingredient is Proposition \ref{prop_intersect} that realizes these mixed multiplicities as intersection numbers of nef divisors on a suitable projective variety. This interpretation allows us to derive inequalities among these mixed multiplicities from known inequalities for intersection numbers of nef divisors. In Corollary \ref{cor_huh}, the Khovanskii-Teissier inequalities \cite[Section 1.6.A]{Laz04a} yield that $\{e_i(\mathfrak{m}R\vert I)\}_{i=1}^{d-1}$ is log-concave with no internal zeros, thus recovering a result of Huh \cite{Huh12}. Our proof differs from that of Huh as we only use intersection theory of divisors, which is easier compared to that of algebraic cycles. In Corollary \ref{cor_rev}, the reverse Khovanskii-Teissier inequality \cite{JL23} is used to obtain a new inequality for mixed multiplicities. It is well-known that the epsilon multiplicity is bounded above by the $j$-multiplicity \cite{UV08, UV11}, while little is known about lower bounds. In Corollary \ref{cor_disk}, we apply Diskant's inequality \cite{BFJ09} to obtain a sharp, computable lower bound for the epsilon multiplicity of an equigenerated ideal. These results extend to filtrations of equigenerated ideals via ``volume=multiplicity''-type formulas, see Proposition \ref{prop_filt}.

The article is organized as follows: Section \ref{cones} collects preliminary material on cones associated to semigroups, following \cite{LM09} and \cite{KK12}. Section \ref{section 4} develops the general theory of density functions for arbitrary filtrations and proves the main existence theorem (Theorem \ref{mainthm}), including the convex-geometric description via Newton-Okounkov bodies. Section \ref{section 5} establishes the Noetherian approximation result (Theorem \ref{approx_noeth_filt}), eventual polynomial behavior under linear growth assumptions (Theorem \ref{even_linear}), and invariance of density functions under integeral closures (Theorem \ref{density_integral}). Section \ref{section 6} treats saturation filtrations and their differentiability properties, answering Conjecture \ref{conj1} for polynomial rings. Section \ref{section 7} studies the density functions for generalized symbolic powers and their birational geometry (Theorem \ref{symb_geom}), their behavior in low dimensions (Theorems \ref{dimension_two} and \ref{dimension_three}), and the equality of these functions with those arising from ordinary powers (Theorem \ref{adic=symb}). In Section \ref{section 8}, we interpret the mixed multiplicities of equigenerated ideals as intersection numbers (Proposition \ref{prop_intersect}) and use this interpretation to derive several inequalities for these mixed multiplicities (Corollaries \ref{cor_huh}, \ref{cor_rev}, and \ref{cor_disk}). These results are also extended to filtrations of equigenerated ideals via ``volume=multiplicity''-type formulas (Proposition \ref{prop_filt}).

\section{Cones associated to semigroups}\label{cones}

Here, we summarize some results of Kaveh and Khovanskii \cite{KK12}.

Suppose that $S$ is an additive subsemigroup of $\mathbb{Z}^d\times \mathbb{N}$ which is not contained in $\mathbb{Z}^d\times \{0\}$. Let $L(S)$ be the $\mathbb{R}$-subspace of $\mathbb{R}^{d+1}$ which is generated by $S$. Let $M(S) = L(S)\cap \left(\mathbb{R}^d\times \mathbb{R}_{\geq 0}\right)$ with boundary $\partial M(S) = L(S)\cap \left(\mathbb{R}^d\times \{0\}\right)$. Let $\mathrm{Con}(S) \subset L(S)$ be the closed convex cone spanned by $S$, i.e., $$\mathrm{Con}(S) = \text{closure of}\;\;\Big\{\sum_{i}a_i\mathbf{v}_i \mid a_i\in\mathbb{R}_{\geq 0}\;\;\text{and}\;\;\mathbf{v}_i\in S\Big\}.$$

$S$ is called \emph{strongly nonnegative} (\cite[Section 1.4]{KK12}) if $\mathrm{Con}(S)$ intersects $\partial M(S)$ only at the origin (this is equivalent to being strongly admissible \cite[Definition 1.9]{KK12}) since with our assumptions, $\mathrm{Con}(S) \subset \mathbb{R}^d\times \mathbb{R}_{\geq 0}$, so the ridge of $S$ (i.e., the biggest linear subspace contained in $\mathrm{Con}(S)$) must be contained in $\partial M(S)$. We now introduce some more notations from \cite[Section 1.4]{KK12}.
\begin{align*}
 S_n &= S \cap \left(\mathbb{R}^d\times \{n\}\right).\\
 q(S) &= \dim \partial M(S) = \dim L(S)-1.\\
 G(S) &= \text{subgroup of $\mathbb{Z}^{d+1}$ generated by $S$}.\\
 m(S) &= \left[\mathbb{Z} \colon \pi(G(S))\right],\;\text{where $\pi\colon \mathbb{R}^{d+1}\to \mathbb{R}$ is the projection map onto the last factor}\\
 \Delta(S) &= \mathrm{Con}(S)\cap \left(\mathbb{R}^d\times \{m(S)\}\right);\;\text{the Newton-Okounkov body of $S$}.\\
 \partial M(S)_{\mathbb{Z}} &= \partial M(S) \cap \mathbb{Z}^{d+1} = L(S)\cap \left(\mathbb{Z}^d\times \{0\}\right).\\
 \mathrm{ind}(S) &= \left[\partial M(S)_{\mathbb{Z}} \colon G(S)\cap \partial M(S)_{\mathbb{Z}}\right] = \left[L(S)\cap \left(\mathbb{Z}^d\times \{0\}\right) \colon G(S)\cap \left(\mathbb{Z}^d\times \{0\}\right)\right].
\end{align*}

If $S$ is strongly nonnegative then $\Delta(S)$ is a compact convex set (i.e., convex body) and $$\Delta(S) = \overline{\bigcup\limits_{n\in\mathbb{Z}_{\geq 1}}\Big\{\tfrac{1}{n}\cdot \mathbf{v}\mid \left(\mathbf{v},m(S)n\right)\in S_{m(S)n}\Big\}}.$$ We write $\mathrm{vol}_{q(S)}\left(\Delta(S)\right)$ for the integral volume of $\Delta(S)$ (see \cite[Definition 1.13]{KK12}). This volume is computed using the translation of the integral measure on $\partial M(S)$.

\begin{theorem}[Kaveh and Khovanskii]\label{KK1} Suppose that $S$ is strongly nonnegative. Then $$\lim_{n\to\infty}\dfrac{\#S_{m(S)n}}{n^{q(S)}} = \dfrac{\mathrm{vol}_{q(S)}\left(\Delta(S)\right)}{\mathrm{ind}(S)}.$$
\end{theorem}
This is proven in \cite[Corollary 1.16]{KK12}. With our assumptions, we have that $S_n = \emptyset$ if $m(S)$ does not divide $n$ and the limit is positive, since $\mathrm{vol}_{q(S)}\left(\Delta(S)\right)>0$.

\begin{theorem}\label{KK2}
Suppose that $q$ is a positive integer such there exists a sequence $n_i \to\infty$ of positive integers such that the sequence $\frac{\#S_{m(S)n_i}}{n_i^q}$ is bounded. Then $S$ is strongly nonnegative with $q(S)\leq q$.
\end{theorem}
This is proven in \cite[Theorem 1.18]{KK12}.

\section{Filtrations of graded ideals in a standard graded domain}\label{section 4}

\subsection{Setup}\label{setup}
Let $k$ be an algebraically closed field. Let $R=\oplus_{m\geq 0}R_m$ be a standard graded finitely generated $k$-algebra with $R_0=k$. Assume that $R$ is a domain of Krull dimension $d\geq 2$. Denote by $\mm_R=\oplus_{m\geq 1}R_m$ the unique graded maximal ideal of $R$.

\begin{definition}\label{def_filt}
 A family of graded ideals $\mathbb{I}=\{I_n\}_{n\in\mathbb{N}}$ in $R$ is called a \emph{filtration} if
 \begin{itemize}
  \item $I_0 = R$,
  \item $I_n \supseteq I_{n+1}$ for all $n\geq 0$, and
  \item $I_m \cdot I_n \subseteq I_{m+n}$ for all $m,n\geq 0$.
 \end{itemize}
All inclusions are understood to be graded. The filtration $\mathbb{I}=\{I_n\}_{n\in\mathbb{N}}$ is said to be \emph{Noetherian} if the Rees algebra $\oplus_{n\geq 0}I_nt^n$ is a finitely generated $R$-algebra.
\end{definition}

We now recall some standard examples of filtrations of graded ideals that will be central to our discussion. An extensive list of examples of filtrations can be found in \cite[Example 2.4.16]{Laz04a}.

The \emph{adic filtration} $\mathbb{I} = \{I^n\}_{n\in\mathbb{N}}$ associated to a graded ideal $I\subseteq R$, is given by the ordinary powers $I^n$ of $I$.

\subsubsection{Integral closure}
An element $x\in R$ is said to be \emph{integral} over $I$ if there exist an integer $n\geq 1$ and elements $a_i\in I^i$ for $i=1,\ldots,n$, such that $$x^n+a_1x^{n-1}+a_2x^{n-2}+\cdots+a_{n-1}x + a_n =0.$$ The set of all elements that are integral over $I$ is called the \emph{integral closure} of $I$ and is denoted $\overline{I}$. It is again a graded ideal. The family $\overline{\mathbb{I}} = \{\overline{I^n}\}_{n\in\mathbb{N}}$ forms a Noetherian filtration of graded ideals in $R$, see \cite[Proposition 5.3.4]{HS06}.

\subsubsection{Saturation}
The \emph{saturation} of $I$ (with respect to $\mathfrak{m}_R$) is defined by $$I \colon \mathfrak{m}_R^{\infty} := \bigcup_{n\geq 1} \left(I \colon \mathfrak{m}_R^{n}\right).$$ Equivalently, $I \colon \mathfrak{m}_R^{\infty}$ is the ideal obtained by simply removing the $\mathfrak{m}_R$-primary component from a (graded) primary decomposition of $I$. The family $\mathbb{I}^{\mathrm{sat}} = \{I^n \colon \mathfrak{m}_R^{\infty}\}_{n\in\mathbb{N}}$ forms a (possibly non-Noetherian) filtration of graded ideals in $R$.

\subsubsection{Symbolic powers}
By \cite[Corollary 4.11]{AM69}, the isolated primary components of an ideal in $R$ are uniquely determined. The \emph{$n$-th symbolic power} of $I$ is defined as $$I^{(n)} = \bigcap_{P\in\mathrm{Min}(R/I)}I^nR_P\cap R,$$ that is, $I^{(n)}$ is obtained by intersecting the isolated primary components of $I^n$. The family $\mathbb{I}^{\mathrm{symb}} = \{I^{(n)}\}_{n\in\mathbb{N}}$ forms a (possibly non-Noetherian) filtration of graded ideals in $R$. Note that if $\dim R/I=1$, then $I^{(n)} = I^n \colon \mathfrak{m}_R^{\infty}$. (The definition of symbolic powers is not uniform in the literature. Some authors use $\mathrm{Ass}(R/I)$ instead of $\mathrm{Min}(R/I)$ in the above definition. When $I$ has no embedded associated primes, both definitions coincide.)

\subsubsection{Generalized symbolic powers}\label{gen_symb}
Fix a graded ideal $J\subseteq R$. The \emph{``$n$-th symbolic power of $I$ with respect to $J$''} is defined to be the graded ideal $$I^n \colon J^{\infty} = \bigcup_{t\ge 1} \left(I^n \colon J^t\right) = \left\{r\in R \mid rJ^t \subseteq I^n\;\;\text{for some $t\ge 1$}\right\},$$ see \cite[Section 3]{HHT07}. Moreover, if $I^n = \bigcap_{i=1}^s Q_{i,n}$ is a graded primary decomposition of $I^n$, then $$I^n \colon J^{\infty} = I^n \colon {\sqrt{J}}^{\infty} = \bigcap_{\left\{i \mid J\not\subset \sqrt{Q_{i,n}}\right\}}Q_{i,n}.$$

If $J= \mm_R$, then $I^n \colon \mm_R^{\infty}$ is the $n$-th saturated power of $I$. Moreover, let $\mathrm{Min}(I)$ be the collection of minimal prime ideals of $I$, and $\mathrm{Ass}^*(I) = \bigcup_{n\ge 1}\mathrm{Ass}(R/I^n)$ be the collection of all associated prime ideals of the powers of $I$ (which is finite by \cite{Bro79}). Clearly, $\mathrm{Min}(I) \subseteq \mathrm{Ass}^*(I)$. If we set $J = \bigcap_{P\in \mathrm{Ass}^*(I)\setminus \mathrm{Min}(I)} P$ for $\mathrm{Ass}^*(I) \neq \mathrm{Min}(I)$ and $J = R$ otherwise, then $I^n \colon J^{\infty}$ agrees with the ordinary $n$-th symbolic power $I^{(n)}$ of $I$.

Henceforth, fix a filtration $\mathbb{I}=\{I_n\}_{n\in\mathbb{N}}$ of nonzero graded ideals in $R$.

\begin{definition}
Following \cite{DRT25}, the \emph{density function} $f_{\mathbb{I}}$ associated to the filtration $\mathbb{I}=\{I_n\}_{n\in\mathbb{N}}$, is defined as the function $$f_{\mathbb{I}} \colon \mathbb{R}_{\geq 0}\to\mathbb{R}_{\geq 0} \quad \text{given by}\quad f_{\mathbb{I}}(x) = \limsup_{n\to\infty}\dfrac{\dim_k {\left(I_n\right)}_{\lfloor xn \rfloor}}{n^{d-1}/d!}.$$ Following \cite{BCG+16}, the \emph{Waldschmidt constant} of the filtration $\mathbb{I}$ is defined as the real number $$\alpha_{\mathbb{I}} = \limsup_{n\to\infty}\dfrac{\min\left\{m\in\mathbb{N} \mid {\left(I_n\right)}_m\neq 0\right\}}{n}.$$
\end{definition}

\begin{lemma}
The invariant $\alpha_{\mathbb{I}}$ exists as a limit.
\end{lemma}
\begin{proof}
Let $\alpha(I_n) = \min\left\{m \mid {\left(I_n\right)}_m\neq 0\right\}$. Since $R$ is domain, the sequence $\{\alpha(I_n)\}_{n\in\mathbb{N}}$ is subadditive. By Fekete's subadditive lemma, $\lim_{n\to\infty}\frac{\alpha(I_n)}{n}$ exists.
\end{proof}

\subsection{Main Theorem}

\begin{theorem}\label{mainthm}
Let the assumptions be as above. Then the following statements are true.
\begin{enumerate}[$(i)$]
 \item For every $x\in \mathbb{R}_{\geq 0}\setminus\{\alpha_{\mathbb{I}}\}$, the limit $$f_{\mathbb{I}}(x) = \lim_{n\to\infty}\dfrac{\dim_k {\left(I_n\right)}_{\lfloor xn \rfloor}}{n^{d-1}/d!}$$ exists.
 \item For all $x<\alpha_{\mathbb{I}}$, one has $f_{\mathbb{I}}(x) = 0$. For all $x>\alpha_{\mathbb{I}}$, there exist bounds $$de(R)(x-\alpha_{\mathbb{I}})^{d-1} \leq f_{\mathbb{I}}(x) \leq de(R)x^{d-1},$$ where $e(R)$ denotes the Hilbert-Samuel multiplicity of $R$.
 \item The function $f_{\mathbb{I}}\colon (\alpha_{\mathbb{I}},\infty) \to \mathbb{R}_{\geq 0}$ is non-decreasing, and log-concave i.e., $${\left(f_{\mathbb{I}}(x+y)\right)}^{1/(d-1)} \geq {\left(f_{\mathbb{I}}(x)\right)}^{1/(d-1)} + {\left(f_{\mathbb{I}}(y)\right)}^{1/(d-1)} \quad \forall x,y>\alpha_{\mathbb{I}}.$$
 \item In particular, $f_{\mathbb{I}}$ is continuous on $\mathbb{R}_{\geq 0}\setminus \{\alpha_{\mathbb{I}}\}$. Moreover, there exists a countable subset $\Omega\subset \mathbb{R}$ such that $f_{\mathbb{I}}$ is differentiable on $\mathbb{R}_{\geq 0}\setminus \Omega$.
 \item For every real number $y>0$, one has $$\int_0^{y} f_{\mathbb{I}}(x)dx = \lim_{n\to\infty}\dfrac{\sum_{m=0}^{\lfloor yn\rfloor}\dim_k {\left(I_n\right)}_m}{n^d/d!}.$$
\end{enumerate}
\end{theorem}

The following lemma proves part $(ii)$ of Theorem \ref{mainthm}.

\begin{lemma}\label{easy}
For all $x<\alpha_{\mathbb{I}}$, $f_{\mathbb{I}}(x)=0$. For all $x>\alpha_{\mathbb{I}}$, $\frac{e(R)}{(d-1)!}(x-\alpha_{\mathbb{I}})^{d-1} \leq f_{\mathbb{I}}(x) \leq \frac{e(R)}{(d-1)!}x^{d-1}$.
\end{lemma}
\begin{proof}
Suppose first that $x<\alpha_{\mathbb{I}}$. By definition, $\lim\limits_{n\to\infty}\frac{\alpha(I_n)}{n} \;=\;\alpha_{\mathbb{I}}$, where $\alpha(I_n) = \min\left\{m \mid {\left(I_n\right)}_m\neq 0\right\}$. Hence, for $n\gg 0$, we have $\frac{\lfloor xn \rfloor}{n} < \frac{\alpha(I_n)}{n}$, which means $\lfloor xn \rfloor < \alpha(I_n)$. Therefore, $(I_n)_{\lfloor xn \rfloor}=0$ for all $n\gg 0$, so $f_{\mathbb{I}}(x)=0$.

Now assume that $x>\alpha_{\mathbb{I}}$. Then for all $n\in\mathbb{N}$, we can write $$\lfloor xn \rfloor = \alpha(I_n) + \theta(n),$$ where $\theta \colon \mathbb{N} \to \mathbb{Z}$ is a function such that $\lim_{n\to\infty}\frac{\theta(n)}{n} = x-\alpha_{\mathbb{I}}>0$. Choose a nonzero element $f_n \in {(I_n)}_{\alpha(I_n)}$. For all $n\gg 0$, we have $$(f_n)\cdot R_{\lfloor xn \rfloor - \alpha(I_n)} \subseteq {(I_n)}_{\lfloor xn \rfloor} \subseteq {R}_{\lfloor xn \rfloor}.$$ Since $R$ is a standard graded finitely generated $k$-algebra of dimension $d$, for all $m\gg0$ we have \[
\dim_k R_m \;=\; \frac{e(R)}{(d-1)!} m^{d-1} + \text{lower degree terms in $m$}.
\] This gives $$de(R)(x-\alpha_{\mathbb{I}})^{d-1} = \lim_{n\to\infty}\dfrac{\dim_k {R}_{\lfloor xn \rfloor - \alpha(I_n)}}{n^{d-1}/d!} \leq \limsup_{n\to\infty}\dfrac{\dim_k {\left(I_n\right)}_{\lfloor xn \rfloor}}{n^{d-1}/d!} \leq \lim_{n\to\infty}\dfrac{\dim_k {R}_{\lfloor xn \rfloor}}{n^{d-1}/d!} = de(R)x^{d-1}.$$
\end{proof}

\subsection{Construction of a monomial valuation}

The following construction, due to Cutkosky \cite{Cut15}, will play a key role. We reproduce the arguments here for completeness.

\begin{lemma}[Cutkosky]\label{lem1}
There exists a regular local ring $(S,\mm_S)$ such that
\begin{enumerate}[$(i)$]
 \item $R\subseteq S$ and $\mm_S\cap R = \mm_R$,
 \item $S$ is essentially of finite of type over $R$ (hence over $k$),
 \item $\dim R = \dim S$,
 \item the induced map on the residue fields $R/\mm_R \hookrightarrow S/\mm_S$ is an isomorphism,
 \item $R$ and $S$ have the same quotient field.
\end{enumerate}
\end{lemma}
\begin{proof}
Let $$\varphi\colon X \longrightarrow \mathrm{Spec}(R)$$ be the normalization of the blow up of the closed point $\{\mm_R\}$ in $\mathrm{Spec}(R)$. The morphism $\varphi$ is of finite type since $R$ is excellent. The exceptional fibre $\varphi^{-1}\left(\{\mm_R\}\right)$ is a closed subscheme of codimension one in $X$. Since $X$ is normal, its singular locus has codimension $\geq 2$. So there exists a closed point $\eta \in \varphi^{-1}\left(\{\mm_R\}\right)$ such that the local ring $\mathcal{O}_{X,\eta}$ is regular and set $S=\mathcal{O}_{X,\eta}$. One can now verify that $(S,\mm_S)$ satisfies all the conclusions of the lemma.
\end{proof}

Henceforth, fix a regular local ring $(S,\mm_S)$ as in Lemma \ref{lem1}. By Cohen's structure theorem, the $\mm_S$-adic completion $\widehat{S} \cong k[[x_1,\ldots,x_d]]$, a formal power series ring in $d$ indeterminates $x_1,\ldots,x_d$ over the residue field $k$. Fix $\mathbb{Q}$-linearly independent real numbers $\xi_1,\ldots,\xi_d$ such that $\xi_j\geq 1$ for each $j$. Any nonzero element $f\in S \subseteq \widehat{S}$ can be written as $$f = \sum_{(i_1,\ldots,i_d)\in\mathbb{N}^d}a_{i_1,\ldots,i_d}x_1^{i_1}\cdots x_d^{i_d},$$ where $a_{i_1,\ldots,i_d}\in k$ and not all zero. Define $$\nu(f) = \min\Big\{\sum_{j=1}^d i_j\xi_j \mid a_{i_1,\ldots,i_d}\neq 0\Big\}.$$

\begin{lemma}\label{lem_valuation}
$\nu$ defines a (rank one) monomial valuation on $\mathrm{QF}(R)$ dominating $(S,\mm_S)$ with value group $\sum_{i=1}^d \mathbb{Z}\xi_i$ and residue field $k$.
\end{lemma}
\begin{proof}
The fact that $\nu$ is a valuation is true more generally, see \cite[Chapter VI, Section 15]{ZS60}. The $\mathbb{Q}$-linear independence of the $\xi_j$'s ensures that different exponent vectors $(i_1,\dots,i_d)$ give different $\nu$-values, so the valuation is monomial of rank one.
\end{proof}

In particular, for each $0\neq f\in S$ there exists a unique exponent vector $\mathbf{v}_f = (i_1,\ldots,i_d)\in \mathbb{N}^d$ such that $\nu(f) = \langle \mathbf{v}_f, \boldsymbol{\xi}\rangle$, where $\boldsymbol{\xi}=(\xi_1,\ldots,\xi_d)\in\mathbb{R}^d$ and $\langle\_ ,\_\rangle$ represents the usual inner product on $\mathbb{R}^d$.

\begin{lemma}\label{lem_count}
For all $(m,n)\in\mathbb{N}^2$, we have
\begin{align*}
 \dim_k {\left(I_n\right)}_m = \# \Big\{(i_1,\ldots,i_d) \in \mathbb{N}^{d} \mid \exists f\in {\left(I_n\right)}_m\;\;\text{with}\;\;\nu(f)=\sum_{j=1}^d i_j\xi_j\Big\}.
\end{align*}
\end{lemma}
\begin{proof}
Fix $(m,n)\in\mathbb{N}^2$. Let $\Omega = \big\{(i_1,\ldots,i_d) \in \mathbb{N}^{d} \mid \exists f\in {\left(I_n\right)}_m\;\text{with}\;\nu(f)=\sum_{j=1}^d i_j\xi_j\big\}$. Let $r=\#\Omega$. Choose elements $f_1,\ldots,f_r \in {\left(I_n\right)}_m$ such that $\nu(f_t) = \sum_{j=1}^d i_{j,t}\xi_j$ and $\Omega = \left\{(i_{1,t},\ldots,i_{d,t})\right\}_{t=1}^r$. Reorder so that $\nu(f_1)<\cdots<\nu(f_r)$. We claim that $f_1,\ldots,f_r$ forms a $k$-basis for ${\left(I_n\right)}_m$.

If $f_1,\ldots,f_r$ are linearly dependent over $k$ then there exists a relation $\sum_{t=1}^r a_tf_t = 0$ where $a_t\in k$ and not all zero. Then we must have $\nu(f_t)=\nu(f_{t^{\prime}})$ for some $t\neq t^{\prime}$ but this contradicts that $\nu(f_t)$'s are all distinct.

Let $f\in {\left(I_n\right)}_m$ and $f\neq 0$. Then $\nu(f)=\nu(f_{t_1})$ for some $1\leq t_1\leq r$. If $f\in k\cdot  f_{t_1}$ then we are done. Otherwise, there exists $a_{t_1}\in k\setminus \{0\}$ such that $\nu(f-a_{t_1}f_{t_1})>\nu(f_{t_1})$ and $f-a_{t_1}f_{t_1} \in {\left(I_n\right)}_m$. So write $\nu(f-a_{t_1}f_{t_1}) = \nu(f_{t_2})$ for some $t_1<t_2\leq r$. If $f \in k\cdot f_{t_1} + k\cdot f_{t_2}$ then we are done. Otherwise, there exists $a_{t_2}\in k\setminus \{0\}$ such that $\nu(f-a_{t_1}f_{t_1}-a_{t_2}f_{t_2})>\nu(f-a_{t_1}f_{t_1})$ and $f-a_{t_1}f_{t_1}-a_{t_2}f_{t_2} \in {\left(I_n\right)}_m$. This process can only continue for finitely many steps. Therefore, we conclude that there is an integer $1\leq s\leq r$ and elements $a_{i_1},\ldots,a_{i_s}\in k\setminus\{0\}$ with $1\leq i_1<\cdots<i_s\leq r$ such that $f = \sum_{j=1}^s a_{i_j}f_{i_j}$.
\end{proof}

\subsection{A semigroup construction}

Let $p,q>0$ be integers such that $p/q>\alpha_{\mathbb{I}}$. Now define
\begin{equation}\label{eq_semigroup}
 \Gamma^{\mathbb{I}}_{(p,q)} = \Big\{(i_1,\ldots,i_d,qn) \in \mathbb{N}^{d+1} \;\big\vert \; \exists f\in {\left(I_{qn}\right)}_{pn}\;\text{with}\;\nu(f)=\sum_{j=1}^d i_j\xi_j\Big\}.
\end{equation}

Recall the notations from section \ref{cones}.

\begin{lemma}\label{lem_properties}
 The following statements are true.
 \begin{enumerate}[$(i)$]
  \item $\Gamma^{\mathbb{I}}_{(p,q)} \subseteq \mathbb{N}^{d+1}$ is an additive semigroup.
  \item $\Gamma^{\mathbb{I}}_{(p,q)} \cap \left(\mathbb{Z}^d \times \{0\}\right) = \{(0,\ldots,0)\}$.
  \item $\Gamma^{\mathbb{I}}_{(p,q)}$ is strongly nonnegative with $q\big(\Gamma^{\mathbb{I}}_{(p,q)}\big) = d-1$.
  \item $m\big(\Gamma^{\mathbb{I}}_{(p,q)}\big)=q$.
  \item The limit $\lim\limits_{n\to\infty}\dfrac{\dim_k {\left(I_{qn}\right)}_{pn}}{n^{d-1}}$ exists and equals $\dfrac{\mathrm{vol}_{d-1}\big(\Delta\big(\Gamma^{\mathbb{I}}_{(p,q)}\big)\big)}{\mathrm{ind}\big(\Gamma^{\mathbb{I}}_{(p,q)}\big)}$.
 \end{enumerate}
\end{lemma}
\begin{proof}
\emph{$(i)$}. Since $\{I_n\}_{n\in\mathbb{N}}$ is a filtration and $\nu$ is a valuation, we conclude that if $f\in {\left(I_{qn}\right)}_{pn}$ and $f^{\prime}\in {\left(I_{qn^{\prime}}\right)}_{pn^{\prime}}$ then $ff^{\prime} \in {\left(I_{q(n+n^{\prime})}\right)}_{p(n+n^{\prime})}$ and $\nu(ff^{\prime}) = \nu(f) + \nu(f^{\prime})$. Thus $\Gamma^{\mathbb{I}}_{(p,q)}$ is closed under addition.

\vspace{0.2cm}
\noindent
\emph{$(ii)$}. We have ${\left(I_{0}\right)}_{0}= R_0 = k$, and $\left. \nu\right|_{k\setminus \{0\}} = 0$. Hence, $\Gamma^{\mathbb{I}}_{(p,q)} \cap \left(\mathbb{Z}^d \times \{0\}\right) = \{(0,\ldots,0)\}$.

\vspace{0.2cm}
\noindent
\emph{$(iii)$}. By Lemma \ref{lem_count}, $$\limsup_{n\to\infty}\dfrac{\#\big(\Gamma^{\mathbb{I}}_{(p,q)} \cap \big(\mathbb{N}^d\times \{qn\}\big)\big)}{n^{d-1}} = \limsup_{n\to\infty}\dfrac{\dim_k {\left(I_{qn}\right)}_{pn}}{n^{d-1}} \leq \lim_{n\to\infty}\dfrac{\dim_k {R}_{pn}}{n^{d-1}} = \frac{e(R)}{(d-1)!}p^{d-1}.$$ Using Theorem \ref{KK2}, we conclude that $\Gamma^{\mathbb{I}}_{(p,q)}$ is strongly nonnegative with $q\big(\Gamma^{\mathbb{I}}_{(p,q)}\big)\leq d-1$.

\vspace{0.2cm}
\noindent
\emph{$(iv)$}. Since $p/q>\alpha_{\mathbb{I}}$, we have ${\left(I_{qn}\right)}_{pn}\neq 0$ for all $n\gg 0$, which implies $qn \in \pi\big(\Gamma^{\mathbb{I}}_{(p,q)}\big)$ for all $n\gg 0$. Thus, $q\mathbb{Z}\subseteq \pi(G(\Gamma^{\mathbb{I}}_{(p,q)}))$, and the reverse inclusion is obvious. Therefore, $m\big(\Gamma^{\mathbb{I}}_{(p,q)}\big)=q$.

\vspace{0.2cm}
\noindent
\emph{$(v)$}. Applying Theorem \ref{KK1} to $\Gamma^{\mathbb{I}}_{(p,q)}$, together with parts $(iii), (iv)$ and Lemma \ref{lem_count}, yields $$\lim_{n\to\infty}\dfrac{\dim_k {\left(I_{qn}\right)}_{pn}}{n^{d-1}} = \lim_{n\to\infty}\dfrac{\#\big(\Gamma^{\mathbb{I}}_{(p,q)} \cap \big(\mathbb{N}^d\times \{qn\}\big)\big)}{n^{d-1}} = \dfrac{\mathrm{vol}_{d-1}\big(\Delta\big(\Gamma^{\mathbb{I}}_{(p,q)}\big)\big)}{\mathrm{ind}\big(\Gamma^{\mathbb{I}}_{(p,q)}\big)}.$$
\end{proof}

Note that part $(v)$ of Lemma \ref{lem_properties} can also be seen as a special case of \cite[Theorem 4]{Das21}.

\subsection{Existence of limit for rational values}

\begin{lemma}\label{lem_converge}
Let $p,q>0$ be integers such that $\tfrac{p}{q}>\alpha_{\mathbb{I}}$. Then there exists an integer $n_0>0$ such that ${\left(I_{qn}\right)}_{pn-1}\neq 0$ and ${\left(I_{qn+1}\right)}_{pn}\neq 0$ for all integers $n\geq n_0$.
\end{lemma}
\begin{proof}
We only prove the first claim as the second is analogous. Choose a real number $\varepsilon$ such that $$0 < \varepsilon < \frac{p}{q}-\alpha_{\mathbb{I}}.$$ Let $\alpha(I_n) = \min\left\{m \mid {\left(I_n\right)}_m\neq 0\right\}$, and by definition $\alpha_{\mathbb{I}} = \lim_{n\to\infty} \frac{\alpha(I_n)}{n}$. There exists an integer $N_{\varepsilon}>0$ such that for all integers $n\geq N_{\varepsilon}$, $$\left\vert \alpha_{\mathbb{I}} - \frac{\alpha(I_{qn})}{qn}\right\vert<\varepsilon.$$ Moreover, for all integers $n>\frac{1}{p-q(\alpha_{\mathbb{I}}+\varepsilon)}$, we have $$\frac{pn-1}{qn}>\alpha_{\mathbb{I}}+\varepsilon.$$ Thus, for all integers $n\geq \max\left\{N_{\varepsilon},\frac{1}{p-q(\alpha_{\mathbb{I}}+\varepsilon)}\right\}$, one has $$\frac{pn-1}{qn} > \frac{\alpha(I_{qn})}{qn},$$ i.e., $pn-1>\alpha(I_{qn})$. Choose nonzero elements $f\in{\left(I_{qn}\right)}_{\alpha(I_{qn})}$ and $z\in R_1$. Since $R$ is a domain, the element $z^{pn-1-\alpha(I_{qn})}\cdot f \in {\left(I_{qn}\right)}_{pn-1}$ is nonzero for all $n\gg 0$.
\end{proof}

\begin{proposition}\label{prop_rational}
Suppose that $x\in \mathbb{Q}$ and $x>\alpha_{\mathbb{I}}$. Then the limit $$f_{\mathbb{I}}(x) = \lim_{n\to\infty}\dfrac{\dim_k {\left(I_n\right)}_{\lfloor xn \rfloor}}{n^{d-1}/d!}$$ exists.
\end{proposition}
\begin{proof}
Write $x=\frac{p}{q}$. By Lemma \ref{lem_converge}, we may replace $(p,q)$ by $(pn_0,qn_0)$ for some $n_0>0$ so that ${\left(I_{q}\right)}_{p-1}\neq 0$ and ${\left(I_{q+1}\right)}_{p}\neq 0$. By Lemma \ref{lem_properties}\;$(v)$, the limit $$L:=\lim_{n\to\infty}\dfrac{\dim_k {\left(I_{qn}\right)}_{pn}}{(qn)^{d-1}}$$ exists. It remains to show that for each $a\in\{0,\ldots,q-1\}$, and $b\in\{0,\ldots,p-1\}$, the sequence $$\left\{\dfrac{\dim_k {\left(I_{qn+a}\right)}_{pn+b}}{(qn+a)^{d-1}}\right\}_{n\in\mathbb{N}}$$ converges to the same value $L$. Choose three nonzero homogeneous elements $z\in R_1$, $f\in {\left(I_{q+1}\right)}_p$, and $g\in {\left(I_{q}\right)}_{p-1}$. Then there are inclusions of $k$-vector spaces $$z^b f^a\cdot {\left(I_{qn}\right)}_{pn} \subseteq {\left(I_{q(n+a)+a}\right)}_{p(n+a)+b}\subseteq {\left(I_{q(n+a)}\right)}_{p(n+a)+b},$$ and
$$g^{p-b}\cdot{\left(I_{q(n+a)}\right)}_{p(n+a)+b}\subseteq {\left(I_{q(n+p+a-b)}\right)}_{p(n+p+a-b)}$$ for all integers $n\geq 0$. Hence $$\dim_k {\left(I_{qn}\right)}_{pn} \leq \dim_k {\left(I_{q(n+a)+a}\right)}_{p(n+a)+b} \leq \dim_k {\left(I_{q(n+p+a-b)}\right)}_{p(n+p+a-b)}$$ for all $n\geq 0$. Dividing by $(qn)^{d-1}$ and adjusting the denominators gives
\begin{multline*}
 \dfrac{\dim_k {\left(I_{qn}\right)}_{pn}}{(qn)^{d-1}} \leq \dfrac{\dim_k {\left(I_{q(n+a)+a}\right)}_{p(n+a)+b}}{(q(n+a)+a)^{d-1}} \cdot \left(1+\frac{a(q+1)}{qn}\right)^{d-1}\\ \leq \dfrac{\dim_k {\left(I_{q(n+p+a-b)}\right)}_{p(n+p+a-b)}}{(q(n+p+a-b))^{d-1}}\cdot \left(1+\frac{p+a-b}{n}\right)^{d-1}
\end{multline*}
for all $n\geq 0$. Both the outer terms converge to $L$, and the sandwich theorem shows that $$\lim_{n\to\infty}\dfrac{\dim_k {\left(I_{qn+a}\right)}_{pn+b}}{(qn+a)^{d-1}}=L.$$
\end{proof}

\subsection{Invariance of the index}

Recall the notations from section \ref{cones}.

\begin{proposition}\label{welldefined}
Let $p,q>0$ be integers such that $\tfrac{p}{q}>\alpha_{\mathbb{I}}$. Then
\begin{align*}
G\big(\Gamma^{\mathbb{I}}_{(p,q)}\big) \cap \left(\mathbb{Z}^d\times \{0\}\right) &= G\big(\Gamma^{R_{\bullet}}_{(1,1)}\big) \cap \left(\mathbb{Z}^d\times \{0\}\right),\;\;\text{and}\\
L\big(\Gamma^{\mathbb{I}}_{(p,q)}\big) \cap \left(\mathbb{Z}^d\times \{0\}\right) &= L\big(\Gamma^{R_{\bullet}}_{(1,1)}\big) \cap \left(\mathbb{Z}^d\times \{0\}\right).
\end{align*}
In particular, $\mathrm{ind}\big(\Gamma^{\mathbb{I}}_{(p,q)}\big)= \mathrm{ind}\big(\Gamma^{{R}_{\bullet}}_{(1,1)}\big)$, where ${R}_{\bullet}$ denotes the constant filtration $\{R\}_{n\in\mathbb{N}}$. Moreover, for any integer $n_0>0$, we have $\frac{1}{q}\cdot\Delta\big(\Gamma^{\mathbb{I}}_{(p,q)}\big) = \frac{1}{qn_0}\cdot\Delta\big(\Gamma^{\mathbb{I}}_{(pn_0,qn_0)}\big)$.
\end{proposition}
\begin{proof} We shall breakdown our proof into multiple claims.

\vspace{0.2cm}
\noindent
\emph{Claim $1$}. For any integer $n_0>0$, we have
\begin{equation*}
L\big(\Gamma^{\mathbb{I}}_{(p,q)}\big) = L\big(\Gamma^{\mathbb{I}}_{(pn_0,qn_0)}\big),\;\; \mathrm{Con}\big(\Gamma^{\mathbb{I}}_{(p,q)}\big) = \mathrm{Con}\big(\Gamma^{\mathbb{I}}_{(pn_0,qn_0)}\big), \;\; \text{and} \;\; \frac{1}{q}\cdot\Delta\big(\Gamma^{\mathbb{I}}_{(p,q)}\big) = \frac{1}{qn_0}\cdot\Delta\big(\Gamma^{\mathbb{I}}_{(pn_0,qn_0)}\big).
\end{equation*}

For each $n_0>0$, there are inclusions $$n_0 \cdot \Gamma^{\mathbb{I}}_{(p,q)} \subseteq \Gamma^{\mathbb{I}}_{(pn_0,qn_0)} \subseteq \Gamma^{\mathbb{I}}_{(p,q)}.$$ These imply that $L\big(\Gamma^{\mathbb{I}}_{(p,q)}\big) = L\big(\Gamma^{\mathbb{I}}_{(pn_0,qn_0)}\big)$ and $\mathrm{Con}\big(\Gamma^{\mathbb{I}}_{(p,q)}\big) = \mathrm{Con}\big(\Gamma^{\mathbb{I}}_{(pn_0,qn_0)}\big)$. Intersecting with the hyperplane $\mathbb{R}^d\times \{1\}$ yields
\begin{multline*}
 \frac{1}{q}\cdot\Delta\big(\Gamma^{\mathbb{I}}_{(p,q)}\big) = \frac{1}{q}\cdot\left(\mathrm{Con}\big(\Gamma^{\mathbb{I}}_{(p,q)}\big) \cap \left(\mathbb{R}^d\times \{q\}\right)\right) = \mathrm{Con}\big(\Gamma^{\mathbb{I}}_{(p,q)}\big) \cap \left(\mathbb{R}^d\times \{1\}\right)\\ =\mathrm{Con}\big(\Gamma^{\mathbb{I}}_{(pn_0,qn_0)}\big) \cap \left(\mathbb{R}^d\times \{1\}\right) = \frac{1}{qn_0}\cdot\left(\mathrm{Con}\big(\Gamma^{\mathbb{I}}_{(pn_0,qn_0)}\big) \cap \left(\mathbb{R}^d\times \{qn_0\}\right)\right) = \frac{1}{qn_0}\cdot\Delta\big(\Gamma^{\mathbb{I}}_{(pn_0,qn_0)}\big).
\end{multline*}

\vspace{0.2cm}
\noindent
\emph{Claim $2$}. For every integer $n_0>0$, we have
\begin{equation*}
G\big(\Gamma^{\mathbb{I}}_{(p,q)}\big) \cap \left(\mathbb{Z}^d\times \{0\}\right) = G\big(\Gamma^{\mathbb{I}}_{(pn_0,qn_0)}\big) \cap \left(\mathbb{Z}^d\times \{0\}\right).
\end{equation*}

Clearly, $G\big(\Gamma^{\mathbb{I}}_{(pn_0,qn_0)}\big) \cap \left(\mathbb{Z}^d\times \{0\}\right) \subseteq G\big(\Gamma^{\mathbb{I}}_{(p,q)}\big) \cap \left(\mathbb{Z}^d\times \{0\}\right)$.
Let $(\mathbf{u},0) \in G\big(\Gamma^{\mathbb{I}}_{(p,q)}\big) \cap \left(\mathbb{Z}^d\times \{0\}\right)$. Then there exist $\mathbf{a}, \mathbf{b}\in\mathbb{N}^d$, and $s\in\mathbb{N}$ such that $(\mathbf{u},0) = (\mathbf{a}, qs) - (\mathbf{b}, qs)$ with $(\mathbf{a}, qs), (\mathbf{b}, qs) \in \Gamma^{\mathbb{I}}_{(p,q)}$. As ${(I_q)}_p\neq 0$, there exist $\mathbf{c}\in\mathbb{N}^d$ with $(\mathbf{c},q)\in \Gamma^{\mathbb{I}}_{(p,q)}$. Write $s = n_0k + t$, where $0\leq t\leq n_0-1$. So,
\begin{align*}
 (\mathbf{u},0) &= (\mathbf{a}, qs) - (\mathbf{b}, qs)\\
 &= (\mathbf{a}, qs) + (n_0-t)\cdot (\mathbf{c},q) - (\mathbf{b}, qs) - (n_0-t)\cdot (\mathbf{c},q)\\
 &= (\mathbf{a} + (n_0-t)\mathbf{c}, (k+1)qn_0) - (\mathbf{b} + (n_0-t)\mathbf{c}, (k+1)qn_0),
\end{align*}
where $(\mathbf{a} + (n_0-t)\mathbf{c}, (k+1)qn_0), (\mathbf{b} + (n_0-t)\mathbf{c}, (k+1)qn_0) \in \Gamma^{\mathbb{I}}_{(pn_0,qn_0)}$. In other words, $(\mathbf{u},0) \in G\big(\Gamma^{\mathbb{I}}_{(pn_0,qn_0)}\big) \cap \left(\mathbb{Z}^d\times \{0\}\right)$. This completes the proof of Claim $2$.

\vspace{0.2cm}
\noindent
\emph{Claim $3$}. Let $p^{\prime},q^{\prime}>0$ be integers such that $\frac{p^{\prime}}{q^{\prime}}>\frac{p}{q}$. Then
\begin{align}
G\big(\Gamma^{\mathbb{I}}_{(p,q)}\big) \cap \left(\mathbb{Z}^d\times \{0\}\right) &= G\big(\Gamma^{\mathbb{I}}_{(p^{\prime},q^{\prime})}\big) \cap \left(\mathbb{Z}^d\times \{0\}\right),\;\;\text{and}\label{equal1}\\
L\big(\Gamma^{\mathbb{I}}_{(p,q)}\big) \cap \left(\mathbb{Z}^d\times \{0\}\right) &= L\big(\Gamma^{\mathbb{I}}_{(p^{\prime},q^{\prime})}\big) \cap \left(\mathbb{Z}^d\times \{0\}\right).\label{equal2}
\end{align}

Choose $n_0, n_0^{\prime}>0$ with $qn_0 = q^{\prime}n_0^{\prime}$. By Claim $1$ and Claim $2$, we may replace $(p,q)$ and $(p^{\prime},q^{\prime})$ by $(pn_0,qn_0)$ and $(p^{\prime}n^{\prime}_0,q^{\prime}n^{\prime}_0)$ respectively, and by Lemma \ref{lem_converge}, we may further assume that ${\left(I_{q}\right)}_{p-1} \neq 0$. Thus, it suffices to assume from the beginning that $q=q^{\prime}$, and ${\left(I_{q}\right)}_{p-1} \neq 0$.

Let $z\in R_1$ and $f\in {\left(I_{q}\right)}_{p-1}$ be two nonzero elements. There exist unique vectors $\mathbf{v}_z, \mathbf{v}_f \in \mathbb{N}^d$ such that $\nu(z) = \langle \mathbf{v}_z, \boldsymbol{\xi}\rangle$, and $\nu(f) = \langle \mathbf{v}_f, \boldsymbol{\xi}\rangle$, where $\boldsymbol{\xi}=(\xi_1,\ldots,\xi_d)\in \mathbb{R}^d$. Put $c=p^{\prime}-p$. For all $n,n^{\prime}>0$, there are inclusions $$z^{cn}\cdot {\left(I_{qn}\right)}_{pn} \subseteq {\left(I_{qn}\right)}_{p^{\prime}n}, \quad \text{and} \quad f^{cn^{\prime}}\cdot {\left(I_{qn^{\prime}}\right)}_{p^{\prime}n^{\prime}} \subseteq {\left(I_{(1+c)qn^{\prime}}\right)}_{(1+c)pn^{\prime}}.$$ In other words, if $(\mathbf{u},qn)\in \Gamma^{\mathbb{I}}_{(p,q)}$ then $(\mathbf{u}+cn\mathbf{v}_z,qn)\in \Gamma^{\mathbb{I}}_{(p^{\prime},q)}$. Similarly, if $(\mathbf{u}^{\prime},qn^{\prime})\in \Gamma^{\mathbb{I}}_{(p^{\prime},q)}$ then $(\mathbf{u}^{\prime}+cn^{\prime}\mathbf{v}_f,(1+c)qn^{\prime})\in \Gamma^{\mathbb{I}}_{\left((1+c)p,(1+c)q\right)}$.

To prove \eqref{equal1}, let $\left(\mathbf{u},0\right) \in G\big(\Gamma^{\mathbb{I}}_{(p,q)}\big) \cap \left(\mathbb{Z}^d\times \{0\}\right)$. Then there exist $\mathbf{a},\mathbf{b}\in \mathbb{N}^d$, and $n\in\mathbb{N}$ such that $\left(\mathbf{u},0\right) = (\mathbf{a},qn) - (\mathbf{b},qn)$, with $(\mathbf{a},qn), (\mathbf{b},qn) \in \Gamma^{\mathbb{I}}_{(p,q)}$. Since we can write $(\mathbf{a},qn) - (\mathbf{b},qn) = (\mathbf{a}+cn\mathbf{v}_z,qn) - (\mathbf{b}+cn\mathbf{v}_z,qn)$, it follows that $\left(\mathbf{u},0\right) \in G\big(\Gamma^{\mathbb{I}}_{(p^{\prime},q)}\big) \cap \left(\mathbb{Z}^d\times \{0\}\right)$. Again, let $\left(\mathbf{u}^{\prime},0\right) \in G\big(\Gamma^{\mathbb{I}}_{(p^{\prime},q)}\big) \cap \left(\mathbb{Z}^d\times \{0\}\right)$. Then we can write $\left(\mathbf{u}^{\prime},0\right) = (\mathbf{a}^{\prime},qn^{\prime}) - (\mathbf{b}^{\prime},qn^{\prime})$, for some $(\mathbf{a}^{\prime},qn^{\prime}), (\mathbf{b}^{\prime},qn^{\prime}) \in \Gamma^{\mathbb{I}}_{(p^{\prime},q)}$. Since $(\mathbf{a}^{\prime},qn^{\prime}) - (\mathbf{b}^{\prime},qn^{\prime}) = (\mathbf{a}^{\prime}+cn^{\prime}\mathbf{v}_f,qn^{\prime}) - (\mathbf{b}^{\prime}+cn^{\prime}\mathbf{v}_f,qn^{\prime})$, it follows that $\left(\mathbf{u}^{\prime},0\right) \in G\big(\Gamma^{\mathbb{I}}_{((1+c)p,(1+c)q)}\big)\cap \left(\mathbb{Z}^d\times \{0\}\right)$. So, $$G\big(\Gamma^{\mathbb{I}}_{(p,q)}\big) \cap \left(\mathbb{Z}^d\times \{0\}\right) \subseteq G\big(\Gamma^{\mathbb{I}}_{(p^{\prime},q)}\big) \cap \left(\mathbb{Z}^d\times \{0\}\right) \subseteq G\big(\Gamma^{\mathbb{I}}_{((1+c)p,(1+c)q)}\big) \cap \left(\mathbb{Z}^d\times \{0\}\right).$$ Clearly, $\Gamma^{\mathbb{I}}_{((1+c)p,(1+c)q)} \subseteq \Gamma^{\mathbb{I}}_{(p,q)}$ and this proves \eqref{equal1}. One also proves \eqref{equal2} by similar methods.

\vspace{0.2cm}
\noindent
\emph{Claim $4$}. We have
\begin{align*}
G\big(\Gamma^{\mathbb{I}}_{(p,q)}\big) \cap \left(\mathbb{Z}^d\times \{0\}\right) &= G\big(\Gamma^{R_{\bullet}}_{(1,1)}\big) \cap \left(\mathbb{Z}^d\times \{0\}\right),\;\;\text{and}\\
L\big(\Gamma^{\mathbb{I}}_{(p,q)}\big) \cap \left(\mathbb{Z}^d\times \{0\}\right) &= L\big(\Gamma^{R_{\bullet}}_{(1,1)}\big) \cap \left(\mathbb{Z}^d\times \{0\}\right).
\end{align*}

By Claim $3$, we may reduce to $q=1$ and $(I_1)_{p-1}\neq 0$. Pick a nonzero element $g \in {\left(I_{1}\right)}_{p-1}$. There is a unique vector $\mathbf{v}_g \in \mathbb{N}^d$ such that $\nu(g) = \langle \mathbf{v}_g, \boldsymbol{\xi}\rangle$. Let $g_{\bullet}$ denote the filtration $\left\{(g^n)\right\}_{n\in\mathbb{N}}$. Then
\begin{align*}
\Gamma^{g_{\bullet}}_{(p,1)} &= \left\{(\mathbf{v}_f,n)\in\mathbb{N}^{d+1} \mid \exists f \in {(g^{n})}_{pn}\;\text{with}\;\nu(f) = \langle \mathbf{v}_f, \boldsymbol{\xi} \rangle\right\}\\
  &= \left\{(n\mathbf{v}_g + \mathbf{v}_{f^{\prime}},n)\in\mathbb{N}^{d+1} \mid \exists f^{\prime} \in {R}_{n}\;\text{with}\;\nu(f^{\prime}) = \langle \mathbf{v}_{f^{\prime}}, \boldsymbol{\xi} \rangle\right\},
\end{align*}
which shows that
\begin{align*}
 G\big(\Gamma^{g_{\bullet}}_{(p,1)}\big)\cap \left(\mathbb{Z}^d\times \{0\}\right) &= G\big(\Gamma^{R_{\bullet}}_{(1,1)}\big)\cap \left(\mathbb{Z}^d\times \{0\}\right),\;\text{and}\\
  L\big(\Gamma^{g_{\bullet}}_{(p,1)}\big)\cap \left(\mathbb{Z}^d\times \{0\}\right) &= L\big(\Gamma^{R_{\bullet}}_{(1,1)}\big)\cap \left(\mathbb{Z}^d\times \{0\}\right).
 \end{align*}
For each $n$, we have $(g^n)_{pn} \subseteq (I_n)_{pn} \subseteq R_{pn}$. Now using the above equalities, one gets
\begin{align*}
 G\big(\Gamma^{R_{\bullet}}_{(1,1)}\big)\cap \left(\mathbb{Z}^d\times \{0\}\right) &\subseteq G\big(\Gamma^{\mathbb{I}}_{(p,1)}\big)\cap \left(\mathbb{Z}^d\times \{0\}\right) \subseteq G\big(\Gamma^{R_{\bullet}}_{(p,1)}\big)\cap \left(\mathbb{Z}^d\times \{0\}\right),\;\text{and}\\
 L\big(\Gamma^{R_{\bullet}}_{(1,1)}\big)\cap \left(\mathbb{Z}^d\times \{0\}\right) &\subseteq L\big(\Gamma^{\mathbb{I}}_{(p,1)}\big)\cap \left(\mathbb{Z}^d\times \{0\}\right) \subseteq L\big(\Gamma^{R_{\bullet}}_{(p,1)}\big)\cap \left(\mathbb{Z}^d\times \{0\}\right).
\end{align*}
The outer terms are equal from Claim $3$, hence all coincide. Thus Claim $4$ is established.

The equality $\mathrm{ind}\big(\Gamma^{\mathbb{I}}_{(p,q)}\big) = \mathrm{ind}\big(\Gamma^{R_{\bullet}}_{(1,1)}\big)$ is a direct consequence of Claim $4$.
\end{proof}

\begin{remark}\label{suffices}
Proposition \ref{welldefined} shows that $\mathrm{ind}\big(\Gamma^{\mathbb{I}}_{(p,q)}\big)$ is independent of the filtration $\mathbb{I}$ and integers $p,q$ such that $\tfrac{p}{q}>\alpha_{\mathbb{I}}$. We denote this common value simply by $\mathrm{ind}(R)$. For such $x=\frac{p}{q}$, Proposition \ref{welldefined} also shows that $$\Delta_x(\mathbb{I}) := \frac{1}{q}\cdot\Delta\big(\Gamma^{\mathbb{I}}_{(p,q)}\big) = \overline{\bigcup_{n\in\mathbb{Z}_{\geq 1}}\left\{\left(\tfrac{i_1}{qn},\ldots,\tfrac{i_d}{qn}\right)\in\mathbb{Q}_{\geq 0}^d \mid \exists f\in {\left(I_{qn}\right)}_{pn}\;\;\text{with}\;\;\nu(f)=i_1\xi_1+\cdots+i_d\xi_d\right\}}$$ is well-defined, independent of the rational representation of $x$. Here $\mathbb{R}^d$ is identified with $\mathbb{R}^d\times \{1\}$. From  Lemma \ref{lem_properties}\;$(v)$ and Proposition \ref{prop_rational}, one has
\begin{equation*}
 f_{\mathbb{I}}(x) = \dfrac{d!}{q^{d-1}}\cdot\dfrac{\mathrm{vol}_{d-1}\big(\Delta\big(\Gamma^{\mathbb{I}}_{(p,q)}\big)\big)}{\mathrm{ind}(R)} = \frac{d!}{\mathrm{ind}(R)}\cdot \mathrm{vol}_{d-1}\left(\Delta_x(\mathbb{I})\right)
\end{equation*}
by homogenity of volume.
\end{remark}

\subsection{Proof of the Main Theorem}

\begin{proof}[Proof of Theorem \ref{mainthm}]
Part $(ii)$ has already been established in Lemma \ref{easy}.

In order to prove parts $(i)$, $(iii)$, and $(iv)$, we will first assume that $x,y\in \mathbb{Q} \cap (\alpha_{\mathbb{I}},\infty)$, and $\lambda \in \mathbb{Q}\cap [0,1]$. We first claim that $$\lambda \cdot \Delta_x(\mathbb{I}) + (1-\lambda)\cdot \Delta_y(\mathbb{I}) \subseteq \Delta_{\lambda x + (1-\lambda)y}(\mathbb{I}).$$

Write $\lambda = \tfrac{s}{t}$, $x=\tfrac{p}{q}$, and $y=\tfrac{p^{\prime}}{q^{\prime}}$. Choose elements $f\in {(I_{qn_0})}_{pn_0}$ and $g\in {(I_{q^{\prime}n_0^{\prime}})}_{p^{\prime}n_0^{\prime}}$ with $\nu(f)=\sum_{j=1}^d i_j\xi_j$ and $\nu(g)=\sum_{j=1}^d i_j^{\prime}\xi_j$ respectively. Then $\big(\tfrac{i_1}{qn_0},\ldots,\tfrac{i_d}{qn_0}\big) \in \Delta_x(\mathbb{I})$ and $\big(\tfrac{i_1^{\prime}}{q^{\prime}n_0^{\prime}},\ldots,\tfrac{i_d^{\prime}}{q^{\prime}n_0^{\prime}}\big) \in \Delta_y(\mathbb{I})$. After replacing $f$ and $g$ with $f^{q^{\prime}n_0^{\prime}}$ and $g^{qn_0}$ respectively, we may assume that $q=q^{\prime}$ and $n_0=n_0^{\prime}$. Then $\lambda x + (1-\lambda)y = \tfrac{(ps+p^{\prime}(t-s))n_0}{qtn_0}$, $\nu(f^sg^{t-s}) = \sum_{j=1}^d \left(si_j+(t-s)i_j^{\prime}\right)\xi_j$, and $$f^sg^{t-s} \in {\left(I_{qsn_0}\right)}_{psn_0}\cdot {\left(I_{q(t-s)n_0}\right)}_{p^{\prime}(t-s)n_0}\subseteq  {\left(I_{qtn_0}\right)}_{(ps+p^{\prime}(t-s))n_0}.$$
Thus,
$$\lambda\cdot\big(\tfrac{i_1}{qn_0},\ldots,\tfrac{i_d}{qn_0}\big) + (1-\lambda)\cdot\big(\tfrac{i_1^{\prime}}{qn_0},\ldots,\tfrac{i_d^{\prime}}{qn_0}\big) = \big(\tfrac{si_1+(t-s)i_1^{\prime}}{qtn_0}, \cdots, \tfrac{si_d+(t-s)i_d^{\prime}}{qtn_0}\big) \in \Delta_{\lambda x + (1-\lambda)y}(\mathbb{I}),$$ and our claim is proven.

Note that $\Delta_x(\mathbb{I})$ and $\Delta_y(\mathbb{I})$ are nonempty compact convex sets in $\mathbb{R}^d$. By the multiplicative form of the Brunn-Minkowski inequality \cite[Theorem 8.2]{Gru07},
\begin{align*}
 \mathrm{vol}_{d-1}\left(\Delta_{\lambda x + (1-\lambda)y}(\mathbb{I})\right) &\geq \mathrm{vol}_{d-1}\left(\lambda\cdot \Delta_x(\mathbb{I}) + (1-\lambda)\cdot \Delta_y(\mathbb{I})\right)\\
 &\geq {\left(\mathrm{vol}_{d-1}\left(\Delta_x(\mathbb{I})\right)\right)}^{\lambda}\cdot {\left(\mathrm{vol}_{d-1}\left(\Delta_y(\mathbb{I})\right)\right)}^{1-\lambda}.
\end{align*}
By Remark \ref{suffices} $$f_{\mathbb{I}}(z) = \frac{d!}{\mathrm{ind}(R)}\cdot \mathrm{vol}_{d-1}\left(\Delta_z(\mathbb{I})\right);\quad z\in \mathbb{Q} \cap (\alpha_{\mathbb{I}},\infty),$$ we obtain $$f_{\mathbb{I}}(\lambda x + (1-\lambda)y) \geq {\left(f_{\mathbb{I}}(x)\right)}^{\lambda}\cdot {\left(f_{\mathbb{I}}(y)\right)}^{1-\lambda}.$$ Thus $g(z) = -\ln\left(f_{\mathbb{I}}(z)\right)$ is convex on $\mathbb{Q} \cap (\alpha_{\mathbb{I}},\infty)$, hence continuous there \cite[Theorem 1.1]{Gru07}. Therefore, $f_{\mathbb{I}}(z) = e^{-g(z)}$ is continuous and log-concave on this domain.

Now let $x>\alpha_{\mathbb{I}}$ be any real number. Choose sequences ${\{x^{\prime}_t\}}_{t\in\mathbb{N}}$, ${\{x^{\prime\prime}_t\}}_{t\in\mathbb{N}}$ of rational numbers converging to $x$ from below and above, respectively. By Proposition \ref{prop_rational}, $f_{\mathbb{I}}(z)$ exists as a limit for all $z\in \mathbb{Q}\cap (\alpha_{\mathbb{I}},\infty)$. Moreover, we just showed that $f_{\mathbb{I}} \colon \mathbb{Q}\cap (\alpha_{\mathbb{I}},\infty) \to \mathbb{R}_{\geq 0}$ is continuous and log-concave. Since $R$ has a nonzero divisor of degree one, $f_{\mathbb{I}}$ is non-decreasing. So $$\lim_{n\to\infty}\dfrac{\dim_k {\left(I_n\right)}_{\lfloor x^{\prime}_t n \rfloor}}{n^{d-1}/d!} \leq \liminf_{n\to\infty}\dfrac{\dim_k {\left(I_n\right)}_{\lfloor xn \rfloor}}{n^{d-1}/d!} \leq \limsup_{n\to\infty}\dfrac{\dim_k {\left(I_n\right)}_{\lfloor xn \rfloor}}{n^{d-1}/d!} \leq \lim_{n\to\infty}\dfrac{\dim_k {\left(I_n\right)}_{\lfloor x^{\prime\prime}_t n \rfloor}}{n^{d-1}/d!}$$ for all $t\in \mathbb{N}$. Letting $t\to\infty$ gives the existence of the limit $f_{\mathbb{I}}(x)$ for all $x\in (\alpha_{\mathbb{I}},\infty)$. Hence, $f_{\mathbb{I}}$ is continuous and log-concave on $(\alpha_{\mathbb{I}},\infty)$. As $g(z) = -\ln\left(f_{\mathbb{I}}(z)\right)$ is convex, it is differentiable outside a countable set, see \cite[Theorem 1.4]{Gru07}. Thus, part $(iv)$ follows.

Finally, part $(v)$ is a direct consequence of Lebesgue's dominated convergence theorem \cite[Theorem 11.32]{Rud76}.
\end{proof}

\subsection{Density functions and Newton-Okounkov bodies}
The following remark is inspired from the construction given in \cite[Remark 4.14]{LM09}.

\begin{remark}\label{global_okounkov}
Consider the closed convex set, $$\Delta(\mathbb{I}) = \overline{\bigcup\left\{(\mathbf{v},x)\in \mathbb{R}^{d+1}\mid \mathbf{v}\in \Delta_x(\mathbb{I})\;\text{and}\;x\in \mathbb{Q}\cap (\alpha_{\mathbb{I}},\infty)\right\}},$$ where $\Delta_x(\mathbb{I})$ is as in Remark \ref{suffices}. Let $\pi\colon \mathbb{R}^{d+1}\to \mathbb{R}$ be the projection onto the last factor. For each real number $x>\alpha_{\mathbb{I}}$, we define $\Delta_x(\mathbb{I})$ to be the compact convex set $\Delta(\mathbb{I}) \cap \pi^{-1}(\{x\})$. Then
\begin{equation}\label{mult=vol}
f_{\mathbb{I}}(x) = \frac{d!}{\mathrm{ind}(R)}\cdot \mathrm{vol}_{d-1}\left(\Delta_x(\mathbb{I})\right)
\end{equation}
from the approximation of $x$ by rationals together with continuity of $f_{\mathbb{I}}$. The set $\Delta(\mathbb{I})$ enjoys the following properties:
\begin{enumerate}
 \item if $\mathbf{w}_1, \mathbf{w}_2 \in \Delta(\mathbb{I})$ then $\mathbf{w}_1 + \mathbf{w}_2 \in \Delta(\mathbb{I})$, and
 \item if $\mathbf{w} \in \Delta(\mathbb{I})$ then $\lambda\cdot\mathbf{w} \in \Delta(\mathbb{I})$ for all real numbers $\lambda$ such that $\lambda\pi(\mathbf{w}) \geq \alpha_{\mathbb{I}}$.
\end{enumerate}
Further, given two real numbers $b>a>\alpha_{\mathbb{I}}$, one has
\begin{align*}
 \lim_{n\to\infty}\dfrac{\sum_{m=\lfloor an\rfloor}^{\lfloor bn\rfloor}\dim_k {\left(I_n\right)}_m}{n^d/d!} = \int_a^b f_{\mathbb{I}}(x)dx &= \frac{d!}{\mathrm{ind}(R)}\int_a^b \mathrm{vol}_{d-1}\left(\Delta_x(\mathbb{I})\right)dx\\
 &= \frac{d!}{\mathrm{ind}(R)}\cdot \mathrm{vol}_{d}\left(\Delta(\mathbb{I}) \cap \pi^{-1}\left((a,b)\right)\right).
\end{align*}
\end{remark}

\begin{corollary}[Brunn-Minkowski inequality]\label{Brunn_Min}
Let $\mathbb{I}=\{I_n\}_{n\in\mathbb{N}}$ and $\mathbb{J}=\{J_n\}_{n\in\mathbb{N}}$ be filtrations of nonzero graded ideals in $R$. Then $${\left(f_{\mathbb{I}\cap \mathbb{J}}(x+y)\right)}^{\tfrac{1}{d-1}} \geq {\left(f_{\mathbb{IJ}}(x+y)\right)}^{\tfrac{1}{d-1}} \geq {\left(f_{\mathbb{I}}(x)\right)}^{\tfrac{1}{d-1}} + {\left(f_{\mathbb{J}}(y)\right)}^{\tfrac{1}{d-1}},$$ for all real numbers $x>\alpha_{\mathbb{I}}$ and $y>\alpha_{\mathbb{J}}$.
\end{corollary}
\begin{proof}
Recall the notations from Remark \ref{global_okounkov}. For all real numbers $x>\alpha_{\mathbb{I}}$ and $y>\alpha_{\mathbb{J}}$, we have $$ \Delta_x(\mathbb{I}) + \Delta_y(\mathbb{J}) \subseteq \Delta_{x+y}(\mathbb{IJ}) \subseteq \Delta_{x+y}(\mathbb{I\cap J}).$$ From the Brunn-Minkowski inequality \cite[Theorem 8.1]{Gru07}, we get $${\left(\mathrm{vol}_{d-1}\left(\Delta_{x+y}(\mathbb{I\cap J})\right)\right)}^{\tfrac{1}{d-1}} \geq {\left(\mathrm{vol}_{d-1}\left(\Delta_{x+y}(\mathbb{IJ})\right)\right)}^{\tfrac{1}{d-1}} \geq {\left(\mathrm{vol}_{d-1}\left(\Delta_x(\mathbb{I})\right)\right)}^{\tfrac{1}{d-1}} + {\left(\mathrm{vol}_{d-1}\left(\Delta_y(\mathbb{J})\right)\right)}^{\tfrac{1}{d-1}}.$$ In view of \eqref{mult=vol}, the desired inequality follows by multiplying throughout by ${\left(\frac{d!}{\mathrm{ind}(R)}\right)}^{1/(d-1)}$.
\end{proof}

\subsection{Density functions for filtrations of ideals in local rings}\label{local_density}

Suppose that $(R,\mm_R)$ be a Noetherian local ring such that its associated graded ring $\mathrm{gr}_{\mm_R}(R) = \oplus_{m\geq 0}\mm_R^m/\mm_R^{m+1}$ is a domain, the residue field $k=R/\mm_R$ is algebraically closed, and $d=\dim R\geq 2$. Let $\mathbb{I} = \{I_n\}_{n\in\mathbb{N}}$ be a filtration of nonzero ideals in $R$. Define the \emph{density} function $f_{\mathbb{I}} \colon \mathbb{R}_{\geq 0} \to \mathbb{R}_{\geq 0}$  of the filtration $\mathbb{I} = \{I_n\}_{n\in\mathbb{N}}$ by $$f_{\mathbb{I}}(x) = \limsup_{n\to\infty} \dfrac{\dim_k\big(\big(I_n\cap \mm_R^{\lfloor xn\rfloor}\big)/\big(I_n\cap \mm_R^{\lfloor xn\rfloor + 1}\big)\big)}{n^{d-1}/d!}.$$ Define the \emph{Waldschmidt constant} of $\mathbb{I}$ to be the real number $$\alpha_{\mathbb{I}} = \limsup_{n\to\infty}\dfrac{\min\left\{m\in\mathbb{N} \mid I_n\cap \mm_R^m \neq I_n\cap \mm_R^{m+1}\right\}}{n}.$$ Since $\mathrm{gr}_{\mm_R}(R)$ is a domain, it follows that $\alpha_{\mathbb{I}}$ exists as a limit.

\begin{theorem}
 Let the notations be as in Section \ref{local_density}. Then the following statements are true.
 \begin{enumerate}[$(i)$]
  \item $f_{\mathbb{I}}(x)$ exists as a limit for all $x\in \mathbb{R}_{\geq 0}\setminus\{\alpha_{\mathbb{I}}\}$.
  \item $f_{\mathbb{I}}(x) = 0$ for all $x<\alpha_{\mathbb{I}}$.
  \item On the interval $(\alpha_{\mathbb{I}},\infty)$, the function $f_{\mathbb{I}}$ is positive, nondecreasing, and log-concave. In particular, it is continuous.
  \item For every real number $y>0$, we have $$\int_{0}^y f_{\mathbb{I}}(x) dx = \lim_{n\to\infty}\dfrac{\lambda_R\big(I_n/\big(I_n\cap \mm_R^{\lfloor yn\rfloor}\big)\big)}{n^d/d!}.$$
 \end{enumerate}
\end{theorem}
\begin{proof}
By \cite[Corollary 10.19]{AM69}, we get $\bigcap_{n\geq 0}\mm_R^n = 0$. Given a nonzero element $f\in R$, define the \emph{order} of $f$ to be $$\mathrm{ord}(f) = \max\{n\in\NN \mid f\in \mm_R^n\}.$$ The \emph{initial form} of $f$ is the residue class $\mathrm{in}(f)$ of $f$ in $\mm_R^{\mathrm{ord}(f)}/\mm_R^{\mathrm{ord}(f)+1}$. Given a nonzero ideal $I\subseteq R$, the \emph{initial ideal} of $I$ is the graded ideal $\mathrm{in}(I)$ of $\mathrm{gr}_{\mm_R}(R)$ generated by $\{\mathrm{in}(f)\mid f\in I\}$. Since the associated graded ring $\mathrm{gr}_{\mm_R}(R)$ is a domain, the order function is a valuation, see \cite[Chapter VIII, Theorem 1]{ZS60}. Note that $\mathrm{in}(\mathbb{I}) = \{\mathrm{in}(I_n)\}_{n\in\mathbb{N}}$ is a filtration of nonzero graded ideals in the standard graded $k$-algebra $\mathrm{gr}_{\mm_R}(R)$, $\alpha_{\mathbb{I}} = \alpha_{\mathrm{in}(\mathbb{I})}$, and $$\dim_k \left(\frac{I_n \cap \mm_R^m}{I_n \cap \mm_R^{m+1}}\right) = \dim_k {\left(\mathrm{in}(I_n)\right)}_m \quad \forall m,n\in\NN.$$ Now the assertions of our theorem is a direct consequence of Theorem \ref{mainthm}.
\end{proof}

There exists an ideal $I$ in a regular local ring $(R,\mm_R)$ such that the filtration $\{\mathrm{in}(I^n)\}_{n\in\mathbb{N}}$ is non-Noetherian, see \cite[Example 3.4]{CHS10}.

\begin{corollary}[Brunn-Minkowski inequality]
 Let $(R,\mm_R)$ be as in Section \ref{local_density}. Let $\mathbb{I} = \{I_n\}_{n\in\mathbb{N}}$ and $\mathbb{J} = \{J_n\}_{n\in\mathbb{N}}$ be filtrations of nonzero ideals in $R$. Then $${\left(f_{\mathbb{I}\cap \mathbb{J}}(x+y)\right)}^{\tfrac{1}{d-1}} \geq {\left(f_{\mathbb{IJ}}(x+y)\right)}^{\tfrac{1}{d-1}} \geq {\left(f_{\mathbb{I}}(x)\right)}^{\tfrac{1}{d-1}} + {\left(f_{\mathbb{J}}(y)\right)}^{\tfrac{1}{d-1}},$$ for all real numbers $x>\alpha_{\mathbb{I}}$ and $y>\alpha_{\mathbb{J}}$.
\end{corollary}
\begin{proof}
Here we apply Corollary \ref{Brunn_Min} to the filtrations $\mathrm{in}(\mathbb{I}) = \{\mathrm{in}(I_n)\}_{n\in\mathbb{N}}$ and $\mathrm{in}(\mathbb{J}) = \{\mathrm{in}(J_n)\}_{n\in\mathbb{N}}$ in the associated graded ring $\mathrm{gr}_{\mm_R}(R) = \oplus_{m\geq 0}\mm_R^m/\mm_R^{m+1}$.
\end{proof}

\section{Truncated filtrations, density functions and integral closures}\label{section 5}

Let $R$ be as in Setup \ref{setup}. Fix a filtration $\mathbb{I}=\{I_n\}_{n\in\mathbb{N}}$ of nonzero graded ideals in $R$.

\subsection{Approximation by Noetherian filtrations}\label{Noeth_Filt}

Following \cite[Definition 4.1]{CSS19}, given an integer $a>0$, we define the \emph{$a$-th truncated filtration} of $\mathbb{I}$ as $\mathbb{I}^{[a]} = \big\{I^{[a]}_n\big\}_{n\in\mathbb{N}}$, where $$I^{[a]}_n=\begin{cases}
 I_n & \text{if}\;\;0\leq n\leq a,\\
 \sum_{i=1}^{n-1} I^{[a]}_i\cdot I^{[a]}_{n-i} & \text{if}\;\;n>a.                                                                                                                                                                                           \end{cases}
$$
By construction, $\mathbb{I}^{[a]}$ is a Noetherian filtration of graded ideals in $R$. The density function associated with a Noetherian filtration is known to be piecewise polynomial, see \cite[Theorem 4.4]{DRT25} for a precise formulation.

\begin{theorem}\label{approx_noeth_filt}
Let $R$ and $\mathbb{I}=\{I_n\}_{n\in\mathbb{N}}$ be as above. Then the following statements are true.
\begin{enumerate}[$(i)$]
 \item $\left\{\alpha_{\mathbb{I}^{[a]}}\right\}_{a=1}^{\infty}$ is a non-increasing sequence of positive rational numbers, and $\lim\limits_{a\to\infty} \alpha_{\mathbb{I}^{[a]}} = \alpha_{\mathbb{I}}$.
 \item For all $x<\alpha_{\mathbb{I}}$, we have $\lim\limits_{a\to\infty} f_{\mathbb{I}^{[a]}}(x)=0$.
 \item For all $x>\alpha_{\mathbb{I}}$, the sequence $\left\{f_{\mathbb{I}^{[a]}}(x)\right\}_{a=1}^{\infty}$ is non-decreasing, and $\lim\limits_{a\to\infty} f_{\mathbb{I}^{[a]}}(x)=f_{\mathbb{I}}(x)$.
 \item Let $[\eta_0,\eta_1]\subset \mathbb{R}_{\geq 0}$ be a compact interval which does not contain the point $\{\alpha_{\mathbb{I}}\}$. Then $f_{\mathbb{I}^{[a]}} \to f_{\mathbb{I}}$ uniformly on $[\eta_0,\eta_1]$.
 \item For every real number $y>0$, we have $$\lim_{a\to\infty}\int_0^{y}f_{\mathbb{I}^{[a]}}(x)dx = \int_0^{y}f_{\mathbb{I}}(x)dx.$$
\end{enumerate}
\end{theorem}
\begin{proof}
\textit{$(i)$} We have $f_{\mathbb{I}^{[a]}}(x) \leq f_{\mathbb{I}^{[a^{\prime}]}}(x) \leq f_{\mathbb{I}}(x)$ for every $a\leq a^{\prime}$ and $x\in\mathbb{R}_{\geq 0}$, which shows that $\{\alpha_{\mathbb{I}^{[a]}}\}_{a=1}^{\infty}$ is non-increasing and bounded below by $\alpha_{\mathbb{I}}$. Moreover, for every real number $x>\alpha_{\mathbb{I}}$, $f_{\mathbb{I}^{[a]}}(x)>0$ for all integers $a\gg 0$, hence $\alpha_{\mathbb{I}^{[a]}}\leq x$. Letting $x\downarrow \alpha_{\mathbb{I}}$ yields $\lim_{a\to\infty}\alpha_{\mathbb{I}^{[a]}}=\alpha_{\mathbb{I}}$.

\textit{$(ii)$} If $x<\alpha_{\mathbb{I}}$, then $f_{\mathbb{I}}(x)=0$ and $f_{\mathbb{I}^{[a]}}(x)\leq f_{\mathbb{I}}(x)=0$, hence $f_{\mathbb{I}^{[a]}}(x)= 0$ and the claim follows.

\textit{$(iii)$} Let $x\in \mathbb{Q}\cap (\alpha_{\mathbb{I}},\infty)$. Write $x=\tfrac{p}{q}$ so that ${(I_q)}_p\neq 0$. Then ${\big(I^{[a]}_q\big)}_p = {(I_q)}_p \neq 0$ for all integers $a\geq q$. Consider the Newton-Okounkov bodies (see Remark \ref{suffices})
$$\Delta_x(\mathbb{I}) = \overline{\bigcup_{n\in\mathbb{Z}_{\geq 1}}\left\{\left(\tfrac{i_1}{qn},\ldots,\tfrac{i_d}{qn}\right)\in\mathbb{Q}_{\geq 0}^d \mid \exists f\in {\left(I_{qn}\right)}_{pn}\;\;\text{with}\;\;\nu(f)=i_1\xi_1+\cdots+i_d\xi_d\right\}},$$ and for each $a\geq q$,
$$\Delta_x\big(\mathbb{I}^{[a]}\big) = \overline{\bigcup_{n\in\mathbb{Z}_{\geq 1}}\left\{\left(\tfrac{i_1}{qn},\ldots,\tfrac{i_d}{qn}\right)\in\mathbb{Q}_{\geq 0}^d \mid \exists f\in {\big(I^{[a]}_{qn}\big)}_{pn}\;\;\text{with}\;\;\nu(f)=i_1\xi_1+\cdots+i_d\xi_d\right\}}.$$
Because $I^{[a]}_{qn} \subseteq I^{[a']}_{qn} \subseteq I_{qn}$ for $a \leq a'$, we obtain $\Delta_x(\mathbb{I}^{[a]}) \subseteq \Delta_x(\mathbb{I}^{[a']}) \subseteq \Delta_x(\mathbb{I})$. Thus $\left\{\Delta_x\big(\mathbb{I}^{[a]}\big)\right\}_{a\geq q}$ is an increasing sequence of convex bodies, and since $I^{[a]}_{qn} = I_{qn}$ whenever $a \geq qn$, it follows that
\[
\overline{\bigcup_{a\geq q} \Delta_x\big(\mathbb{I}^{[a]}\big)} = \Delta_x(\mathbb{I}).
\]
By the countable additivity of Lebesgue measure, $\lim_{a\to\infty}\mathrm{vol}_{d-1}\left(\Delta_x\big(\mathbb{I}^{[a]}\big)\right) = \mathrm{vol}_{d-1}\left(\Delta_x(\mathbb{I})\right)$. Using Remark \ref{suffices} we conclude that $$\lim\limits_{a\to\infty} f_{\mathbb{I}^{[a]}}(x) = \lim\limits_{a\to\infty} \dfrac{\mathrm{vol}_{d-1}\left(\Delta_x\big(\mathbb{I}^{[a]}\big)\right)}{\mathrm{ind}(R)/d!} = \dfrac{\mathrm{vol}_{d-1}\left(\Delta_x(\mathbb{I})\right)}{\mathrm{ind}(R)/d!} = f_{\mathbb{I}}(x).$$ Now using the continuity property, one can show that $\lim\limits_{a\to\infty} f_{\mathbb{I}^{[a]}}(x)=f_{\mathbb{I}}(x)$ for all $x\in (\alpha_{\mathbb{I}},\infty)$.

\textit{$(iv)$} The uniform convergence of the sequence of functions $\{f_{\mathbb{I}^{[a]}}\}_{a=1}^{\infty}$ on a compact interval follows from Dini's theorem \cite[Theorem 7.13]{Rud76}.

\textit{$(v)$} The conclusion regarding integration then follows from part $(iv)$ together with \cite[Theorem 7.16]{Rud76}.
\end{proof}

\subsection{Filtrations of graded ideals having at most linear growth}

We begin with an elementary observation concerning pointwise limits of polynomial sequences.

\begin{lemma}\label{elementary}
Let $[a,b]\subset \mathbb{R}$ be a compact interval with $b>a$. Suppose there exists a sequence of polynomials $\{P_n\}_{n\in\mathbb{N}}$ that converges pointwise on $[a,b]$ to a function $P$, and assume that the degrees of the polynomials $P_n$ are uniformly bounded. Then $P$ is a polynomial function.
\end{lemma}

\begin{proof}
By assumption, there exists an integer $e>0$ such that $\deg P_n\le e$ for all $n$. Choose $e+1$ distinct points $x_0,\dots,x_e\in[a,b]$, for instance, $x_j = a+\tfrac{(b-a)j}{e}$. Write
\[
P_n(x) = \sum_{k=0}^e c_{n,k}x^k,
\]
where the coefficients $c_{n,k}\in\mathbb{R}$ and $c_{n,k}=0$ if $\deg P_n<e$. Let $V$ denote the Vandermonde matrix
\[
V =
\begin{bmatrix}
1 & x_0 & x_0^2 & \cdots & x_0^e\\
1 & x_1 & x_1^2 & \cdots & x_1^e\\
\vdots & \vdots & \vdots & \ddots & \vdots\\
1 & x_e & x_e^2 & \cdots & x_e^e
\end{bmatrix},
\]
which is invertible since $\det(V) = \prod\limits_{0\leq i<j\leq e}(x_i-x_j) \neq 0$. For each $n$, we have
\[
V
\begin{bmatrix}
c_{n,0}\\
c_{n,1}\\
\vdots\\
c_{n,e}
\end{bmatrix}
=
\begin{bmatrix}
P_n(x_0)\\
P_n(x_1)\\
\vdots\\
P_n(x_e)
\end{bmatrix}.
\]
Hence
\[
\begin{bmatrix}
c_{n,0}\\
c_{n,1}\\
\vdots\\
c_{n,e}
\end{bmatrix}
=
V^{-1}
\begin{bmatrix}
P_n(x_0)\\
P_n(x_1)\\
\vdots\\
P_n(x_e)
\end{bmatrix}.
\]
Since $\lim_{n\to\infty}P_n(x_j) = P(x_j)$ for each $j$, and matrix multiplication is continuous,
it follows that the coefficient vectors $(c_{n,0},\dots,c_{n,e})$ converge to some vector $(c_0,\dots,c_e)$. Therefore,
\[
P(x) = \lim_{n\to\infty} P_n(x) = \sum_{k=0}^e c_k x^k \quad \forall x\in[a,b].
\]
Thus $P$ is a polynomial of degree at most $e$.
\end{proof}

\begin{theorem}\label{even_linear}
Let $R$ and $\mathbb{I}=\{I_n\}_{n\in\mathbb{N}}$ be as above. Assume that there exists a constant $c>0$ such that each $I_n$ is generated in degrees $\leq cn$. Then there exists a polynomial $P_{\mathbb{I}}(x)$ of degree $d-1$ with real coefficients such that for all $x>c$, $$f_{\mathbb{I}}(x) = P_{\mathbb{I}}(x).$$
\end{theorem}
\begin{proof}
For every integer $a>0$, consider the $a$-th truncated filtration $\mathbb{I}^{[a]} = \big\{I^{[a]}_n\big\}_{n\in\mathbb{N}}$, as defined in Section \ref{Noeth_Filt}. By construction, $\mathbb{I}^{[a]}$ is a Noetherian filtration and the ideal $I^{[a]}_n$ is generated in degrees $\leq cn$. It then follows from \cite[Theorem 4.4]{DRT25} that there exists a polynomial $P_{\mathbb{I}^{[a]}}(x)$ of degree $d-1$ with rational coefficients such that for all $x>c$, $$f_{\mathbb{I}^{[a]}}(x) := \lim_{n\to\infty}\dfrac{\dim_k {\big(I^{[a]}_n\big)}_{\lfloor xn \rfloor}}{n^{d-1}/d!} = P_{\mathbb{I}^{[a]}}(x).$$ Moreover, by Theorem \ref{approx_noeth_filt}, the sequence $\left\{P_{\mathbb{I}^{[a]}}(x)\right\}_{a=1}^{\infty}$ converges pointwise on $(c,\infty)$ to the density function $f_{\mathbb{I}}(x):=\lim_{n\to\infty}\frac{\dim_k {\left(I_n\right)}_{\lfloor xn \rfloor}}{n^{d-1}/d!}$. Since the degrees of $P_{\mathbb{I}^{[a]}}$ are uniformly bounded by $d-1$ and the convergence is pointwise on $(c,\infty)$, Lemma \ref{elementary} implies that $f_{\mathbb{I}}$ coincides on $(0,\infty)$ with a polynomial $P_{\mathbb{I}}(x)$ of degree at most $d-1$. Because $P_{\mathbb{I}}(x) - P_{\mathbb{I}^{[a]}}(x) \geq 0$ on $(c,\infty)$ and every $P_{\mathbb{I}^{[a]}}(x)$ has degree exactly $d-1$, the leading coefficient of $P_{\mathbb{I}}(x)$ cannot vanish; consequently $\deg P_{\mathbb{I}} = d-1$.
\end{proof}

Some examples of filtrations of graded ideals satisfying the \emph{linear growth} condition in Theorem \ref{even_linear} can be found in \cite{HHT02}.

\subsection{Density functions and integral closures}
Let $\mathbb{I}=\{I_n\}_{n\in\mathbb{N}}$ and $\mathbb{J}=\{J_n\}_{n\in\mathbb{N}}$ be filtrations of nonzero graded ideals in $R$. We write $\mathbb{I} \subseteq \mathbb{J}$ if $I_n \subseteq J_n$ for all $n \in \mathbb{N}$. The inclusion $\mathbb{I} \subseteq \mathbb{J}$ is said to be \emph{integral} if the corresponding inclusion of the Rees algebras $$\bigoplus_{n\geq 0}I_nt^n \hookrightarrow \bigoplus_{n\geq 0}J_nt^n$$ is an integral extension of graded rings.

\begin{theorem}\label{density_integral}
Let $\mathbb{I}=\{I_n\}_{n\in\mathbb{N}} \subseteq \mathbb{J}=\{J_n\}_{n\in\mathbb{N}}$ be filtrations of nonzero graded ideals in $R$ such that the inclusion $\mathbb{I} \subseteq \mathbb{J}$ is integral. Then $\alpha_{\mathbb{I}} = \alpha_{\mathbb{J}}$; denote this common value by $\alpha$. Moreover, $$f_{\mathbb{I}}(x) = f_{\mathbb{J}}(x) \quad \forall x\in \mathbb{R}_{\geq 0}\setminus\{\alpha\}.$$
\end{theorem}
\begin{proof}
 We proceed following the strategy employed in the proof of \cite[Theorem 6.9]{CSS19}.

 \textit{Step 1.} Let $I \subseteq J$ be two nonzero graded ideals in $R$, and set $\mathbb{I} = \{I^n\}_{n \in \mathbb{N}}$ and $\mathbb{J} = \{J^n\}_{n \in \mathbb{N}}$. Suppose that $\mathbb{I} \subseteq \mathbb{J}$ is integral. Then, by \cite[Theorem 6.2]{DRT25}, the conclusions of Theorem \ref{density_integral} hold for these filtrations.

 \textit{Step 2.} Suppose that $\mathbb{I}=\{I_n\}_{n\in\mathbb{N}} \subseteq \mathbb{J}=\{J_n\}_{n\in\mathbb{N}}$ are Noetherian filtrations of nonzero graded ideals in $R$ such that $\mathbb{I} \subseteq \mathbb{J}$ is integral. Then there exists an integer $c>0$ such that $I_{cn} = I_c^n$ and $J_{cn} = J_c^n$ for all $n\in\mathbb{N}$. Consider the sub-filtrations $\mathbb{I}^{(c)} :=\{I_c^n\}_{n\in\mathbb{N}}$ and $\mathbb{J}^{(c)} :=\{J_c^n\}_{n\in\mathbb{N}}$. Their density functions satisfy $$f_{\mathbb{I}}(x) = \frac{1}{c^{d-1}}\cdot f_{\mathbb{I}^{(c)}}(cx), \quad \text{and}\quad f_{\mathbb{J}}(x) = \frac{1}{c^{d-1}}\cdot f_{\mathbb{J}^{(c)}}(cx) \quad \forall x\in \mathbb{R}_{\geq 0}.$$ Since $\mathbb{I} \subseteq \mathbb{J}$ is integral, it follows that $\mathbb{I}^{(c)} \subseteq \mathbb{J}^{(c)}$ is also integral. The desired conclusions now follow from Step 1 together with the above relations between the density functions.

 \textit{Step 3.} Suppose that $\mathbb{I}=\{I_n\}_{n\in\mathbb{N}} \subseteq \mathbb{J}=\{J_n\}_{n\in\mathbb{N}}$ are filtrations of nonzero graded ideals in $R$ such that $\mathbb{I} \subseteq \mathbb{J}$ is integral. Let $a>0$ be an integer and consider the $a$-th truncated filtration $\mathbb{J}^{[a]} = \big\{J^{[a]}_n\big\}_{n\in\mathbb{N}}$, as defined in subsection \ref{Noeth_Filt}. Then there exists an integer $c_a>0$ such that every element of $\oplus_{n\geq 0}J^{[a]}_nt^n$ (considered as a subring of $\oplus_{n\geq 0}J_nt^n$) is integral over $\oplus_{n\geq 0}I^{[c_a]}_nt^n$, where $\mathbb{I}^{[c_a]} = \big\{I^{[c_a]}_n\big\}_{n\in\mathbb{N}}$ is the $c_a$-th truncated filtration of $\mathbb{I}$, as defined in subsection \ref{Noeth_Filt}.

 Define a Noetherian filtration of graded ideals $\mathbb{K}_a = \{K_{a,n}\}_{n\in\mathbb{N}}$ in $R$ by $$K_{a,n} = \sum_{i=0}^n I^{[c_a]}_i J^{[a]}_{n-i}.$$ Since $I^{[c_a]}_0 = R = J^{[a]}_0$, it follows that $\mathbb{I}^{[c_a]} \subseteq \mathbb{K}_a$, and this inclusion is integral. By Step 2, the equality $$f_{\mathbb{I}^{[c_a]}}(x) = f_{\mathbb{K}_a}(x)$$ holds for all $x\in\mathbb{R}_{\geq 0}$, except possibly for a single value of $x$. From Theorem \ref{approx_noeth_filt}, we obtain $$\lim_{a\to\infty} f_{\mathbb{I}^{[c_a]}}(x) = f_{\mathbb{I}}(x) \quad \forall x\in\mathbb{R}_{\geq 0}\setminus\{\alpha_{\mathbb{I}}\},$$ and hence $$\lim_{a\to\infty} f_{\mathbb{K}_a}(x) = f_{\mathbb{I}}(x) \quad \forall x\in\mathbb{R}_{\geq 0}\setminus\{\alpha_{\mathbb{I}}\}.$$

 We also have inclusions $$\mathbb{J}^{[a]} \subseteq \mathbb{K}_a \subseteq \mathbb{J},$$ and again, by Theorem \ref{Noeth_Filt}, $$\lim_{a\to\infty} f_{\mathbb{J}^{[a]}}(x) = f_{\mathbb{J}}(x) \quad \forall x\in\mathbb{R}_{\geq 0}\setminus\{\alpha_{\mathbb{J}}\}.$$ Applying the sandwich theorem yields $$\lim_{a\to\infty}f_{\mathbb{K}_a}(x) = f_{\mathbb{J}}(x) \quad \forall x\in\mathbb{R}_{\geq 0}\setminus\{\alpha_{\mathbb{J}}\}.$$ The theorem now follows by combining the above equalities.
\end{proof}

\section{Saturated filtrations of graded ideals in a standard graded polynomial ring}\label{section 6}

\subsection{Setup}\label{setup3}
Let $k$ be an algebraically closed field. Let $R = k[x_1, \ldots, x_d]$ be the standard graded polynomial ring in the variables $x_1, \ldots, x_d$ over $k$ with graded maximal ideal $\mm_R = (x_1, \ldots, x_d)$. Assume that $d\geq 2$.

Let $\mathbb{I}=\{I_n\}_{n\in\mathbb{N}}$ be a filtration of graded ideals in $R$ satisfying $1\leq \mathrm{height}\;I_1\leq d-1$. Let $J\subseteq R$ be a graded ideal such that $\sqrt{J} \not\subset \sqrt{I_1}$ and $J \subseteq \mathfrak{m}_R$. Consider the filtration $\mathbb{I}^{\mathrm{sat}}_J = \left\{I_n \colon J^{\infty}\right\}_{n\in\mathbb{N}}$, where $$I_n \colon J^{\infty} = \bigcup_{t\ge 1} \left(I_n \colon J^t\right) = \left\{r\in R \mid rJ^t \subseteq I^n\;\;\text{for some $t\ge 1$}\right\}.$$ This is the \emph{saturation} of the filtration $\mathbb{I}$ with respect to $J$. We shall now study the associated density function $f_{\mathbb{I}^{\mathrm{sat}}_J}$.

\subsection{Generic initial ideals}
Let $>$ denote the \emph{graded reverse lexicographic order} on $R$ defined by: $\prod_{i=1}^d x_i^{a_i} > \prod_{i=1}^d x_i^{b_i}$ if $\sum_{i=1}^d a_i > \sum_{i=1}^d b_i$ or $\sum_{i=1}^d a_i = \sum_{i=1}^d b_i$ and the last nonzero entry of the vector $(a_1-b_1,\ldots,a_d-b_d)$ is negative. For a graded ideal $I \subseteq R$, write $\mathrm{Gin}(I)$ for the \emph{generic initial ideal} of $I$ with respect to $>$.

The following lemma records basic properties of generic initial ideals that will be repeatedly used in our asymptotic analysis. Some useful references are Chapter 15 of Eisenbud's book \cite{eisenbud} and the lecture notes by Green \cite{Gre98}.

\begin{lemma}\label{generic_lemma}
Let $R$, and $\mathbb{I}=\{I_n\}_{n\in\mathbb{N}}$ be as in Setup \ref{setup3}. Then the following statements are true.
\begin{enumerate}[$(i)$]
 \item $\mathrm{Gin}(\mathbb{I}) = \{\mathrm{Gin}(I_n)\}_{n\in\mathbb{N}}$ is a filtration of Borel-fixed monomial ideals in $R$.
 \item $\dim_k {\left(I_n\right)}_m = \dim_k {\left(\mathrm{Gin}(I_n)\right)}_m$ for all $m,n\in \mathbb{N}$.
 \item $\mathrm{Gin}\left(I_n \colon \mm_R^{\infty}\right) = \mathrm{Gin}(I_n) \colon \mm_R^{\infty} = \mathrm{Gin}(I_n) \colon (x_d)^{\infty}$. Consequently, none of the minimal generators of $\mathrm{Gin}\left(I_n \colon \mm_R^{\infty}\right)$ involves the variable $x_d$.
\end{enumerate}
\end{lemma}
\begin{proof}
\textit{$(i)$}\;
It is well-known that generic initial ideals are Borel-fixed monomial ideals; see \cite[Theorem 1.27]{Gre98} or \cite[Theorem 15.20]{eisenbud}. To verify that $\mathrm{Gin}(\mathbb{I})$ is a filtration, we adapt an argument from \cite[Lemma 2.12]{May14}. For each $n \ge 1$, Galligo's theorem \cite[Theorem 1.27]{Gre98} ensures that there exists a nonempty Zariski open subset $U_n \subseteq \mathrm{GL}_d(k)$ such that $\mathrm{Gin}(I_n) = \mathrm{in}_>(g(I_n))$ for all $g \in U_n$.
Since the intersection $U_n \cap U_m \cap U_{n+m}$ is nonempty, choose $g$ in this intersection so that $\mathrm{Gin}(I_i) = \mathrm{in}_>(g(I_i))$ for $i=n,m,n+m$. For $f \in I_n$ and $h \in I_m$, set $f^{\prime} = \mathrm{in}_>(g(f))$ and $h^{\prime} = \mathrm{in}_>(g(h))$. Then
\[
    f^{\prime} h^{\prime} = \mathrm{in}_>(g(f))\, \mathrm{in}_>(g(h)) = \mathrm{in}_>(g(f)g(h))
    = \mathrm{in}_>(g(fh)) \in \mathrm{Gin}(I_{n+m}),
\]
so that $\mathrm{Gin}(I_n) \cdot \mathrm{Gin}(I_m) \subseteq \mathrm{Gin}(I_{n+m})$. A similar argument will also show that if $I_{n+1} \subseteq I_n$ then $\mathrm{Gin}(I_{n+1}) \subseteq \mathrm{Gin}(I_n)$. Hence $\mathrm{Gin}(\mathbb{I})$ is a filtration of Borel-fixed monomial ideals.

\textit{$(ii)$}\; Both an invertible linear change of coordinates and the operation of taking an initial ideal preserve the Hilbert function; thus the equality of graded pieces follows immediately.

\textit{$(iii)$}\; For the graded reverse lexicographic order, it is known that
\[
    \mathrm{Gin}(I_n \colon \mm_R^{\infty})
    = \mathrm{Gin}(I_n) \colon \mm_R^{\infty}
    = \mathrm{Gin}(I_n) \colon (x_d)^{\infty},
\]
see \cite[Proposition 2.21]{Gre98} and \cite[Proposition 15.24]{eisenbud}.
Consequently, no minimal generator of $\mathrm{Gin}(I_n \colon \mm_R^{\infty})$ involves $x_d$.
\end{proof}

An example in \cite[Section 3]{HHT02} shows that the filtration $\{\operatorname{Gin}(I^n)\}_{n\in\mathbb{N}}$, where $I$ is a graded ideal in a polynomial ring $R$, need not be Noetherian.

\subsection{Differentiability of the saturation density function}

\begin{theorem}\label{saturated_density_diff}
Let $R$, $\mathbb{I}=\{I_n\}_{n\in\mathbb{N}}$, and $J$ be as in Setup \ref{setup3}. Then the saturation density function $$f_{\mathbb{I}^{\mathrm{sat}}_J}(x) = \lim_{n\to\infty}\dfrac{\dim_k {\left(I_n \colon J^{\infty}\right)}_{\lfloor xn \rfloor}}{n^{d-1}/d!},$$ is a continuous function on $\mathbb{R}_{\geq 0}$ whose support is the open interval $(\alpha_{\mathbb{I}^{\mathrm{sat}}_J},\infty)$. On its support, the function $f_{\mathbb{I}^{\mathrm{sat}}_J}$ is nondecreasing, log-concave and continuously differentiable. Moreover, its derivative $\frac{d}{dx}f_{\mathbb{I}^{\mathrm{sat}}_J}(x)$ is strictly positive, nondecreasing, and log-concave on the interval $(\alpha_{\mathbb{I}^{\mathrm{sat}}_J},\infty)$.
\end{theorem}
\begin{proof}
\emph{Step 1.} Assume that $J = \mathfrak{m}_R$. By Lemma  \ref{generic_lemma}, for all $m, n \in \mathbb{N}$, $$\dim_k {\left(I_n \colon \mm_R^{\infty}\right)}_m = \dim_k {\left(\mathrm{Gin}\left(I_n \colon \mm_R^{\infty}\right)\right)}_m = \dim_k {\left(\mathrm{Gin}(I_n) \colon (x_d)^{\infty}\right)}_m.$$ Since $\mathrm{Gin}(I_n) \colon (x_d)^{\infty}$ has no minimal generator involving $x_d$, there exists a monomial ideal $J_n \subseteq k[x_1,\ldots,x_{d-1}]$ such that
\[
J_nR = \mathrm{Gin}(I_n) \colon (x_d)^{\infty}.
\]
Thus $\mathbb{J} = \{J_n\}_{n\in\mathbb{N}}$ is a filtration of Borel-fixed monomial ideals in $k[x_1,\ldots,x_{d-1}]$. For each $m\in\mathbb{N}$, $${(J_nR)}_m = \sum_{i=0}^m {(J_n)}_i\cdot x_d^{m-i},$$ and hence $$\dim_k {\left(I_n \colon \mm_R^{\infty}\right)}_{\lfloor xn \rfloor} = \dim_k {\left(J_nR\right)}_{\lfloor xn \rfloor} = \sum_{i=0}^{\lfloor xn \rfloor} \dim_k {(J_n)}_i.$$ By Theorem \ref{mainthm}, the density function $$f_{\mathbb{J}}(y) = \lim_{n\to\infty}\frac{\dim_k {\left(J_n\right)}_{\lfloor yn \rfloor}}{n^{d-2}/(d-1)!}$$ exists as a limit, and is continuous on $\mathbb{R}_{\geq 0}\setminus \{\alpha_{\mathcal{J}}\}$. Moreover, $\alpha_{\mathcal{J}} = \alpha_{\mathbb{I}^{\mathrm{sat}}}$, since $$\dim_k (I_n \colon \mm_R^{\infty})_m = \dim_k (J_n R)_m \quad \forall m,n\in\mathbb{N}.$$ Using part $(v)$ of Theorem \ref{mainthm}, one has
\begin{align*}
 f_{\mathbb{I}^{\mathrm{sat}}_{\mathfrak{m}_R}}(x) = \lim_{n\to\infty}\dfrac{\dim_k {\left(\mathrm{Gin}(I_n)\colon (x_d)^{\infty}\right)}_{\lfloor xn \rfloor}}{n^{d-1}/d!} = \lim_{n\to\infty}\dfrac{\sum_{m=0}^{\lfloor xn \rfloor}\dim_k {\left(J_n\right)}_m}{n^{d-1}/d!} = d\int_0^x f_{\mathbb{J}}(y)dy.
\end{align*}
Since $f_{\mathbb{J}}$ is continuous except possibly at $y = \alpha_{\mathbb{I}^{\mathrm{sat}}_{\mathfrak{m}_R}}$, the function $$x \mapsto \int_0^x f_{\mathbb{J}}(y)dy$$ is continuous on $\mathbb{R}_{\ge 0}$, and continuously differentiable on $\mathbb{R}_{\ge 0} \setminus \{\alpha_{\mathbb{I}^{\mathrm{sat}}}\}$. By the fundamental theorem of calculus \cite[Theorem 6.20]{Rud76}, $$\frac{d}{dx}f_{\mathbb{I}^{\mathrm{sat}}_{\mathfrak{m}_R}}(x) = d\cdot\frac{d}{dx}\int_0^x f_{\mathbb{J}}(y)dy = d\cdot f_{\mathbb{J}}(x) \quad \forall x\neq \alpha_{\mathbb{I}^{\mathrm{sat}}_{\mathfrak{m}_R}}.$$   The remaining claims follow from part $(iii)$ of Theorem \ref{mainthm}.

\emph{Step 2}. Let $J$ be as in the hypothesis of the theorem. For every $n \geq 1$, one checks that $$I_n \colon J^{\infty} = \left(I_n \colon J^{\infty}\right) \colon \mm_R^{\infty}.$$ Now apply Step 1 to the filtration $\mathbb{I}^{\mathrm{sat}}_J=\left\{I_n \colon J^{\infty}\right\}_{n\in\mathbb{N}}$.
\end{proof}

The above Theorem \ref{saturated_density_diff} can be compared with Theorem \ref{symb_geom} and Remark \ref{symb_geom_rem} below.

\section{Generalized symbolic powers of a graded ideal in a standard graded domain}\label{section 7}

\subsection{Setup}\label{setup4}
Let $R$ be as in Setup \ref{setup}. Let $I\subseteq R$ be a graded ideal satisfying $1\leq \mathrm{height}\;I\leq d-1$. Fix another graded ideal $J\subseteq R$. Consider the following filtrations; $$\mathbb{I} = \{I^n\}_{n\in\mathbb{N}}\quad \text{and}\quad \mathbb{I}^{\mathrm{sat}}_J = \{I^n \colon J^{\infty}\}_{n\in\mathbb{N}},$$ together with their respective density functions $f_{\mathbb{I}}$ and  $f_{\mathbb{I}^{\mathrm{sat}}_J}$, and their corresponding Waldschmidt constants $\alpha_{\mathbb{I}}$ and $\alpha_{\mathbb{I}^{\mathrm{sat}}_J}$. To avoid trivialities, we further insist that $\sqrt{J} \not\subset \sqrt{I}$ and $J \subseteq \mathfrak{m}_R$.

\subsection{A geometric construction}
The material presented in this section is inspired by the approach of \cite[Section 5]{DRT25}, where the \emph{saturation density function} $f_{\mathbb{I}^{\mathrm{sat}}_{\mathfrak{m}_R}}$ was studied via its relationship with the volume function of divisors. We adapt and generalize these ideas to investigate the \emph{generalized symbolic density function} $f_{\mathbb{I}^{\mathrm{sat}}_J}$. For notations and standard facts about volume functions, we refer the reader to \cite[Chapter 2]{Laz04a} and \cite{LM09}.

\begin{theorem}\label{symb_geom}
Let $R$, $I$, and $J$ be as in Setup \ref{setup4}. Consider the projective variety $V = \mathrm{Proj}\, R$ with a very ample invertible sheaf $\mathcal{O}_V(1)$, and denote by $\mathcal{I}$ the ideal sheaf on $V$ corresponding to $I$. Then the following statements are true.
\begin{enumerate}[$(i)$]
 \item There exist a normal projective variety $X$, a projective birational morphism $\varphi\colon X \longrightarrow V$ dominating the blow-up of $V$ along $\mathcal{I}$, and an effective Cartier divisor $F$ on $X$ with exceptional support, such that for every real number $x\geq 0$,
 \begin{align*}
 f_{\mathbb{I}^{\mathrm{sat}}_J}(x) & = \lim_{n\to\infty}\dfrac{\dim_k {\left(\overline{I^n} \colon J^{\infty}\right)}_{\lfloor xn \rfloor}}{n^{d-1}/d!} = d\cdot\mathrm{vol}_X(xH-F),
\end{align*}
where $H$ denotes the pullback to $X$ of a hyperplane section on $V$.
\item The Waldschmidt constant $\alpha_{\mathbb{I}^{\mathrm{sat}}_J}$ is strictly positive and admits the following characterization:
 \begin{align*}
   \alpha_{\mathbb{I}^{\mathrm{sat}}_J}
      &= \min \{\, x \in \mathbb{R}_{\geq 0} \mid xH - F \text{ is pseudoeffective} \,\}.
 \end{align*}
 \item The density function $f_{\mathbb{I}^{\mathrm{sat}}_J}$ is continuous on $\mathbb{R}_{\geq 0}$, and is supported on the open interval $(\alpha_{\mathbb{I}^{\mathrm{sat}}_J}, \infty)$. Moreover, $f_{\mathbb{I}^{\mathrm{sat}}_J}$ is strictly increasing, log-concave, and continuously differentiable on $(\alpha_{\mathbb{I}^{\mathrm{sat}}_J}, \infty)$.
\end{enumerate}
\end{theorem}
\begin{proof}
Let $\mathcal{J}$ be the ideal sheaf on $V$ associated to the ideal $J$. Note that if $J$ is $\mathfrak{m}_R$-primary then $\mathcal{J}=\mathcal{O}_V$, otherwise $\mathcal{J}\subset \mathcal{O}_V$ is a proper ideal sheaf. From the discussions in Section \ref{gen_symb}, we may assume that $\mathcal{J}$ is a radical ideal sheaf, i.e., $\mathrm{Proj}\;R/J$ is reduced. Consider a primary decomposition $$\overline{\mathcal{I}^n} = \bigcap_{i=1}^s \mathcal{Q}_{i,n}$$ of the integral closure $\overline{\mathcal{I}^n}$ of the ideal sheaf $\mathcal{I}^n$. Here, $\mathcal{Q}_{i,n}$ are primary ideal sheaves for some closed integral subvarieties of $V$. This can be obtained by sheafifying a graded primary decomposition of the ideal $\overline{I^n} \colon \mathfrak{m}_R^{\infty}$. Then $$\overline{\mathcal{I}^n} : \mathcal{J}^{\infty} = \bigcap_{\{i \mid \mathcal{J} \not\subset \sqrt{\mathcal{Q}_{i,n}}\}} \mathcal{Q}_{i,n}.$$

This can be alternatively described as follows. Consider the morphism $$\pi\colon X^{\prime} \longrightarrow V,$$ where $X^{\prime}$ is the normalization of the blow-up of $V$ along $\mathcal{I}$. By construction $\mathcal{I}\mathcal{O}_{X^{\prime}}$ is an invertible sheaf, so there exists an effective Cartier divisor $E^{\prime}$ on $X^{\prime}$ with $\mathcal{I}\mathcal{O}_{X^{\prime}} = \mathcal{O}_{X^{\prime}}(-E^{\prime})$ . Write $$E^{\prime} = \sum_{i=1}^{s} a_i E_i^{\prime},$$ where $a_i\in\mathbb{Z}_{>0}$ and $E_i^{\prime}$ are the irreducible components of the support of $E^{\prime}$. For each $n\geq 1$ and $1\leq i\leq s$ define the (reflexive rank one) ideal sheaf $\mathcal{O}_{X^{\prime}}(-a_inE_i^{\prime})$ on an open affine subset $U \subseteq X^{\prime}$ as follows: if $U\cap E^{\prime} \neq \emptyset$, then define $$\Gamma\left(U,\mathcal{O}_{X^{\prime}}(-a_inE_i^{\prime})\right) = \left\{f\in k[U] \mid \mathrm{ord}_{E_i^{\prime}}(f)\geq a_in\right\}.$$ Otherwise, if $U\cap E^{\prime} = \emptyset$ then define $\Gamma\left(U,\mathcal{O}_{X^{\prime}}(-a_inE_i^{\prime})\right)=0$. Note that the subsheaves $\mathcal{O}_{V} \cap \pi_*\mathcal{O}_{X^{\prime}}(-a_inE_i) \subseteq \mathcal{O}_{V}$ are primary ideal sheaves on $V$, and one has
\begin{equation}\label{reesval0}
\mathcal{O}_{V} \cap \pi_*\mathcal{O}_{X^{\prime}}(-nE^{\prime}) = \bigcap_{i=1}^s \left(\mathcal{O}_{V} \cap \pi_*\mathcal{O}_{X^{\prime}}(-a_inE_i^{\prime})\right) =  \overline{\mathcal{I}^n},
\end{equation}
see \cite[Section 9.6.A]{Laz04b} and \cite[Chapter 10]{HS06}. Thus, $\bigcap_{i=1}^s \left(\mathcal{O}_{V} \cap \pi_*\mathcal{O}_{X^{\prime}}(-a_inE_i^{\prime})\right)$ is a (possibly redundant) primary decomposition of $\overline{\mathcal{I}^n}$. In particular, every associated prime of $\overline{\mathcal{I}^n}$ arises from the centre on $V$ of some divisorial valuation $\mathrm{ord}_{E_i'}$. Consider the subset $$\left\{\mathrm{ord}_{E_i^{\prime}} \mid \mathcal{J} \not\subset \sqrt{\mathcal{O}_{V} \cap \pi_*\mathcal{O}_{X^{\prime}}(-a_iE_i^{\prime})}\right\}.$$ After reindexing, we may assume that $\mathrm{ord}_{E_1^{\prime}},\ldots,\mathrm{ord}_{E_r^{\prime}}$ are precisely the divisorial valuations among $\mathrm{ord}_{E_1^{\prime}},\ldots,\mathrm{ord}_{E_s^{\prime}}$ that occur in the set above. Now set $$F^{\prime}=\sum_{i=1}^{r} a_i E_i^{\prime},$$ and observe that for all integers $n\geq 1$,
\begin{equation}\label{reesval1}
 \mathcal{O}_V \cap \pi_*\mathcal{O}_{X^{\prime}}(-nF^{\prime}) = \bigcap_{i=1}^r \left(\mathcal{O}_{V} \cap \pi_*\mathcal{O}_{X^{\prime}}(-a_inE_i^{\prime})\right) = \overline{\mathcal{I}^n} : \mathcal{J}^{\infty}.
\end{equation}

Note that if $V$ is normal, then $\pi_*\mathcal{O}_{X^{\prime}}(-a_inE_i) \subseteq \pi_*\mathcal{O}_{X^{\prime}} = \mathcal{O}_{V}$ are already primary ideal sheaves on $V$, in which case the additional intersection with $\mathcal{O}_V$ is unnecessary.

Let $H^\prime$ be the pullback to $X^{\prime}$ of a hyperplane section on $V$. By the projection formula \cite[Chapter II, Exercise 5.1]{Har77}, for all $m,n\in\mathbb{N}$, we have
\begin{align}\label{reesval2}
 \Gamma\left(X^{\prime},\mathcal{O}_{X^{\prime}}(mH^{\prime}-nF^{\prime})\right) &= \Gamma\left(V,\pi_*\mathcal{O}_{X^{\prime}}(-nF^{\prime})\otimes \mathcal{O}_V(m)\right).
\end{align}

Note that $F^{\prime}$ is an effective Weil divisor but not necessarily Cartier. To fix this, let $$\theta\colon X \longrightarrow X^{\prime}$$ be the normalization of the blow-up of the divisorial ideal sheaf $\mathcal{O}_{X^{\prime}}(-F^{\prime})$, see also \cite[Corollary 3.4]{FKL16}. Then there exist effective Cartier divisors $F$, $E$, and $H$ on $X$ such that $E\geq F$, $\theta^*\mathcal{O}_{X^{\prime}}(-F^{\prime}) = \mathcal{O}_X(-F)$, $\theta^*\mathcal{O}_{X^{\prime}}(-E^{\prime}) = \mathcal{O}_X(-E)$, and $\theta^*\mathcal{O}_{X^{\prime}}(H^{\prime}) = \mathcal{O}_X(H)$ respectively. For all $m,n\in\mathbb{N}$, we have
\begin{equation}\label{reesval3}
\theta_*\mathcal{O}_{X}(mH-nF) = \mathcal{O}_{X^{\prime}}(mH^{\prime}-nF^{\prime}), \quad \text{and} \quad \theta_*\mathcal{O}_{X}(mH-nE) = \mathcal{O}_{X^{\prime}}(mH^{\prime}-nE^{\prime}).
\end{equation}
Again using the projection formula, we obtain that for all $m,n\in\mathbb{N}$,
\begin{align}\label{reesval4}
 \Gamma\left(X,\mathcal{O}_{X}(mH-nF)\right) &= \Gamma\left(X^{\prime},\theta_*\mathcal{O}_{X}(mH-nF)\right) = \Gamma\left(X^{\prime},\mathcal{O}_{X^{\prime}}(mH^{\prime}-nF^{\prime})\right), \;\text{and}\\
 \Gamma\left(X,\mathcal{O}_{X}(mH-nE)\right) &= \Gamma\left(X^{\prime},\theta_*\mathcal{O}_{X}(mH-nE)\right) = \Gamma\left(X^{\prime},\mathcal{O}_{X^{\prime}}(mH^{\prime}-nE^{\prime})\right).\nonumber
\end{align}

\vspace{0.6em}
\noindent\textit{Claim 1.} For every real number $x\geq 0$,
\begin{align*}
 \lim_{n\to\infty}\dfrac{\dim_k \Gamma \left(X,\mathcal{O}_X\left(\lfloor xn\rfloor H - nF\right)\right)}{n^{d-1}/(d-1)!} &= \mathrm{vol}_X(xH-F),\;\text{and}\\
 \lim_{n\to\infty}\dfrac{\dim_k \Gamma \left(X,\mathcal{O}_X\left(\lfloor xn\rfloor H - nE\right)\right)}{n^{d-1}/(d-1)!} &= \mathrm{vol}_X(xH-E).
\end{align*}
\textit{Proof of Claim 1.} For $x\in\mathbb{Z}_{\geq 0}$, the claim follows directly from the definition of the volume of a Cartier divisor and the fact that the volume exists as a limit, see \cite[Theorem A]{LM09}. For $x\in\mathbb{Q}_{\geq 0}$, the claim holds by arguments similar to that in \cite[Lemma 2.2.38]{Laz04a}. For $x\in\mathbb{R}_{\geq 0}$, the result follows from the continuity of the volume function on $N^1(X)_{\mathbb{R}}$, see \cite[Corollary 2.2.45]{Laz04a}.

\vspace{0.6em}
\noindent\textit{Claim 2.} For every real number $x\geq 0$,
\begin{align*}
 \lim_{n\to\infty}\dfrac{\dim_k \Gamma \big(V,(\overline{\mathcal{I}^n}\colon \mathcal{J}^{\infty})\otimes\mathcal{O}_V\left(\lfloor xn\rfloor\right)\big)}{n^{d-1}/(d-1)!} &= \mathrm{vol}_X(xH-F), \;\text{and}\\
 \lim_{n\to\infty}\dfrac{\dim_k \Gamma \big(V,\overline{\mathcal{I}^n}\otimes\mathcal{O}_V\left(\lfloor xn\rfloor\right)\big)}{n^{d-1}/(d-1)!} &= \mathrm{vol}_X(xH-E).
\end{align*}
\textit{Proof of Claim 2.} Let $\varphi = \pi\circ \theta $ be the composition map from $X$ to $V$. Since $\varphi$ is birational, there is an exact sequence of $\mathcal{O}_V$-modules, $$0 \to \mathcal{O}_V \to \varphi_*\mathcal{O}_X \to \mathcal{G}\to 0,$$ where $\dim \left(\mathrm{Supp}\;\mathcal{G}\right)\leq \dim V -1 = d-2$. For each $n\in\mathbb{N}$, this induces exact sequences $$0 \to \overline{\mathcal{I}^n}\colon \mathcal{J}^{\infty} = \mathcal{O}_{V} \cap \varphi_*\mathcal{O}_{X}(-nF) \to \varphi_*\mathcal{O}_{X}(-nF) \to \mathcal{G}_n \to 0,$$ where $\mathcal{G}_n \subseteq \mathcal{G}$ is a coherent $\mathcal{O}_V$-module. Tensoring with $\mathcal{O}_V\left(\lfloor xn\rfloor\right)$ and taking global sections yields the left exact sequence $$0 \to \Gamma\big(V,(\overline{\mathcal{I}^n}\colon \mathcal{J}^{\infty})\otimes\mathcal{O}_V\left(\lfloor xn\rfloor\right)\big) \to \Gamma\big(V,\varphi_*\mathcal{O}_X(-nF) \otimes \mathcal{O}_V\left(\lfloor xn\rfloor\right)\big) \to \Gamma\big(V,\mathcal{G}_n\otimes \mathcal{O}_V\left(\lfloor xn\rfloor\right)\big).$$ Since $\mathcal{G}$ is supported on a closed subscheme of dimension $\leq d-2$, we have $$\lim_{n\to\infty}\dfrac{\dim_k \Gamma\big(V,\mathcal{G}_n\otimes \mathcal{O}_V\left(\lfloor xn\rfloor\right)\big)}{n^{d-1}/(d-1)!} \leq \lim_{n\to\infty}\dfrac{\dim_k \Gamma\big(V,\mathcal{G}\otimes \mathcal{O}_V\left(\lfloor xn\rfloor\right)\big)}{n^{d-1}/(d-1)!} = 0.$$ Using this fact together with the projection formula and Claim 1 gives the desired limit, i.e.,
\begin{align*}
 \lim_{n\to\infty}\dfrac{\dim_k \Gamma \big(V,(\overline{\mathcal{I}^n}\colon \mathcal{J}^{\infty})\otimes\mathcal{O}_V\left(\lfloor xn\rfloor\right)\big)}{n^{d-1}/(d-1)!} &= \lim_{n\to\infty}\dfrac{\dim_k \Gamma \big(V,\varphi_*\mathcal{O}_{X}(-nF)\otimes\mathcal{O}_V\left(\lfloor xn\rfloor\right)\big)}{n^{d-1}/(d-1)!}\\
 &= \lim_{n\to\infty}\dfrac{\dim_k \Gamma \left(X,\mathcal{O}_X\left(\lfloor xn\rfloor H - nF\right)\right)}{n^{d-1}/(d-1)!}\\
 &= \mathrm{vol}_X(xH-F).
\end{align*}

\vspace{0.6em}
\noindent\textit{Claim 3.} For every real number $x\geq 0$,
\begin{align*}
 \lim_{n\to\infty}\dfrac{\dim_k \Gamma \big(V,(\mathcal{I}^n\colon \mathcal{J}^{\infty})\otimes\mathcal{O}_V\left(\lfloor xn\rfloor\right)\big)}{n^{d-1}/(d-1)!} &= \mathrm{vol}_X(xH-F),\;\text{and}\\
 \lim_{n\to\infty}\dfrac{\dim_k \Gamma \big(V,\mathcal{I}^n\otimes\mathcal{O}_V\left(\lfloor xn\rfloor\right)\big)}{n^{d-1}/(d-1)!} &= \mathrm{vol}_X(xH-E).
\end{align*}
\noindent\textit{Proof of Claim 3.} By \cite[Theorem 1.4]{Ree61a}, there exists an integer $c>0$, depending only on $\mathcal{I}$, such that for all $n>0$, $\overline{\mathcal{I}^{n+c}} \subseteq \mathcal{I}^{n} \subseteq \overline{\mathcal{I}^n}$. Hence, for all $n>0$, $${(\overline{\mathcal{I}^{n+c}})}\colon \mathcal{J}^{\infty} \subseteq {(\mathcal{I}^{n})}\colon \mathcal{J}^{\infty} \subseteq {(\overline{\mathcal{I}^n})}\colon \mathcal{J}^{\infty}.$$

Because $cF$ is an effective Cartier divisor on $X$, there is an exact sequence of $\mathcal{O}_X$-modules $$0 \to \mathcal{O}_X(-cF) \to \mathcal{O}_X \to \mathcal{O}_{cF} \to 0.$$ Tensoring this with $\mathcal{O}_X\left(\lfloor xn\rfloor H - nF\right)$ and taking global sections gives the left exact sequence $$0 \to \Gamma \left(X,\mathcal{O}_X\left(\lfloor xn\rfloor H - (n+c)F\right)\right) \to \Gamma \left(X,\mathcal{O}_X\left(\lfloor xn\rfloor H - nF\right)\right) \to \Gamma \left(X,\mathcal{O}_{cF}\otimes\mathcal{O}_X\left(\lfloor xn\rfloor H - nF\right)\right).$$ Since $\mathcal{O}_{cF}$ is supported on a scheme of dimension $\leq d-2$, $$\lim_{n\to\infty}\dfrac{\dim_k \Gamma \left(X,\mathcal{O}_{cF}\otimes\mathcal{O}_X\left(\lfloor xn\rfloor H - nF\right)\right)}{n^{d-1}/(d-1)!}=0.$$ Using this along with Claim 2, we get $$\lim_{n\to\infty}\dfrac{\dim_k \Gamma \left(X,\mathcal{O}_X\left(\lfloor xn\rfloor H - (n+c)F\right)\right)}{n^{d-1}/(d-1)!} = \lim_{n\to\infty}\dfrac{\dim_k \Gamma \left(X,\mathcal{O}_X\left(\lfloor xn\rfloor H - nF\right)\right)}{n^{d-1}/(d-1)!} = \mathrm{vol}_X(xH-F).$$

Moreover, by an argument analogous to that used in the proof of Claim 2, $$\lim_{n\to\infty}\dfrac{\dim_k \Gamma \big(V,(\overline{\mathcal{I}^{n+c}}\colon \mathcal{J}^{\infty})\otimes\mathcal{O}_V\left(\lfloor xn\rfloor\right)\big)}{n^{d-1}/(d-1)!} = \lim_{n\to\infty}\dfrac{\dim_k \Gamma \left(X,\mathcal{O}_X\left(\lfloor xn\rfloor H - (n+c)F\right)\right)}{n^{d-1}/(d-1)!} = \mathrm{vol}_X(xH-F).$$

Combining this with Claim 2 yields the desired asymptotic equalities by sandwich arguments.

\vspace{0.6em}
\noindent\textit{Claim 4.} For every real number $x\geq 0$,
\begin{align*}
 \lim_{n\to\infty}\dfrac{\dim_k {\left(I^n \colon J^{\infty}\right)}_{\lfloor xn \rfloor}}{n^{d-1}/(d-1)!} = \mathrm{vol}_X(xH-F), \quad \text{and} \quad \lim_{n\to\infty}\dfrac{\dim_k {\left(I^n \colon \mathfrak{m}_R^{\infty}\right)}_{\lfloor xn \rfloor}}{n^{d-1}/(d-1)!} = \mathrm{vol}_X(xH-E).
\end{align*}
\textit{Proof of Claim 4.} First note that $\mathcal{I}^n\colon \mathcal{J}^{\infty}$ is the ideal sheaf associated to $I^n\colon J^{\infty}$ on $V$. By the Serre-Grothendieck correspondence \cite[Proposition 2.2]{Har67}, there is an exact sequence $$0 \to I^n \colon J^{\infty} \to \bigoplus_{m\in\mathbb{N}}\Gamma\big(V,(\mathcal{I}^n\colon \mathcal{J}^{\infty})\otimes\mathcal{O}_V(m)\big) \to H^1_{\mathfrak{m}_R}\big(I^n \colon J^{\infty}\big) \to 0$$ of graded $R$-modules. There exists an integer $m_0\geq 0$ (for instance, $m_0 = \mathrm{reg}(R)$) such that $\left(H^1_{\mathfrak{m}_R}(R)\right)_m = 0$ for all $m\geq m_0$. Also, observe that $H^0_{\mathfrak{m}_R}(R/(I^n\colon J^{\infty}))=0$. Hence, for all $m\geq m_0$ and $n\geq 0$, we have $${\big(I^n : J^{\infty}\big)}_m = \Gamma\big(V,(\mathcal{I}^n\colon \mathcal{J}^{\infty})\otimes\mathcal{O}_V(m)\big).$$ In particular, for $x>0$ and $n\geq \frac{m_0+1}{x}$, this gives $${\big(I^n\colon J^{\infty}\big)}_{\lfloor xn\rfloor} = \Gamma\big(V,(\mathcal{I}^n : \mathcal{J}^{\infty})\otimes\mathcal{O}_V(\lfloor xn\rfloor)\big).$$ Applying Claim 3 completes the proof for $x>0$. When $x=0$, the equalities are trivially satisfied. This finishes the proof of part $(i)$ of our theorem.

\textit{$(ii)$} From part $(i)$ we have the identification $f_{\mathbb{I}^{\mathrm{sat}}_J}(x) = d\cdot\mathrm{vol}_X(xH-F)$. Further, the support of the volume function is the big cone, which is the interior of the pseudoeffective cone, see \cite[Theorem 2.2.26]{Laz04a}. Consequently,
\begin{align*}
 \alpha_{\mathbb{I}^{\mathrm{sat}}_J} = \inf \left\{\, x \in \mathbb{R}_{\geq 0} \mid \mathrm{vol}_X(xH-F)>0 \,\right\} &= \inf \left\{\, x \in \mathbb{R}_{\geq 0} \mid xH-F\;\;\text{is big} \,\right\}\\
 &= \min \left\{\, x \in \mathbb{R}_{\geq 0} \mid xH-F\;\;\text{is pseudoeffective} \,\right\}.
\end{align*}

Fix a closed embedding $X\hookrightarrow \mathbb{P}^N_k$ for some $N>0$, and recall the notion of the degree of a Weil divisor, as given in \cite[Chapter II, Exercise 6.2]{Har77}. If the line bundle $\mathcal{O}_X(mH-nF)$ admits a nontrivial global section, then $\deg (mH-nF) = m\deg H - n\deg F \geq 0$, from which it follows $$\alpha_{\mathbb{I}^{\mathrm{sat}}_J}\geq \frac{\deg F}{\deg H}>0.$$

\textit{$(iii)$} The volume function $\mathrm{vol}_X$ is continuous on all of the N\'{e}ron-Severi space $ N^1(X)_{\mathbb{R}}$, see \cite[Corollary 2.2.45]{Laz04a}. Moreover, its restriction to the big cone $\mathrm{Big}(X)$ is strictly positive, log-concave \cite[Corollary 4.12]{LM09}, and continuously differentiable \cite[Corollary 4.27]{LM09}. Since $f_{\mathbb{I}^{\mathrm{sat}}_J}(x) = d\cdot\mathrm{vol}_X(xH-F)$, the desired properties follow immediately.

For all real numbers $x_1>x_2>\alpha_{\mathbb{I}^{\mathrm{sat}}_J}$, both divisors $x_2H-F$ and $(x_1-x_2)H$ are big. Applying the Brunn-Minkowski-type inequality for the volumes of big divisors \cite[Corollary 4.12]{LM09}, we get
\begin{align*}
 {\left(f_{\mathbb{I}^{\mathrm{sat}}_J}(x_1)\right)}^{\tfrac{1}{d-1}} = {\left(d\cdot\mathrm{vol}_X(x_1H-F)\right)}^{\tfrac{1}{d-1}} &\geq {\left(d\cdot\mathrm{vol}_X(x_2H-F)\right)}^{\tfrac{1}{d-1}} + {\left(d\cdot\mathrm{vol}_X\left((x_1-x_2)H\right)\right)}^{\tfrac{1}{d-1}}\\
 &= {\left(f_{\mathbb{I}^{\mathrm{sat}}_J}(x_2)\right)}^{\tfrac{1}{d-1}} + (x_1-x_2){(d\cdot e(R))}^{\tfrac{1}{d-1}},
\end{align*}
where $e(R) = H^{d-1}$ denotes the Hilbert-Samuel multiplicity of $R$. In particular, this inequality shows that $f_{\mathbb{I}^{\mathrm{sat}}_J}$ is strictly increasing on its support. The proof of our theorem is now complete.
\end{proof}

\begin{remark}\label{symb_geom_rem}
The conclusions of Theorem \ref{symb_geom} remain valid for the filtration $\mathbb{I}^{\mathrm{sat}}_J = \{I_n \colon J^{\infty}\}_{n\in\mathbb{N}}$, where $\mathbb{I} = \{I_n\}_{n\in\mathbb{N}}$ is a Noetherian filtration of graded ideals satisfying $1\leq \mathrm{height}\;I_1\leq d-1$. Indeed, since $\mathbb{I}$ is Noetherian, \cite[Theorem 2.1]{HHT07} shows that there exists an integer $c>0$ such that $$I_{cn} = (I_c)^n \quad \forall n\in\mathbb{N}.$$ By Theorem \ref{mainthm}, the density function $f_{\mathbb{I}^{\mathrm{sat}}_J}$ can be computed by passing to this subsequence. So, $$f_{\mathbb{I}^{\mathrm{sat}}_J}(x) = \lim_{n\to\infty}\dfrac{\dim_k\left(I_n \colon J^{\infty}\right)_{\lfloor xn\rfloor}}{n^{d-1}/d!} = \dfrac{1}{c^{d-1}}\lim_{n\to\infty}\dfrac{\dim_k\left((I_c)^n \colon J^{\infty}\right)_{\lfloor cxn\rfloor}}{n^{d-1}/d!}.$$ The desired statement then follows from applying Theorem \ref{symb_geom}.
\end{remark}

\subsection{Equality of density functions}

This subsection is inspired from \cite{DRT26}. We recall the following important result from \cite{FKL16}.

\begin{proposition}\cite[Theorem B]{FKL16}\label{equality_volumes}
 Suppose that $Y$ is a normal projective variety over a perfect field, and $D_1,D_2$ are big $\mathbb{R}$-Weil divisors on $Y$ with $D_1\leq D_2$. Then the following are equivalent: $$\mathrm{vol}_Y(D_1) = \mathrm{vol}_Y(D_2) \iff \Gamma\left(Y,\mathcal{O}_Y\left(\lfloor nD_1\rfloor\right)\right) = \Gamma\left(Y,\mathcal{O}_Y\left(\lfloor nD_2\rfloor\right)\right)\;\; \forall n\in\mathbb{N}.$$
\end{proposition}

\begin{theorem}\label{sat=symb}
 Consider the Setup \ref{setup4}. Then the following statements are equivalent:
 \begin{enumerate}[$(i)$]
  \item $\displaystyle f_{\mathbb{I}^{\mathrm{sat}}_{\mathfrak{m}_R}}(x) = f_{\mathbb{I}^{\mathrm{sat}}_J}(x)$ for all $x\in\mathbb{R}_{\geq 0}$.
  \item $\displaystyle\int_0^{\infty}e^{-x}f_{\mathbb{I}^{\mathrm{sat}}_{\mathfrak{m}_R}}(x) dx = \int_0^{\infty}e^{-x}f_{\mathbb{I}^{\mathrm{sat}}_J}(x) dx$.
  \item $\overline{I^n} : \mathfrak{m}_R^{\infty} = \overline{I^n} : J^{\infty}$ for all $n\in\mathbb{N}$.
 \end{enumerate}
\end{theorem}
\begin{proof}
$(ii) \implies (i)$. By Theorem \ref{symb_geom}, the density functions $f_{\mathbb{I}^{\mathrm{sat}}_{\mathfrak{m}_R}}$ and $f_{\mathbb{I}^{\mathrm{sat}}_{J}}$ are both continuous on $\mathbb{R}_{\geq 0}$. Moreover, $f_{\mathbb{I}^{\mathrm{sat}}_{J}}(x) \geq f_{\mathbb{I}^{\mathrm{sat}}_{\mathfrak{m}_R}}(x)$ for all $x\geq 0$. Hence the function
\[
g(x) = e^{-x}\bigl(f_{\mathbb{I}^{\mathrm{sat}}_{J}}(x) - f_{\mathbb{I}^{\mathrm{sat}}_{\mathfrak{m}_R}}(x)\bigr)
\] is continuous and nonnegative on $\mathbb{R}_{\ge 0}$. The hypothesis $\int_0^{\infty} g(x)\,dx = 0$
therefore forces $g(x)$ to be identically zero; see \cite[Exercise 6.2]{Rud76}. Consequently,
$f_{\mathbb{I}^{\mathrm{sat}}_{\mathfrak{m}_R}}(x) = f_{\mathbb{I}^{\mathrm{sat}}_{J}}(x)$ for all $x \geq 0$. The other direction $(i) \implies (ii)$ is trivial.

$(i)\implies (iii)$. Assume that $f_{\mathbb{I}^{\mathrm{sat}}_{\mathfrak{m}_R}}(x) = f_{\mathbb{I}^{\mathrm{sat}}_{J}}(x)$ for all $x\in\mathbb{R}_{\geq 0}$. Adopt the notations and constructions used in the proof of Theorem \ref{symb_geom}. It relates density functions with volumes, i.e., $$f_{\mathbb{I}^{\mathrm{sat}}_{\mathfrak{m}_R}}(x) = d\cdot\mathrm{vol}_X(xH-E), \quad \text{and} \quad f_{\mathbb{I}^{\mathrm{sat}}_{J}}(x) = d\cdot\mathrm{vol}_X(xH-F).$$ From our hypothesis, we have $$\mathrm{vol}_X(xH-E) = \mathrm{vol}_X(xH-F)\quad \forall x\in\mathbb{R}_{\geq 0}.$$ The Waldschmidt constant $\alpha_{\mathbb{I}^{\mathrm{sat}}_{J}}$ (resp. $\alpha_{\mathbb{I}^{\mathrm{sat}}_{\mathfrak{m}_R}}$) is characterized as the infimum of $x$ for which $xH-F$ (resp. $xH-E$) is big. Since the two volumes coincide, we obtain $\alpha_{\mathbb{I}^{\mathrm{sat}}_{J}} = \alpha_{\mathbb{I}^{\mathrm{sat}}_{\mathfrak{m}_R}}$ and denote this common value by $\alpha$. For every $x>\alpha$, both $xH-F$ and $xH-E$ are big  $\mathbb{R}$-Cartier divisors and $xH-E \leq xH-F$ (as $E\geq F$). So we may apply Proposition \ref{equality_volumes} to obtain $$\Gamma\left(X,\mathcal{O}_X\left(\lfloor xnH-nE\rfloor \right)\right) = \Gamma\left(X,\mathcal{O}_X\left(\lfloor xnH-nF\rfloor\right)\right)$$ for all reals $x>\alpha$ and integers $n\geq 1$. Taking $x = m/n$ with integers $m, n > 0$ and $m/n > \alpha$ yields $$\Gamma\left(X,\mathcal{O}_X\left(mH-nE\right)\right) = \Gamma\left(X,\mathcal{O}_X\left(mH-nF\right)\right),$$ and by the projection formula it translates to the equality $$\Gamma\left(V,\varphi_*\mathcal{O}_X(-nE)\otimes \mathcal{O}_V(m)\right) = \Gamma\left(V,\varphi_*\mathcal{O}_X(-nF)\otimes \mathcal{O}_V(m)\right).$$

Moreover, combining \eqref{reesval0}, \eqref{reesval1}, and \eqref{reesval3} yields for all $n\in \mathbb{N}$, $$\mathcal{O}_V \cap \varphi_*\mathcal{O}_X(-nE) = \overline{\mathcal{I}^n},\quad \text{and}\quad \mathcal{O}_V \cap \varphi_*\mathcal{O}_X(-nF) = \overline{\mathcal{I}^n}\colon \mathcal{J}^{\infty}.$$ Applying the Serre-Grothendieck correspondence \cite[Proposition 2.2]{Har67}, one deduces that for all integers $m\geq \mathrm{reg}(R)$ and $n\geq 0$,
\begin{align*}
\Gamma\left(V,\mathcal{O}_V(m)\right) \cap \Gamma\left(V,\varphi_*\mathcal{O}_X(-nE)\otimes \mathcal{O}_V(m)\right) &= \Gamma \big(V,\overline{\mathcal{I}^n}\otimes \mathcal{O}_V(m)\big) = {(\overline{I^n} \colon \mathfrak{m}_R^{\infty})}_m,\quad \text{and}\\
\Gamma\left(V,\mathcal{O}_V(m)\right) \cap \Gamma\left(V,\varphi_*\mathcal{O}_X(-nF)\otimes \mathcal{O}_V(m)\right) &= \Gamma \big(V,(\overline{\mathcal{I}^n}\colon \mathcal{J}^{\infty})\otimes \mathcal{O}_V(m)\big) = \big(\overline{I^n} : J^{\infty}\big)_m.
\end{align*}

By combining the above equalities, for all integers $m\geq \max\{\alpha n+1,\mathrm{reg}(R)\}$ and $n\geq 1$, we obtain
\[
 \big(\overline{I^n} \colon \mathfrak{m}_R^{\infty}\big)_m
  = \big(\overline{I^n} : J^{\infty}\big)_m
\]
Thus, for each fixed $n$, the two graded ideals coincide in all sufficiently large degrees $m$. Hence, the quotient $Q_n = \big(\overline{I^n} : J^{\infty}\big) \big/ \big(\overline{I^n} : \mathfrak{m}_R^{\infty}\big)$ has finite length as an $R$-module, since it is supported only at $\mathfrak{m}_R$. Both ideals in question are already $\mathfrak{m}_R$-saturated. So, $Q_n =0$ for every $n\geq 1$.

$(iii)\implies (i)$. This claim follows from the equalities obtained in part $(i)$ of Theorem \ref{symb_geom}, i.e., $$f_{\mathbb{I}^{\mathrm{sat}}_{\mathfrak{m}_R}}(x) = \lim\limits_{n\to\infty}\frac{\dim_k {(\overline{I^n} : \mathfrak{m}_R^{\infty})}_{\lfloor xn\rfloor}}{n^{d-1}/d!},\quad \text{and}\quad f_{\mathbb{I}^{\mathrm{sat}}_J}(x) = \lim\limits_{n\to\infty}\frac{\dim_k {(\overline{I^n} : J^{\infty})}_{\lfloor xn\rfloor}}{n^{d-1}/d!}.$$
\end{proof}

\begin{theorem}\label{adic=symb}
 Consider the Setup \ref{setup4}. Then the following statements are equivalent:
 \begin{enumerate}[$(i)$]
  \item $\displaystyle f_{\mathbb{I}}(x) = f_{\mathbb{I}^{\mathrm{sat}}_J}(x)$ for all $x\in\mathbb{R}_{\geq 0}$.
  \item $\displaystyle \int_0^{\infty}e^{-x}f_{\mathbb{I}}(x) dx = \int_0^{\infty}e^{-x}f_{\mathbb{I}^{\mathrm{sat}}_J}(x) dx$.
  \item $\displaystyle \overline{I^n} = \overline{I^n} \colon J^{\infty}$ for all $n\in\mathbb{N}$.
 \end{enumerate}
\end{theorem}
\begin{proof}
 $(ii)\implies (iii)$. Assume that $\int_0^{\infty}e^{-x}f_{\mathbb{I}}(x) dx = \int_0^{\infty}e^{-x}f_{\mathbb{I}^{\mathrm{sat}}_J}(x) dx$. Since $f_{\mathbb{I}}(x) \leq f_{\mathbb{I}^{\mathrm{sat}}_{\mathfrak{m}_R}}(x) \leq f_{\mathbb{I}^{\mathrm{sat}}_J}(x)$, the equality of the outer terms forces $\int_0^{\infty}e^{-x}f_{\mathbb{I}}(x) dx = \int_0^{\infty}e^{-x}f_{\mathbb{I}^{\mathrm{sat}}_{\mathfrak{m}_R}}(x)dx = \int_0^{\infty}e^{-x}f_{\mathbb{I}^{\mathrm{sat}}_J}(x) dx$. From Theorem \ref{sat=symb}, we obtain that
 \begin{equation}\label{equal_1}
  \overline{I^n} : \mathfrak{m}_R^{\infty} = \overline{I^n} : J^{\infty} \quad \forall n\in\mathbb{N}.
 \end{equation}
Moreover, $g(x)= e^{-x}\left(f_{\mathbb{I}^{\mathrm{sat}}_{\mathfrak{m}_R}}(x) - f_{\mathbb{I}}(x)\right)$ is a compactly supported nonnegative function that is continuous everywhere except possibly at $x=\alpha_{\mathbb{I}}$. Since $\int_0^{\infty}g(x)dx = 0$, it follows that $f_{\mathbb{I}}(x) = f_{\mathbb{I}^{\mathrm{sat}}_{\mathfrak{m}_R}}(x)$ for all $x\neq\alpha_{\mathbb{I}}$. On the other hand, the $\varepsilon$-multiplicity of $I$, which is defined by $\varepsilon(I) = \lim\limits_{n\to\infty}\frac{\lambda_R\left(H^0_{\mathfrak{m}_R}\left(R/I^n\right)\right)}{n^d/d!}$, admits an integral representation $$\varepsilon(I) = \int\limits_0^{\infty} \left(f_{\mathbb{I}^{\mathrm{sat}}_{\mathfrak{m}_R}}(x) - f_{\mathbb{I}}(x)\right)dx,$$ see \cite[Theorem 5.19]{DRT25}. Since $f_{\mathbb{I}}(x) = f_{\mathbb{I}^{\mathrm{sat}}_{\mathfrak{m}_R}}(x)$ almost everywhere, we get $\varepsilon(I)=0$. By combining \cite[Theorem 4.7]{KV10} and \cite[Theorem 5.4.6]{HS06}, $\varepsilon(I)=0$ is equivalent to the condition $\mathfrak{m}_R \not\in \mathrm{Ass}_R\left(R/\overline{I^n}\right)$ for all $n\in\mathbb{N}$, that is,
 \begin{equation}\label{equal_2}
  \overline{I^n} = \overline{I^n} \colon \mathfrak{m}_R^{\infty} \quad \forall n\in\mathbb{N}.
 \end{equation}
Combining \eqref{equal_1} and \eqref{equal_2} immediately yields $\overline{I^n} = \overline{I^n} \colon J^{\infty}$ for all $n\in\mathbb{N}$. The rest of the implications are clear.
\end{proof}

\begin{corollary}\label{sym_reg}
Consider the Setup \ref{setup4}. Further, assume that $R$ is regular and $I$ is radical. Let $\mathbb{I}^{\mathrm{symb}} = \{I^{(n)}\}_{n\in\mathbb{N}}$ be the filtration of symbolic powers of $I$. Then the following are equivalent:
 \begin{enumerate}[$(i)$]
  \item $\displaystyle f_{\mathbb{I}}(x) = f_{\mathbb{I}^{\mathrm{symb}}}(x)$ for all $x\in\mathbb{R}_{\geq 0}$.
  \item $\displaystyle \int_0^{\infty}e^{-x}f_{\mathbb{I}}(x) dx = \int_0^{\infty}e^{-x}f_{\mathbb{I}^{\mathrm{symb}}}(x) dx$.
  \item $\displaystyle \overline{I^n} = I^{(n)}$ for all $n\in\mathbb{N}$.
 \end{enumerate}
\end{corollary}
\begin{proof}
$(i) \iff (iii)$. Recall that $I^{(n)} = \bigcap_{P\in\mathrm{Min}(R/I)}(I^nR_P \cap R) = I^n \colon J^{\infty}$, where $J=\bigcap_{P\in \mathrm{Ass}^*(I)\setminus \mathrm{Min}(I)} P$ as described in Section \ref{gen_symb}. In view of Theorem \ref{adic=symb} it suffices to show that $I^n \colon J^{\infty} = \overline{I^n} \colon J^{\infty}$. Since localization commutes with integeral closure by \cite[Proposition 1.1.4]{HS06}, we have $\overline{I^n}R_P = \overline{I^n R_P}$. Because $I=\bigcap_{P\in\mathrm{Min}(R/I)}P$ and localization commutes with taking powers, it follows that $I^n R_P = P^nR_P$. As $R$ is regular, $R_P$ is a regular local ring; hence its maximal ideal $PR_P$ is normal. Thus, $$\overline{I^n}R_P = \overline{I^n R_P} = \overline{P^nR_P} = P^nR_P = I^nR_P$$ and this proves our claim.
\end{proof}

\subsection{Density functions for saturated filtrations in dimension two}

The motivation for the following result comes from \cite[Lemma 6.8]{DDRV25}.

\begin{theorem}\label{dimension_two}
 Let $R$ be as in Setup \ref{setup4} and suppose further that $\dim R =2$. Let $\mathbb{I} = \{I_n\}_{n\in\mathbb{N}}$ be any filtration of graded ideals such that $\mathrm{height}\;I_1 = 1$. Let $J\subseteq \mathfrak{m}_R$ be another graded ideal such that $\sqrt{J}\not\subset \sqrt{I_1}$. Consider the saturation filtration $\mathbb{I}^{\mathrm{sat}}_J = \{I_n \colon J^{\infty}\}_{n\in\mathbb{N}}$. Then its associated density function $f_{\mathbb{I}^{\mathrm{sat}}_J}$ has the following form:
 $$f_{\mathbb{I}^{\mathrm{sat}}_J}(x) = \begin{cases}
                                         0 & \text{if}\;\;0\leq x\leq \alpha_{\mathbb{I}^{\mathrm{sat}}_J},\\
                                         2\big(e(R)x-\alpha_{\mathbb{I}^{\mathrm{sat}}_J}\big) & \text{if}\;\;\alpha_{\mathbb{I}^{\mathrm{sat}}_J} \leq x<\infty.
                                        \end{cases}
$$
\end{theorem}
\begin{proof}
\emph{Step 1}: Assume that $J=\mathfrak{m}_R$. Let $V = \mathrm{Proj}\;R$ be the projective curve with a very ample invertible sheaf $\mathcal{O}_V(1)$. Let $\mathcal{I}_n$ be the ideal sheaf associated to the ideal $I_n$ on $V$. Let $\varphi\colon X \to V$ denote the normalization of $V$. Since $X$ is a nonsingular projective curve, we can write $$\mathcal{I}_n\mathcal{O}_X = \mathcal{O}_X(-E_n)$$ for some effective divisor $E_n$ on $X$. The filtration property ensures that $E_1 \leq E_2 \leq \cdots$, and $E_{n_1} + E_{n_2} \geq E_{n_1+n_2}$ for all $n_1,n_2\in\mathbb{N}$. In other words, $\{\deg E_n\}_{n\in\mathbb{N}}$ is a non-decreasing subadditive sequence of positive integers, and therefore the limit $\delta = \lim\limits_{n\to\infty}\frac{\deg E_n}{n}$ exists by Fekete's subadditive lemma. Let $H$ be the pullback to $X$ of a linear hyperplane section on $V$. As $\varphi$ is a finite morphism, $H$ is ample and \cite[Chapter III, Exercise 4.1]{Har77} yields
\begin{equation}\label{dim2_eq1}
 H^i\left(X,\mathcal{O}_X(mH-E_n)\right) = H^i\left(V,\varphi_*\mathcal{O}_X(-E_n)\otimes \mathcal{O}_V(m)\right)
\end{equation}
for all $i,n \in\mathbb{N}$ and $m\in\mathbb{Z}$. Then the Riemann-Roch formula \cite[Chapter IV.1]{Har77} shows that
\begin{equation}\label{dim2_eq2}
 \dim_k\Gamma\left(X,\mathcal{O}_X(mH-E_n)\right) = \begin{cases}
                                                     0 & \text{if}\;\;m<\frac{\deg E_n}{\deg H},\\
                                                     m\deg H - \deg E_n + 1-g & \text{if}\;\;m>\frac{\deg E_n}{\deg H} + \frac{2g-2}{\deg H},
                                                    \end{cases}
\end{equation}
where $g = \dim_k H^1(X,\mathcal{O}_X)$ is the genus of $X$.

Since $\varphi$ is also birational, there is an exact sequence of $\mathcal{O}_V$-modules, $$0 \to \mathcal{O}_V \to \varphi_*\mathcal{O}_X \to \mathcal{G}\to 0,$$ where $\dim \left(\mathrm{Supp}\;\mathcal{G}\right)\leq 0$. For each $n\in\mathbb{N}$, this induces exact sequences $$0 \to \overline{\mathcal{I}_n} = \mathcal{O}_{V} \cap \varphi_*\mathcal{O}_{X}(-E_n) \to \varphi_*\mathcal{O}_{X}(-E_n) \to \mathcal{G}_n \to 0,$$ where $\overline{\mathcal{I}_n}$ is the integral closure of the ideal sheaf $\mathcal{I}_n$, and $\mathcal{G}_n \subseteq \mathcal{G}$ is a coherent $\mathcal{O}_V$-module. Tensoring with $\mathcal{O}_V\left(\lfloor xn\rfloor\right)$ and taking global sections yields the left exact sequence
\begin{equation}\label{dim2_eq3}
 0 \to \Gamma\big(V,\overline{\mathcal{I}_n}\otimes\mathcal{O}_V\left(\lfloor xn\rfloor\right)\big) \to \Gamma\big(V,\varphi_*\mathcal{O}_X(-E_n) \otimes \mathcal{O}_V\left(\lfloor xn\rfloor\right)\big) \to \Gamma\big(V,\mathcal{G}_n\otimes \mathcal{O}_V\left(\lfloor xn\rfloor\right)\big).
\end{equation}
Since $\mathcal{G}$ is supported on a zero-dimensional closed subscheme, we have $$\lim_{n\to\infty}\dfrac{\dim_k \Gamma\big(V,\mathcal{G}_n\otimes \mathcal{O}_V\left(\lfloor xn\rfloor\right)\big)}{n} \leq \lim_{n\to\infty}\dfrac{\dim_k \Gamma\big(V,\mathcal{G}\otimes \mathcal{O}_V\left(\lfloor xn\rfloor\right)\big)}{n} = 0.$$ Using this fact together with \eqref{dim2_eq1}, \eqref{dim2_eq2}, and \eqref{dim2_eq3} gives
\begin{multline}\label{dim2_eq4}
 \lim_{n\to\infty}\dfrac{\dim_k \Gamma\big(V,\overline{\mathcal{I}_n}\otimes\mathcal{O}_V\left(\lfloor xn\rfloor\right)\big)}{n} = \lim_{n\to\infty}\dfrac{\dim_k \Gamma \big(V,\varphi_*\mathcal{O}_{X}(-E_n)\otimes\mathcal{O}_V\left(\lfloor xn\rfloor\right)\big)}{n}\\
 = \lim_{n\to\infty}\dfrac{\dim_k \Gamma \left(X,\mathcal{O}_X\left(\lfloor xn\rfloor H - E_n\right)\right)}{n}
 = \begin{cases}
    0 & \text{if}\;\;x\leq \frac{\delta}{e(R)},\\
    e(R)x-\delta & \text{if}\;\;x\geq \frac{\delta}{e(R)}.
   \end{cases}
\end{multline}

By the Serre-Grothendieck correspondence \cite[Proposition 2.2]{Har67}, there is an exact sequence $$0 \to \overline{I_n} \colon \mathfrak{m}_R^{\infty} \to \bigoplus_{m\in\mathbb{N}}\Gamma\big(V,\overline{\mathcal{I}_n}\otimes\mathcal{O}_V(m)\big) \to H^1_{\mathfrak{m}_R}\big(\overline{I_n} \colon \mathfrak{m}_R^{\infty}\big) \to 0$$ of graded $R$-modules. As $\left(H^1_{\mathfrak{m}_R}(R)\right)_m = 0$ for all $m\geq \mathrm{reg}(R)$, we conclude that for all integers $m\geq \mathrm{reg}(R)$ and $n\geq 0$, $$\dim_k{\big(\overline{I_n} \colon \mathfrak{m}_R^{\infty}\big)}_m = \dim_k\Gamma\big(V,\overline{\mathcal{I}_n}\otimes\mathcal{O}_V(m)\big).$$ Thus, Theorem \ref{density_integral} and \eqref{dim2_eq4} together yield $$\lim_{n\to\infty}\dfrac{\dim_k{\big(I_n \colon \mathfrak{m}_R^{\infty}\big)}_{\lfloor xn\rfloor}}{n} = \lim_{n\to\infty}\dfrac{\dim_k{\big(\overline{I_n} \colon \mathfrak{m}_R^{\infty}\big)}_{\lfloor xn\rfloor}}{n} = \begin{cases}
    0 & \text{if}\;\;x\leq \frac{\delta}{e(R)},\\
    e(R)x-\delta & \text{if}\;\;x\geq \frac{\delta}{e(R)}.
   \end{cases}$$

\emph{Step 2}. Let $J$ be as in the theorem. For every $n \geq 1$, one has $I_n \colon J^{\infty} = \left(I_n \colon J^{\infty}\right) \colon \mm_R^{\infty}$. Now apply Step 1 to the filtration $\mathbb{I}^{\mathrm{sat}}_J=\left\{I_n \colon J^{\infty}\right\}_{n\in\mathbb{N}}$.
\end{proof}

\subsection{Density functions for generalized symbolic powers in dimension three}

We begin with an elementary lemma whose proof is omitted.

\begin{lemma}\label{lemma_elementary}
Let $p_1(x), p_2(x) \in \mathbb{Q}[x]$ be distinct polynomials such that $1\leq \deg p_1(x), \deg p_2(x) \leq 2$. If $a\in \mathbb{R}$ satisfies $p_1(a) = p_2(a)$ and $p^{\prime}_1(a) = p^{\prime}_2(a)$, then $a\in \mathbb{Q}$.
\end{lemma}

The next theorem shows that in dimension three, the generalized symbolic density function is a piecewise polynomial, possibly with countably many pieces.

\begin{theorem}\label{dimension_three}
Adopt Setup \ref{setup4}. Assume further that $\dim R=3$ and that the base field $k$ has characteristic zero. Fix a compact interval $[\alpha,\beta]\subset\mathbb{R}$ satisfying
$\alpha_{\mathbb{I}^{\mathrm{sat}}_J}<\alpha<\beta$. Then the restriction of the generalized symbolic density function $f_{\mathbb{I}^{\mathrm{sat}}_J}$ to $[\alpha,\beta]$ is of exactly one of the following two forms:
\begin{enumerate}
\item[$(i)$] There exists a polynomial $p(x)\in\mathbb{Q}[x]$ with $1\leq\deg p(x)\leq 2$ such that $$f_{\mathbb{I}^{\mathrm{sat}}_J}(x)=p(x)$$ for all $x\in [\alpha,\beta]$.
\item[$(ii)$] There exist finitely many rational numbers $\eta_1,\ldots,\eta_r$ with $$\eta_0:=\alpha<\eta_1<\cdots<\eta_r<\eta_{r+1}:=\beta,$$ and distinct polynomials $p_0(x),\ldots,p_r(x)\in\mathbb{Q}[x]$, each satisfying $1\leq\deg p_i(x)\leq 2$, such that $$p_{i-1}(\eta_i)=p_i(\eta_i) \quad\text{and}\quad p^{\prime}_{i-1}(\eta_i)=p^{\prime}_i(\eta_i),$$ for all $i=1,\ldots,r$, and $$f_{\mathbb{I}^{\mathrm{sat}}_J}(x)=p_{i}(x)$$ for all $x\in[\eta_{i},\eta_{i+1}]$ and $i=0,\ldots,r$.
\end{enumerate}
\end{theorem}
\begin{proof}
By Theorem \ref{symb_geom} and passing to a resolution of singularities, there exist a nonsingular projective surface $X$ with a projective birational morphism $$\varphi\colon X \longrightarrow V = \mathrm{Proj}\;R,$$ and an effective Cartier divisor $F$ on $X$ supported on the exceptional locus, such that for all $x\in\mathbb{R}_{\geq 0}$, $$f_{\mathbb{I}^{\mathrm{sat}}_J}(x) = 3\cdot\mathrm{vol}_X(xH-F),$$ where $H$ denotes the pullback to $X$ of a hyperplane section on $V$. Set $D_x =xH-F$. Then $D_x$ is big for all $x>\alpha_{\mathbb{I}^{\mathrm{sat}}_J}$.

\emph{Case 1:} Assume that $D_x$ is non-nef for all $\alpha\leq x<\beta$. By \cite[Theorem 1.1]{BKS04}, the big non-nef divisor $D_x$ admits a Zariski decomposition $$D_x = P_{D_x} + N_{D_x},$$ such that
\begin{itemize}
 \item $P_{D_x}$ is a big and nef $\mathbb{R}$-divisor,
 \item $N_{D_x} = \sum_i N_{i,D_x}$ is a nonzero effective $\mathbb{R}$-divisor such that the intersection matrix $[(N_{i,D_x} \cdot N_{j,D_x})]$ of its irreducible components is negative definite,
 \item $(P_{D_x} \cdot N_{i,D_x}) =0$ for each irreducible component $N_{i,D_x}$ of $N_{D_x}$.
\end{itemize}
In particular, $$\mathrm{vol}_X(D_x)  = P_{D_x}^2 = (D_x^2) - (N_{D_x}^2) = (xH-F)^2 - (N_{D_x}^2).$$ Moreover, if $x\in\mathbb{Q}$ then $P_{D_x}$ and $N_{D_x}$ are $\mathbb{Q}$-divisors. By \cite[Theorem 1.2]{BKS04}, there is a locally finite decomposition of the big cone $\mathrm{Big}(X)$ into rational locally polyhedral subcones (called \emph{Zariski chambers}) such that in each subcone the support of the negative part of the Zariski decomposition of the divisors is constant, and the volume function is given by a single quadratic polynomial.

The line segment $\{xH-F \mid \alpha \leq x\leq \beta\} \subset \mathrm{Big}(X)$ will intersect only finitely many Zariski chambers of $\mathrm{Big}(X)$. If it intersects at least two chambers, there exist finitely many real numbers $$\eta_0:=\alpha<\eta_1<\cdots<\eta_r<\eta_{r+1}:=\beta,$$ such that for every $i=0,\ldots,r$, the subsegment $\{xH-F \mid \eta_i < x<\eta_{i+1}\}$ is contained inside a single Zariski chamber. Thus, for all $x\in(\eta_{i},\eta_{i+1})$, we can write $$N_{D_x} = \sum_{j=1}^{s_i}a_{i,j}(x)E_{i,j},$$ where $E_{i,1},\ldots,E_{i,s_i}$ are fixed prime divisors, and the coefficients $a_{i,1}(x),\ldots,a_{i,s_i}(x)$ are determined by the linear system of equations
$$\left(xH-F - \sum_{j=1}^{s_i}a_{i,j}(x)E_{i,j}\right) \cdot E_{i,j} = 0 \quad \forall j=1,\ldots,s_i.$$

Since the intersection matrix $[(E_{i,j}\cdot E_{i,j^{\prime}})]_{1\leq j,j^{\prime}\leq s_i}$ is negative definite with integer entries, it is invertible over $\mathbb{Q}$. As $(H\cdot E_{i,j})$ and $(F\cdot E_{i,j})$ are also integers, solving this linear system of equations shows that each $a_{i,j}(x)$ is an affine linear function with rational coefficients. Hence, $$\mathrm{vol}_X(xH-F) = (D_x^2) - (N_{D_x}^2) = (xH-F)^2 - \left(\sum_{j=1}^{s_i}a_{i,j}(x)E_{i,j}\right)^2.$$ Thus, on the open interval $(\eta_{i},\eta_{i+1})$, one has $f_{\mathbb{I}^{\mathrm{sat}}_J}(x)=p_{i}(x)$, where $p_i(x)\in\mathbb{Q}[x]$ is a nonzero polynomial of degree at most two. Moreover, by the differentiability of the volume function \cite[Corollary 4.27]{LM09}, adjacent polynomials and their derivatives agree on common endpoints, i.e., $$p_{i-1}(\eta_i)=p_i(\eta_i) \quad\text{and}\quad p^{\prime}_{i-1}(\eta_i)=p^{\prime}_i(\eta_i) \quad \forall i=1,\ldots,r.$$ It follows from Lemma \ref{lemma_elementary} that $\eta_1,\ldots,\eta_r \in \mathbb{Q}$.

Otherwise, if $\{xH-F \mid \alpha< x<\beta\}$ lies in a single Zariski chamber, the above arguments show that there exists a polynomial $p(x)\in\mathbb{Q}[x]$ with $1\leq\deg p(x)\leq 2$ such that for all $x\in [\alpha,\beta]$, $$f_{\mathbb{I}^{\mathrm{sat}}_J}(x)=p(x),$$

\emph{Case 2:} Assume that $D_x$ is nef for some $\alpha\leq x<\beta$. Set $x_0 = \min\{x \mid xH-F \;\text{is nef and}\;\alpha\leq x\leq \beta\}$. Then for all $x\geq x_0$, we have $$\mathrm{vol}_X(xH-F) = (xH-F)^2.$$ On the interval $[\alpha,x_0]$, we may proceed as in Case 1.
\end{proof}

One can refer to \cite[Section 7]{DRT25} for some examples.

\section{Mixed multiplicities of filtrations of equigenerated ideals}\label{section 8}

\subsection{Mixed multiplicities of filtrations of ideals}

Let $(R,\mathfrak{m}_R)$ be a Noetherian local ring of dimension $d>0$. Let $L$ be an $\mathfrak{m}_R$-primary ideal and $I$ be an ideal of positive height in $R$. It follows from \cite{Bha57} that for all $m,n\gg 0$, the Hilbert function $\lambda_R \left(L^{m}I^{n}/L^{m+1}I^{n}\right)$ agrees with a polynomial in $m,n$ of total degree $d-1$. Moreover, it can be written as $$\sum_{i=0}^{d-1} \dfrac{e_i\left(L \vert I\right)}{(d-1-i)!i!}m^{d-1-i}n^i + \text{lower degree terms},$$ where the coefficients $e_i\left(L \vert I\right) \in \mathbb{N}$. In \cite{KV89} and \cite{Tru01}, the numbers $e_i\left(L \vert I\right)$ are called the \emph{mixed multiplicities} of $L$ and $I$. Alternatively, we may write
$$\sum_{i=0}^{d-1} \dfrac{e_i\left(L \vert I\right)}{(d-i)!i!}m^{d-i}n^i = \lim_{t\to\infty}\dfrac{\lambda_R\left(I^{nt}/L^{mt}I^{nt}\right)}{t^d}$$
for all $m,n\in\mathbb{N}$. Under mild hypotheses on the ring $R$, this definition can be generalized to filtrations of ideals, see \cite{CSS19} and \cite{CM22}.

\begin{theorem}\cite[Theorem 3.11]{CM22}\label{CM}
Further assume that $R$ is analytically irreducible. Let $\mathbb{L}=\{L_{n}\}_{n\in\mathbb{N}}$ be a filtration of $\mathfrak{m}_R$-primary ideals and $\mathbb{I}=\{I_{n}\}_{n\in\mathbb{N}}$ be a filtration of nonzero ideals in $R$. Then there exists a homogeneous polynomial $H_{\mathbb{L}\vert \mathbb{I}}(m,n)$ of total degree $d$ with nonnegative real coefficients such that $$H_{\mathbb{L}\vert \mathbb{I}}(m,n) = \lim_{t\to\infty}\dfrac{\lambda_R\left(I_{nt}/L_{mt}I_{nt}\right)}{t^d}$$ for all $m,n\in\mathbb{N}$. Additionally, the polynomial $H_{\mathbb{L}\vert \mathbb{I}}(m,n)$ has no term of the form $en^d$ with $e\neq 0$.
\end{theorem}

In view of Theorem \ref{CM} we may write $$H_{\mathbb{L}\vert \mathbb{I}}(m,n) = \sum_{i=0}^{d-1} \dfrac{e_i\left(\mathbb{L} \vert \mathbb{I}\right)}{(d-i)!i!}m^{d-i}n^i,$$ where $e_i\left(\mathbb{L} \vert \mathbb{I}\right) \in \mathbb{R}_{\geq 0}$ and will be called the \emph{mixed multiplicities} of $\mathbb{L}$ and $\mathbb{I}$. The following is a version of the ``Volume = Multiplicity formula''.

\begin{theorem}\cite[Corollary 3.13]{CM22}\label{vol=mult}
 Let the assumptions be as in Theorem \ref{CM}. Then $$e_i\left(\mathbb{L} \vert \mathbb{I}\right) = \lim_{p\to\infty}\dfrac{e_i\left(L_p \vert I_p\right)}{p^d}.$$
\end{theorem}

Henceforth, we restrict to the following graded setup:

\subsection{Setup}\label{setup5} Let $R$ be as in Setup \ref{setup}. Let $I\subseteq R$ be a graded ideal generated in equal degrees $d_I$ such that $1\leq \mathrm{height}\;I\leq d-1$.

\subsection{Intersection numbers and mixed multiplicities of equigenerated ideals}

We recall a construction from Theorem \ref{symb_geom}. Adopt Setup \ref{setup5}. Let $V = \mathrm{Proj}\;R$ with a very ample invertible sheaf $\mathcal{O}_V(1)$, and $\mathcal{I}$ be the ideal sheaf associated to $I$ on $V$. There is an integer $m_0\geq 0$ such that ${\left(H^1_{\mathfrak{m}_R}(R)\right)}_m = 0$ for all $m\geq m_0$. From the Serre-Grothendieck correspondence one has $$\Gamma\left(V, \mathcal{I}^n\otimes \mathcal{O}_V(m)\right) = {\left(I^n : \mathfrak{m}_R^{\infty}\right)}_m$$ for all integers $m\geq m_0$ and $n\geq 0$. Let $$\pi \colon X = \mathrm{Bl}_{\mathcal{I}}(V) \longrightarrow V$$ be the blow-up of $V$ along $\mathcal{I}$. Then $\mathcal{I}$ becomes locally principal on $X$, i.e., there is an effective Cartier divisor $E_I$ on $X$ such that $\mathcal{I}\mathcal{O}_X = \mathcal{O}_{X}(-E_I)$. Let $H$ be the pullback to $X$ of a hyperplane section on $V$. By \cite[Lemma 3.3]{CEL01} or \cite[Lemma 5.4.24]{Laz04a}, there exists an integer $n_0\geq 0$ such that for all integers $n\geq n_0$, $$\pi_*\mathcal{O}_X(-nE_I) = \mathcal{I}^n.$$ We conclude using the projection formula that for all integers $m\geq m_0$ and $n\geq n_0$,
\begin{equation}\label{eq_imp_1}
\Gamma\left(X,\mathcal{O}_X(mH-nE_I)\right) = \Gamma\left(V, \mathcal{I}^n\otimes \mathcal{O}_V(m)\right) = {\left(I^n : \mathfrak{m}_R^{\infty}\right)}_m.
\end{equation}

The next result (see also \cite[Proposition 5.23]{DRT25}) identifies the mixed multiplicities $e_i\left(\mathfrak{m}_R \vert I\right)$ with intersection numbers of nef divisors on $X$. We give a direct proof.

\begin{proposition}\label{prop_intersect}
Let the notations be as above. Then the following statements are true:
\begin{enumerate}
 \item[$(i)$] $H$ is a big and nef divisor on $X$.
 \item[$(ii)$] $d_IH-E_I$ is a nef divisor on $X$. Moreover, it is big if and only if $I$ has maximal analytic spread.
 \item[$(iii)$] For each $i=0,\ldots,d-1$, one has $e_i\left(\mathfrak{m}_R \vert I\right) = \left((d_IH-E_I)^{i}\cdot H^{d-1-i}\right)$.
\end{enumerate}
\end{proposition}
\begin{proof}
Note that $\mathcal{O}_V(1)$ is globally generated on $V$, so $\pi^* \mathcal{O}_V(1) = \mathcal{O}_X(H)$ is globally generated on $X$. Moreover, $\mathrm{vol}_X(H) = (H^{d-1}) = {\mathcal{O}_V(1)}^{d-1} = e(R)>0$. Hence, $H$ is big and nef.

As $I$ is generated in degree $d_I$, the coherent sheaf $\mathcal{I}\otimes \mathcal{O}_V(d_I)$ is globally generated on $V$. Consequently, its pullback $\pi^*\left(\mathcal{I}\otimes \mathcal{O}_V(d_I)\right) = \mathcal{O}_X(d_IH-E_I)$ is globally generated on $X$, and hence $d_IH-E_I$ nef.

From the definition of mixed multiplicities and using a grading argument \cite[Lemma 5.2]{DDRV25}, it follows that for all integers $m\geq 1$ and $n\geq 1$, one has $$\sum_{i=0}^{d-1}\binom{d-1}{i} e_i(\mathfrak{m}_R\vert I) m^{d-1-i}n^i = \lim_{t\to\infty}\dfrac{\lambda_R\left(\mathfrak{m}_R^{mt}I^{nt}/\mathfrak{m}_R^{mt+1}I^{nt}\right)}{t^{d-1}/(d-1)!} = \lim_{t\to\infty}\dfrac{\dim_k \left(I^{nt}\right)_{d_Int+mt}}{t^{d-1}/(d-1)!}.$$

It follows from \cite{CHT99}, \cite{Kod00}, and \cite{TW05}, that there exists an integer $u_0\geq 0$ such that $\mathrm{reg}(I^v) \leq d_Iv+u_0$ for all integers $v\geq 0$. Therefore, $${\left(I^v : \mathfrak{m}_R^{\infty}\right)}_u = {\left(I^v\right)}_u$$ for all integers $u\geq d_Iv+u_0$ and $v\geq 0$. From the above identities and \eqref{eq_imp_1}, we obtain
\begin{align*}
 \lim_{t\to\infty}\dfrac{\dim_k \left(I^{nt}\right)_{d_Int+mt}}{t^{d-1}/(d-1)!} &= \lim_{t\to\infty}\dfrac{\dim_k \left(I^{nt} : \mathfrak{m}_R^{\infty}\right)_{d_Int+mt}}{t^{d-1}/(d-1)!}\\
 &= \lim_{t\to\infty}\dfrac{\dim_k\Gamma\left(X,\mathcal{O}_X\left(t\left(n(d_IH-E_I) + mH\right)\right)\right)}{t^{d-1}/(d-1)!}\\
 &= \left(mH+n(d_IH-E_I)\right)^{d-1}\\
 &= \sum_{i=0}^{d-1}\binom{d-1}{i} \left((d_IH-E_I)^i\cdot H^{d-1-i}\right)m^{d-1-i}n^i.
\end{align*}

Here, we have used the fact that the volume of a nef divisor equals its top self-intersection number, see \cite[Corollary 1.4.41]{Laz04a}. The desired expression of $e_i(\mathfrak{m}_R\vert I)$ follows from comparing the coefficients. In particular, $e_{d-1}\left(\mathfrak{m}_R \vert I\right) = \left(d_IH-E_I\right)^{d-1} = \mathrm{vol}_X\left(d_IH-E_I\right)$. Thus, $d_IH-E_I$ is big if and only if $e_{d-1}\left(\mathfrak{m}_R \vert I\right)>0$ if and only if $I$ has maximal analytic spread, see \cite[Corollary 3.6 and Corollary 3.7]{Tru01}.
\end{proof}

\begin{remark}\label{rem_intersect}
Proposition \ref{prop_intersect} admits the following generalization. Let $X^{\prime}$ be a projective variety equipped with a projective birational morphism $\varphi \colon X^{\prime} \to V$ such that $\mathcal{I}\mathcal{O}_{X^{\prime}}$ is locally principal. Write $\mathcal{I}\mathcal{O}_{X^{\prime}} = \mathcal{O}_{X^{\prime}}(-E_I^{\prime})$ for some effective Cartier divisor $E_I^{\prime}$ on $X^{\prime}$. Let $H^{\prime}$ denote the pullback to $X^{\prime}$ of a hyperplane section on $V$. By the universal property of blow up \cite[Proposition 7.14]{Har77}, there exists a unique morphism $\theta \colon X^{\prime} \to X$ such that $\varphi = \pi \circ \theta$. Hence, by \cite[Chapter I, Proposition 6]{Kle66}, for every integer $0\leq j\leq d-1$, we have $$\big(H^{j}\cdot E_I^{d-1-j}\big) = \big((\theta^*H)^{j}\cdot (\theta^*E_I)^{d-1-j}\big) = \big({H^\prime}^{j}\cdot {E_I^{\prime}}^{d-1-j}\big).$$ In particular, for every integer $0\leq i\leq d-1$, $$e_i\left(\mathfrak{m}_R \vert I\right) = \big((d_IH^{\prime}-E_I^{\prime})^{i}\cdot {H^{\prime}}^{d-1-i}\big).$$
\end{remark}

\subsection{Inequalities involving intersection numbers of nef divisors}

The following inequalities were established in \cite[Corollary 1.6.3 and Example 1.6.4]{Laz04a}, where the base field is assumed to be algebraically closed (also see \cite[Lemma 33]{Huh12}). These inequalities are in fact true over an arbitrary field, see \cite[Corollary 6.3]{Cut15}.

\begin{theorem}[Inequalities of Khovanskii and Teissier]\label{Khovanskii-Teissier}
Let $Y$ be an $n$-dimensional integral projective variety over a field $k$. Let $\alpha, \beta \in {N^1\left(Y\right)}_{\mathbb{R}}$ be nef classes on $Y$. Put $s_i = \left(\alpha^i\cdot \beta^{n-i}\right)$. Then for all $i=1,\ldots,n-1$,
\begin{equation}\label{ineq1}
 s_i^2 \geq s_{i-1}s_{i+1}.
\end{equation}
Moreover,
\begin{equation}\label{ineq3}
 {\left(\left(\alpha + \beta\right)^{n}\right)}^{\tfrac{1}{n}} \geq {\left(\alpha^{n}\right)}^{\tfrac{1}{n}} + {\left(\beta^{n}\right)}^{\tfrac{1}{n}}.
\end{equation}
\end{theorem}

The following version of Diskant inequality appears in \cite[Theorem F]{BFJ09}, where the base field is assumed to be algebraically closed of characteristic zero. Its extension to arbitrary fields is given in \cite{Cut15}.

\begin{theorem}[Diskant's inequality]\label{diskant}
Let $Y$ be an $n$-dimensional integral projective variety over a field $k$. Let $\alpha, \beta \in {N^1\left(Y\right)}_{\mathbb{R}}$ be big and nef classes on $Y$. Put $s_i = \left(\alpha^i\cdot \beta^{n-i}\right)$. Let $r(\alpha,\beta)$ be the largest real number $r$ such $\alpha - r\beta$ is pseudo-effective. Then $$\left(\frac{s_{n-1}}{s_0}\right)^{\tfrac{1}{n-1}} - \left(\left(\frac{s_{n-1}}{s_0}\right)^{\tfrac{n}{n-1}} - \frac{s_{n}}{s_0}\right)^{\tfrac{1}{n}} \leq r(\alpha,\beta) \leq \frac{s_n}{s_{n-1}}.$$
\end{theorem}

The following result, known as the reverse Khovanskii-Teissier inequality, gives an inequality involving three nef classes. It was first proved in \cite{LX17} on compact K\"{a}hler manifolds. The algebraic version below appears in \cite[Theorem 1.1]{JL23}, where the base field was assumed to be algebraically closed. The result holds over an arbitrary field, as noted in \cite[Remark 1.14]{GHMSW25}. A purely combinatorial proof is given in \cite{HX24}.

\begin{theorem}[reverse Khovanskii-Teissier inequality]\label{reverse}
Let $Y$ be an $n$-dimensional integral projective variety over a field $k$. Let $\alpha, \beta, \gamma \in {N^1\left(Y\right)}_{\mathbb{R}}$ be nef classes on $Y$. Then
$$\binom{n}{i} \left(\beta^i\cdot \alpha^{n-i}\right) \left(\alpha^i\cdot \gamma^{n-i}\right) \geq \left(\alpha^n\right) \left(\beta^i\cdot \gamma^{n-i}\right)$$ for all $i=1,\ldots,n-1$.
\end{theorem}

\subsection{Inequalities involving mixed multiplicities of equigenerated ideals}

The following corollary recovers a result of Huh \cite[Corollary 22]{Huh12}, originally asked by Trung and Verma \cite[Question 2.7]{TV07}. For equigenerated monomial ideals, the log-concavity \eqref{log-concave} follows from \cite[Theorem 2.4]{TV07} and the Alexandrov-Fenchel inequality for mixed volumes.

\begin{corollary}[Huh]\label{cor_huh}
Let $R$ and $I$ be as in Setup \ref{setup5}. Then
\begin{equation}\label{log-concave}
 {\left(e_i\left(\mathfrak{m}_R \vert I\right)\right)}^2 \geq e_{i-1}\left(\mathfrak{m}_R \vert I\right) e_{i+1}\left(\mathfrak{m}_R \vert I\right)
\end{equation}
for all $i=1,\ldots,d-2$.
\end{corollary}
\begin{proof}
By Proposition \ref{prop_intersect}, for every $i=0,\ldots,d-1$, $$e_i\left(\mathfrak{m}_R \vert I\right) = \left((d_IH-E_I)^{i}\cdot H^{d-1-i}\right),$$ where $d_IH-E_I$ and $H$ are nef divisors on $X$. Applying Theorem \ref{Khovanskii-Teissier} with $Y=X$, $\alpha = d_IH-E_I$ and $\beta = H$, yields \eqref{log-concave}.
\end{proof}

\begin{corollary}\label{cor_rev}
Let $R$ and $I$ be as in Setup \ref{setup}. Let $J\subseteq R$ be another graded ideal generated in equal degrees $d_J$ such that $1\leq \mathrm{height}\;J\leq d-1$. Then
\begin{equation}\label{ineq0}
 {\left(e_{d-1}\left(\mathfrak{m}_R \vert IJ\right)\right)}^{\tfrac{1}{d-1}} \geq {\left(e_{d-1}\left(\mathfrak{m}_R \vert I\right)\right)}^{\tfrac{1}{d-1}} + {\left(e_{d-1}\left(\mathfrak{m}_R \vert J\right)\right)}^{\tfrac{1}{d-1}},
\end{equation}
and
\begin{equation}\label{Khov_Teiss}
 \sum_{i=0}^{d-1} {\binom{d-1}{i}}^2 e_{i}\left(\mathfrak{m}_R \vert I\right) e_{d-1-i}\left(\mathfrak{m}_R \vert J\right) \geq e(R) e_{d-1}\left(\mathfrak{m}_R \vert IJ\right).
\end{equation}
\end{corollary}
\begin{proof}
Let $\pi\colon X \to V$ be the blow-up of $V$ along the product ideal sheaf $\mathcal{IJ}$, where $\mathcal{I}$ and $\mathcal{J}$ are the ideal sheaves associated to $I$ and $J$, respectively. Write $\mathcal{I}\mathcal{O}_X = \mathcal{O}_X(-E_I)$ and $\mathcal{J}\mathcal{O}_X = \mathcal{O}_X(-E_J)$, where $E_I$ and $E_J$ are effective Cartier divisors on $X$. Let $H$ denote the pullback to $X$ of a hyperplane section on $V$. Then by Proposition \ref{prop_intersect} and Remark \ref{rem_intersect}, for every $i=0,\ldots,d-1$,
\begin{multline*}
 e_{i}\left(\mathfrak{m}_R \vert I\right) = \left(\left(d_IH-E_I\right)^{i}\cdot H^{d-1-i}\right),\; e_{i}\left(\mathfrak{m}_R \vert J\right) = \left(\left(d_JH-E_J\right)^{i}\cdot H^{d-1-i}\right),\;\;\text{and}\\ e_{i}\left(\mathfrak{m}_R \vert IJ\right) = \left(\left((d_I+d_J)H-E_I-E_J\right)^{i}\cdot H^{d-1-i}\right),
\end{multline*}
where $d_IH-E_I$, $d_JH-E_J$, and $H$ are nef divisors on $X$. We obtain \eqref{ineq0} by applying the inequality \eqref{ineq3} with $Y=X$, $\alpha = d_IH-E_I$ and $\beta=d_JH-E_J$.

For each $i=0,\ldots,d-1$, we apply the reverse Khovanskii-Teissier inequality (Theorem \ref{reverse}) with $Y=X$, $\alpha = H$, $\beta = d_IH-E_I$, and $\gamma = d_JH-E_J$, and obtain
\begin{equation*}
{\binom{d-1}{i}} \left(\left(d_IH-E_I\right)^{i}\cdot H^{d-1-i}\right) \left(H^i\cdot\left(d_JH-E_J\right)^{d-1-i}\right) \geq \left(H^{d-1}\right) \left(\left(d_IH-E_I\right)^{i}\cdot\left(d_JH-E_J\right)^{d-1-i}\right).
\end{equation*}
Now observe that,
\begin{align*}
 \sum_{i=0}^{d-1} {\binom{d-1}{i}}^2 e_{i}\left(\mathfrak{m}_R \vert I\right) e_{d-1-i}\left(\mathfrak{m}_R \vert J\right) &= \sum_{i=0}^{d-1} {\binom{d-1}{i}}^2 \left(\left(d_IH-E_I\right)^{i}\cdot H^{d-1-i}\right) \left(H^i\cdot\left(d_JH-E_J\right)^{d-1-i}\right)\\
 &\geq \sum_{i=0}^{d-1} {\binom{d-1}{i}} \left(H^{d-1}\right) \left(\left(d_IH-E_I\right)^{i}\cdot\left(d_JH-E_J\right)^{d-1-i}\right)\\
 &= e(R) {\big((d_IH-E_I) + (d_JH-E_J)\big)}^{d-1}\\
 &= e(R) e_{d-1}\left(\mathfrak{m}_R \vert IJ\right).
\end{align*}
This proves \eqref{Khov_Teiss}.
\end{proof}

In general, epsilon multiplicities are difficult to compute, and can even be irrational; see \cite[Section 3]{CHST05}. Nevertheless, the next result gives a new lower bound for the epsilon multiplicity of an equigenerated ideal. This lower bound is strictly positive when $I$ has maximal analytic spread and can be computed using the Macaulay2 package {\sf MixedMultiplicity} developed in \cite{GMRV23}.

\begin{corollary}\label{cor_disk}
 Let $R$ and $I$ be as in Setup \ref{setup5}. Then
 \begin{equation}\label{disk_ineq}
 {\left(\dfrac{e_{d-2}\left(\mathfrak{m}_R \vert I\right)}{e(R)}\right)}^{\tfrac{1}{d-2}} - {\left({\left(\dfrac{e_{d-2}\left(\mathfrak{m}_R \vert I\right)}{e(R)}\right)}^{\tfrac{d-1}{d-2}} - \dfrac{e_{d-1}\left(\mathfrak{m}_R \vert I\right)}{e(R)}\right)}^{\tfrac{1}{d-1}} \leq d_I - \alpha_{\mathbb{I}^{\mathrm{sat}}_{\mathfrak{m}_R}} \leq \dfrac{e_{d-1}\left(\mathfrak{m}_R \vert I\right)}{e(R)},
\end{equation}
 where $\alpha_{\mathbb{I}^{\mathrm{sat}}_{\mathfrak{m}_R}}$ is the Waldschmidt constant for the saturation filtration $\mathbb{I}^{\mathrm{sat}}_{\mathfrak{m}_R} = \{I^n : \mathfrak{m}_R^{\infty}\}_{n\in\mathbb{N}}$. Moreover,
 $$\varepsilon(I):= \lim\limits_{n\to\infty}\dfrac{\lambda_R\left(H^0_{\mathfrak{m}_R}(R/I^n)\right)}{n^d/d!} \geq e(R) {\left({\left(\dfrac{e_{d-2}\left(\mathfrak{m}_R \vert I\right)}{e(R)}\right)}^{\tfrac{1}{d-2}} - {\left({\left(\dfrac{e_{d-2}\left(\mathfrak{m}_R \vert I\right)}{e(R)}\right)}^{\tfrac{d-1}{d-2}} - \dfrac{e_{d-1}\left(\mathfrak{m}_R \vert I\right)}{e(R)}\right)}^{\tfrac{1}{d-1}}\right)}^d.$$
\end{corollary}
\begin{proof}
If the analytic spread of $I$ is not maximal then $e_{d-1}(\mathfrak{m}_R \vert I)=0$, $d_I= \alpha_{\mathbb{I}^{\mathrm{sat}}_{\mathfrak{m}_R}}$, and $\varepsilon(I)=0$, in which case the inequalities are obvious. Henceforth, assume that $I$ has maximal analytic spread. By Proposition \ref{prop_intersect}, the divisors $d_IH-E_I$ and $H$ are big and nef on $X$, with $$e_0(\mathfrak{m}_R\vert I) = (H^{d-1}) = e(R),\; e_{d-2}\left(\mathfrak{m}_R \vert I\right) = \left((d_IH-E_I)^{d-2}\cdot H\right),\; \text{and}\; e_{d-1}\left(\mathfrak{m}_R \vert I\right) = \left(d_IH-E_I\right)^{d-1}.$$ Furthermore, by part $(ii)$ of Theorem \ref{symb_geom}, it follows that $$d_I-\alpha_{\mathbb{I}^{\mathrm{sat}}_{\mathfrak{m}_R}} = \max\{r\in\mathbb{R}\mid (d_IH-E_I) - rH\;\text{is pseudo-effective}\}.$$  Now to obtain \eqref{disk_ineq}, simply apply the Diskant's inequality (Theorem \ref{diskant}) with $Y=X$, $\alpha = d_IH-E_I$, and $\beta = H$.

Since $I$ is equigenerated, there exists a polynomial $P_I(x) \in \mathbb{Q}[x]$ of degree $d-1$ such that
$$f_{\mathbb{I}}(x) = 0\;\;\forall x<d_I\;\;\text{and}\;\; f_{\mathbb{I}}(x) = P_I(x) = f_{\mathbb{I}^{\mathrm{sat}}_{\mathfrak{m}_R}}(x)\;\; \forall x>d_I,$$ see \cite[Theorems 4.4 and 5.19]{DRT25}. Here, $f_{\mathbb{I}}(x)$ denotes the density function associated with the filtration $\mathbb{I} = \{I^n\}_{n\in\mathbb{N}}$. By \cite[Theorem 5.19]{DRT25} and part $(ii)$ of Theorem \ref{mainthm}, we have $$\varepsilon(I) = \int_{0}^{\infty} \big(f_{\mathbb{I}^{\mathrm{sat}}_{\mathfrak{m}_R}}(x) - f_{\mathbb{I}}(x)\big)dx = \int_{\alpha_{\mathbb{I}^{\mathrm{sat}}_{\mathfrak{m}_R}}}^{d_I} f_{\mathbb{I}^{\mathrm{sat}}_{\mathfrak{m}_R}}(x)dx \geq de(R) \int_{\alpha_{\mathbb{I}^{\mathrm{sat}}_{\mathfrak{m}_R}}}^{d_I} {\big(x-\alpha_{\mathbb{I}^{\mathrm{sat}}_{\mathfrak{m}_R}}\big)}^{d-1}dx = e(R){\big(d_I-\alpha_{\mathbb{I}^{\mathrm{sat}}_{\mathfrak{m}_R}}\big)}^{d}.$$ Finally, applying \eqref{disk_ineq} yields the second assertion.
\end{proof}

The following example shows that the lower bound in Corollary \ref{cor_disk} is sharp.

\begin{example}
Let $R$ be as in Setup \ref{setup5}. Suppose further that $R$ is regular. Let $f$ be a nonzero homogeneous element of $R$, and set $I=(f)\mathfrak{m}_R$. Then $I$ is generated in degree $\deg f+1$, and one can verify that $I^n:\mathfrak{m}_R^\infty=(f^n)$ for all $n\geq 1$. Consequently, $\alpha_{\mathbb{I}^{\mathrm{sat}}_{\mathfrak{m}_R}}=\deg f$, and $e_i(\mathfrak{m}_R\vert I)=1$ for all $i=1,\ldots,d-1$. In particular, all the inequalities in Corollaries \ref{cor_huh} and \ref{cor_disk} become equalities in this example.
\end{example}

In the following result, we use Theorem \ref{vol=mult} and extend Corollaries \ref{cor_huh}, \ref{cor_rev}, and \ref{cor_disk} to filtrations of equigenerated ideals.

\begin{proposition}\label{prop_filt}
 Let $R$ as in Setup \ref{setup5}. Let $\mathbb{I}=\{I_n\}_{n\in\mathbb{N}}$ be a filtration of graded ideals in $R$ such that each $I_n$ is generated in equal degrees $d_{I_n}$ and $1\leq \mathrm{height}\;I_1 \leq d-1$. Let $\mathbf{m}_R = \{\mathfrak{m}_R^n\}_{n\in\mathbb{N}}$. Then the following statements are true:
 \begin{enumerate}
  \item[$(i)$] For all $i=1,\ldots,d-2$, $${\left(e_i\left(\mathbf{m}_R \vert \mathbb{I}\right)\right)}^2 \geq e_{i-1}\left(\mathbf{m}_R \vert \mathbb{I}\right) e_{i+1}\left(\mathbf{m}_R \vert \mathbb{I}\right).$$
  \item[$(ii)$] Let $\mathbb{J}=\{J_n\}_{n\in\mathbb{N}}$ be another filtration of graded ideals in $R$ such that each $J_n$ is generated in equal degrees $d_{J_n}$ and $1\leq \mathrm{height}\;J_1 \leq d-1$. Then $${\left(e_{d-1}\left(\mathbf{m}_R \vert \mathbb{IJ}\right)\right)}^{\tfrac{1}{d-1}} \geq {\left(e_{d-1}\left(\mathbf{m}_R \vert \mathbb{I}\right)\right)}^{\tfrac{1}{d-1}} + {\left(e_{d-1}\left(\mathbf{m}_R \vert \mathbb{J}\right)\right)}^{\tfrac{1}{d-1}},$$
  and $$\sum_{i=0}^{d-1} {\binom{d-1}{i}}^2 e_{i}\left(\mathbf{m}_R \vert \mathbb{I}\right) e_{d-1-i}\left(\mathbf{m}_R \vert \mathbb{J}\right) \geq e(R) e_{d-1}\left(\mathbf{m}_R \vert \mathbb{IJ}\right).$$
  \item[$(iii)$] Let $\alpha_{\mathbb{I}}$ and $\alpha_{\mathbb{I}^{\mathrm{sat}}_{\mathfrak{m}_R}}$ be the Waldschmidt constants associated with the filtrations $\mathbb{I}=\{I_n\}_{n\in\mathbb{N}}$ and $\mathbb{I}^{\mathrm{sat}}_{\mathfrak{m}_R} = \{I_n : \mathfrak{m}_R^{\infty}\}_{n\in\mathbb{N}}$, respectively. Then
  $${\left(\dfrac{e_{d-2}\left(\mathbf{m}_R \vert \mathbb{I}\right)}{e(R)}\right)}^{\tfrac{1}{d-2}} - {\left({\left(\dfrac{e_{d-2}\left(\mathbf{m}_R \vert \mathbb{I}\right)}{e(R)}\right)}^{\tfrac{d-1}{d-2}} - \dfrac{e_{d-1}\left(\mathbf{m}_R \vert \mathbb{I}\right)}{e(R)}\right)}^{\tfrac{1}{d-1}} \leq \alpha_{\mathbb{I}} - \alpha_{\mathbb{I}^{\mathrm{sat}}_{\mathfrak{m}_R}} \leq \dfrac{e_{d-1}\left(\mathbf{m}_R \vert \mathbb{I}\right)}{e(R)}.$$
 \end{enumerate}
\end{proposition}
\begin{proof}
We shall only prove the second assertion of part $(ii)$ of our proposition. The remaining ones are variations of this argument. We apply the inequality \eqref{Khov_Teiss} of Corollary \ref{cor_rev} with $I=I_n$ and $J=J_n$, and obtain
$$\sum_{i=0}^{d-1} {\binom{d-1}{i}}^2 e_{i}\left(\mathfrak{m}_R \vert I_n\right) e_{d-1-i}\left(\mathfrak{m}_R \vert J_n\right) \geq e(R) e_{d-1}\left(\mathfrak{m}_R \vert I_nJ_n\right).$$
Now divide both sides by $n^{d-1}$ so that
$$\sum_{i=0}^{d-1} {\binom{d-1}{i}}^2 \cdot \dfrac{e_{i}\left(\mathfrak{m}_R \vert I_n\right)}{n^i}\cdot \dfrac{e_{d-1-i}\left(\mathfrak{m}_R \vert J_n\right)}{n^{d-1-i}} \geq e(R)\cdot \dfrac{e_{d-1}\left(\mathfrak{m}_R \vert I_nJ_n\right)}{n^{d-1}}.$$
The desired inequality then follows by letting $n\to\infty$ and using the ``volume = multiplicity''-type formulas
$$e_i(\mathbf{m}_R \vert \mathbb{I}) = \lim_{n\to\infty}\dfrac{e_i(\mathfrak{m}_R^n \vert I_n)}{n^d} = \lim_{n\to\infty}\dfrac{n^{d-i}e_i(\mathfrak{m}_R \vert I_n)}{n^d} = \lim_{n\to\infty}\dfrac{e_i(\mathfrak{m}_R \vert I_n)}{n^i}.$$
\end{proof}

\section{Acknowledgments}

The authors would like to thank T\`{a}i Huy H\`{a} and Sudeshna Roy for valuable discussions.

\bibliographystyle{alpha}
\bibliography{References}

@book {laz04a,
    AUTHOR = {Lazarsfeld, Robert},
     TITLE = {Positivity in algebraic geometry. {I}},
    SERIES = {Ergebnisse der Mathematik und ihrer Grenzgebiete. 3. Folge. A
              Series of Modern Surveys in Mathematics [Results in
              Mathematics and Related Areas. 3rd Series. A Series of Modern
              Surveys in Mathematics]},
    VOLUME = {48},
      NOTE = {Classical setting: line bundles and linear series},
 PUBLISHER = {Springer-Verlag, Berlin},
      YEAR = {2004},
     PAGES = {xviii+387},
      ISBN = {3-540-22533-1},
   MRCLASS = {14-02 (14C20)},
  MRNUMBER = {2095471},
MRREVIEWER = {Mihnea\ Popa},
       DOI = {10.1007/978-3-642-18808-4},
       URL = {https://doi.org/10.1007/978-3-642-18808-4},
}

@book {laz04b,
    AUTHOR = {Lazarsfeld, Robert},
     TITLE = {Positivity in algebraic geometry. {II}},
    SERIES = {Ergebnisse der Mathematik und ihrer Grenzgebiete. 3. Folge. A
              Series of Modern Surveys in Mathematics [Results in
              Mathematics and Related Areas. 3rd Series. A Series of Modern
              Surveys in Mathematics]},
    VOLUME = {49},
      NOTE = {Positivity for vector bundles, and multiplier ideals},
 PUBLISHER = {Springer-Verlag, Berlin},
      YEAR = {2004},
     PAGES = {xviii+385},
      ISBN = {3-540-22534-X},
   MRCLASS = {14-02 (14C20 14F05 14F17)},
  MRNUMBER = {2095472},
MRREVIEWER = {Mihnea\ Popa},
       DOI = {10.1007/978-3-642-18808-4},
       URL = {https://doi.org/10.1007/978-3-642-18808-4},
}

@book {ZS60,
    AUTHOR = {Zariski, Oscar and Samuel, Pierre},
     TITLE = {Commutative algebra. {V}ol. {II}},
    SERIES = {The University Series in Higher Mathematics},
 PUBLISHER = {D. Van Nostrand Co., Inc., Princeton, N.J.-Toronto-London-New
              York},
      YEAR = {1960},
     PAGES = {x+414},
   MRCLASS = {16.00 (14.00)},
  MRNUMBER = {120249},
MRREVIEWER = {H.\ T.\ Muhly},
}

@book {AM69,
    AUTHOR = {Atiyah, M. F. and Macdonald, I. G.},
     TITLE = {Introduction to commutative algebra},
 PUBLISHER = {Addison-Wesley Publishing Co., Reading, Mass.-London-Don
              Mills, Ont.},
      YEAR = {1969},
     PAGES = {ix+128},
   MRCLASS = {13.00},
  MRNUMBER = {242802},
MRREVIEWER = {Johnny\ A.\ Johnson},
}

@book {eisenbud,
    AUTHOR = {Eisenbud, David},
     TITLE = {Commutative Algebra},
    SERIES = {Graduate Texts in Mathematics},
    VOLUME = {150},
      NOTE = {With a View Toward Algebraic Geometry},
 PUBLISHER = {Springer-Verlag, New York},
      YEAR = {1995},
     PAGES = {xvi+785},
      ISBN = {0-387-94268-8; 0-387-94269-6},
   MRCLASS = {13-01 (14A05)},
  MRNUMBER = {1322960},
MRREVIEWER = {Matthew\ Miller},
       DOI = {10.1007/978-1-4612-5350-1},
       URL = {https://doi.org/10.1007/978-1-4612-5350-1},
}

@book {HS06,
    AUTHOR = {Huneke, Craig and Swanson, Irena},
     TITLE = {Integral closure of ideals, rings, and modules},
    SERIES = {London Mathematical Society Lecture Note Series},
    VOLUME = {336},
 PUBLISHER = {Cambridge University Press, Cambridge},
      YEAR = {2006},
     PAGES = {xiv+431},
      ISBN = {978-0-521-68860-4; 0-521-68860-4},
   MRCLASS = {13B22 (13A18 13A30 13A35 13H15 14A05)},
  MRNUMBER = {2266432},
MRREVIEWER = {Liam\ O'Carroll},
}

@article {Bro79,
    AUTHOR = {Brodmann, M.},
     TITLE = {Asymptotic stability of {${\rm Ass}(M/I\sp{n}M)$}},
   JOURNAL = {Proc. Amer. Math. Soc.},
  FJOURNAL = {Proceedings of the American Mathematical Society},
    VOLUME = {74},
      YEAR = {1979},
    NUMBER = {1},
     PAGES = {16--18},
      ISSN = {0002-9939,1088-6826},
   MRCLASS = {13E05},
  MRNUMBER = {521865},
MRREVIEWER = {John\ W.\ Petro},
       DOI = {10.2307/2042097},
       URL = {https://doi.org/10.2307/2042097},
}

@book {Rud76,
    AUTHOR = {Rudin, Walter},
     TITLE = {Principles of mathematical analysis},
    SERIES = {International Series in Pure and Applied Mathematics},
   EDITION = {Third},
 PUBLISHER = {McGraw-Hill Book Co., New York-Auckland-D\"usseldorf},
      YEAR = {1976},
     PAGES = {x+342},
   MRCLASS = {26-02},
  MRNUMBER = {385023},
}

@article {Nagata59,
    AUTHOR = {Nagata, Masayoshi},
     TITLE = {On the {$14$}-th problem of {H}ilbert},
   JOURNAL = {Amer. J. Math.},
  FJOURNAL = {American Journal of Mathematics},
    VOLUME = {81},
      YEAR = {1959},
     PAGES = {766--772},
      ISSN = {0002-9327,1080-6377},
   MRCLASS = {12.00 (14.00)},
  MRNUMBER = {105409},
MRREVIEWER = {P.\ Samuel},
       DOI = {10.2307/2372927},
       URL = {https://doi.org/10.2307/2372927},
}

@article{CSTZ24,
      title={F-signature functions of diagonal hypersurfaces}, 
      author={Alessio Caminata and Samuel Shideler and Kevin Tucker and Francesco Zerman},
      journal={arXiv preprint arXiv:2403.12863},
      year={2024}
}

@article{MM25,
      title={$h$-function, {H}ilbert-{K}unz density function and {F}robenius-{P}oincar\'e function}, 
      author={Cheng Meng and Alapan Mukhopadhyay},
      journal={arXiv preprint arXiv:2310.10270},
      year={2025}
}

@article{GHMSW25,
      title={Linear operators preserving volume polynomials}, 
      author={Lukas Grund and June Huh and Mateusz Micha\l{ek} and Hendrik S\"{u}ss and Botong Wang},
      journal={arXiv preprint arXiv:2506.22415},
      year={2025},
}

@article {Li17,
    AUTHOR = {Li, Chi},
     TITLE = {K-semistability is equivariant volume minimization},
   JOURNAL = {Duke Math. J.},
  FJOURNAL = {Duke Mathematical Journal},
    VOLUME = {166},
      YEAR = {2017},
    NUMBER = {16},
     PAGES = {3147--3218},
      ISSN = {0012-7094,1547-7398},
   MRCLASS = {14B05 (13A18 14J45 52A27 53C25 53C55)},
  MRNUMBER = {3715806},
MRREVIEWER = {Ruadha\'i\ Dervan},
       DOI = {10.1215/00127094-2017-0026},
       URL = {https://doi.org/10.1215/00127094-2017-0026},
}

@article {BC11,
    AUTHOR = {Boucksom, S\'ebastien and Chen, Huayi},
     TITLE = {Okounkov bodies of filtered linear series},
   JOURNAL = {Compos. Math.},
  FJOURNAL = {Compositio Mathematica},
    VOLUME = {147},
      YEAR = {2011},
    NUMBER = {4},
     PAGES = {1205--1229},
      ISSN = {0010-437X,1570-5846},
   MRCLASS = {14G25 (11G50 14C20)},
  MRNUMBER = {2822867},
MRREVIEWER = {Dennis\ Eriksson},
       DOI = {10.1112/S0010437X11005355},
       URL = {https://doi.org/10.1112/S0010437X11005355},
}

@article {Bou14,
    AUTHOR = {Boucksom, S\'ebastien},
     TITLE = {Corps d'{O}kounkov (d'apr\`es {O}kounkov,
              {L}azarsfeld-{M}usta\c t\v a{} et {K}aveh-{K}hovanskii)},
   JOURNAL = {Ast\'erisque},
  FJOURNAL = {Ast\'erisque},
    NUMBER = {361},
      YEAR = {2014},
     PAGES = {Exp. No. 1059, vii, 1--41},
      ISSN = {0303-1179,2492-5926},
      ISBN = {978-285629-785-8},
   MRCLASS = {14J60 (13A02 13A18 14C20 20M14 52B12)},
  MRNUMBER = {3289276},
MRREVIEWER = {Yanir\ A.\ Rubinstein},
}

@article {BFJ09,
    AUTHOR = {Boucksom, S\'ebastien and Favre, Charles and Jonsson, Mattias},
     TITLE = {Differentiability of volumes of divisors and a problem of
              {T}eissier},
   JOURNAL = {J. Algebraic Geom.},
  FJOURNAL = {Journal of Algebraic Geometry},
    VOLUME = {18},
      YEAR = {2009},
    NUMBER = {2},
     PAGES = {279--308},
      ISSN = {1056-3911,1534-7486},
   MRCLASS = {14C20 (14C17)},
  MRNUMBER = {2475816},
MRREVIEWER = {James\ McKernan},
       DOI = {10.1090/S1056-3911-08-00490-6},
       URL = {https://doi.org/10.1090/S1056-3911-08-00490-6},
}

@article {Cut21,
    AUTHOR = {Cutkosky, Steven Dale},
     TITLE = {The {M}inkowski equality of filtrations},
   JOURNAL = {Adv. Math.},
  FJOURNAL = {Advances in Mathematics},
    VOLUME = {388},
      YEAR = {2021},
     PAGES = {Paper No. 107869, 63},
      ISSN = {0001-8708,1090-2082},
   MRCLASS = {13H15 (13A18 13B22 14C17)},
  MRNUMBER = {4283761},
MRREVIEWER = {Antoni\ Rangachev},
       DOI = {10.1016/j.aim.2021.107869},
       URL = {https://doi.org/10.1016/j.aim.2021.107869},
}

@article {BLQ24,
    AUTHOR = {Blum, Harold and Liu, Yuchen and Qi, Lu},
     TITLE = {Convexity of multiplicities of filtrations on local rings},
   JOURNAL = {Compos. Math.},
  FJOURNAL = {Compositio Mathematica},
    VOLUME = {160},
      YEAR = {2024},
    NUMBER = {4},
     PAGES = {878--914},
      ISSN = {0010-437X,1570-5846},
   MRCLASS = {14B05 (13H15)},
  MRNUMBER = {4716588},
MRREVIEWER = {Guillaume\ Rond},
       DOI = {10.1112/S0010437X23007972},
       URL = {https://doi.org/10.1112/S0010437X23007972},
}

@article {Fuj18,
    AUTHOR = {Fujita, Kento},
     TITLE = {Optimal bounds for the volumes of {K}\"ahler-{E}instein {F}ano
              manifolds},
   JOURNAL = {Amer. J. Math.},
  FJOURNAL = {American Journal of Mathematics},
    VOLUME = {140},
      YEAR = {2018},
    NUMBER = {2},
     PAGES = {391--414},
      ISSN = {0002-9327,1080-6377},
   MRCLASS = {32Q25 (14J45 53C25 53C55)},
  MRNUMBER = {3783213},
MRREVIEWER = {Cristiano\ Spotti},
       DOI = {10.1353/ajm.2018.0009},
       URL = {https://doi.org/10.1353/ajm.2018.0009},
}

@article {BJ20,
    AUTHOR = {Blum, Harold and Jonsson, Mattias},
     TITLE = {Thresholds, valuations, and {K}-stability},
   JOURNAL = {Adv. Math.},
  FJOURNAL = {Advances in Mathematics},
    VOLUME = {365},
      YEAR = {2020},
     PAGES = {107062, 57},
      ISSN = {0001-8708,1090-2082},
   MRCLASS = {14C20 (14M25)},
  MRNUMBER = {4067358},
MRREVIEWER = {Chenyang\ Xu},
       DOI = {10.1016/j.aim.2020.107062},
       URL = {https://doi.org/10.1016/j.aim.2020.107062},
}

@article {CSS19,
    AUTHOR = {Cutkosky, Steven Dale and Sarkar, Parangama and Srinivasan,
              Hema},
     TITLE = {Mixed multiplicities of filtrations},
   JOURNAL = {Trans. Amer. Math. Soc.},
  FJOURNAL = {Transactions of the American Mathematical Society},
    VOLUME = {372},
      YEAR = {2019},
    NUMBER = {9},
     PAGES = {6183--6211},
      ISSN = {0002-9947,1088-6850},
   MRCLASS = {13H15 (14C17)},
  MRNUMBER = {4024518},
MRREVIEWER = {P.\ Schenzel},
       DOI = {10.1090/tran/7745},
       URL = {https://doi.org/10.1090/tran/7745},
}

@article {Cut15,
    AUTHOR = {Cutkosky, Steven Dale},
     TITLE = {Asymptotic multiplicities},
   JOURNAL = {J. Algebra},
  FJOURNAL = {Journal of Algebra},
    VOLUME = {442},
      YEAR = {2015},
     PAGES = {260--298},
      ISSN = {0021-8693,1090-266X},
   MRCLASS = {13A02},
  MRNUMBER = {3395062},
MRREVIEWER = {Aryampilly\ V.\ Jayanthan},
       DOI = {10.1016/j.jalgebra.2015.03.017},
       URL = {https://doi.org/10.1016/j.jalgebra.2015.03.017},
}

@article {CHST05,
    AUTHOR = {Cutkosky, Steven Dale and H\`a, Huy T\`ai and Srinivasan, Hema
              and Theodorescu, Emanoil},
     TITLE = {Asymptotic behavior of the length of local cohomology},
   JOURNAL = {Canad. J. Math.},
  FJOURNAL = {Canadian Journal of Mathematics. Journal Canadien de
              Math\'ematiques},
    VOLUME = {57},
      YEAR = {2005},
    NUMBER = {6},
     PAGES = {1178--1192},
      ISSN = {0008-414X,1496-4279},
   MRCLASS = {13D45},
  MRNUMBER = {2178557},
MRREVIEWER = {Siamak\ Yassemi},
       DOI = {10.4153/CJM-2005-046-4},
       URL = {https://doi.org/10.4153/CJM-2005-046-4},
}

@article {CHS10,
    AUTHOR = {Cutkosky, Steven Dale and Herzog, J\"urgen and Srinivasan,
              Hema},
     TITLE = {Asymptotic growth of algebras associated to powers of ideals},
   JOURNAL = {Math. Proc. Cambridge Philos. Soc.},
  FJOURNAL = {Mathematical Proceedings of the Cambridge Philosophical
              Society},
    VOLUME = {148},
      YEAR = {2010},
    NUMBER = {1},
     PAGES = {55--72},
      ISSN = {0305-0041,1469-8064},
   MRCLASS = {13A50 (13H15)},
  MRNUMBER = {2575372},
MRREVIEWER = {Yukihide\ Takayama},
       DOI = {10.1017/S0305004109990144},
       URL = {https://doi.org/10.1017/S0305004109990144},
}

@article {Cut14,
    AUTHOR = {Cutkosky, Steven Dale},
     TITLE = {Asymptotic multiplicities of graded families of ideals and
              linear series},
   JOURNAL = {Adv. Math.},
  FJOURNAL = {Advances in Mathematics},
    VOLUME = {264},
      YEAR = {2014},
     PAGES = {55--113},
      ISSN = {0001-8708,1090-2082},
   MRCLASS = {13H15 (14C20)},
  MRNUMBER = {3250280},
MRREVIEWER = {Sergio\ Mathew\ Da Silva},
       DOI = {10.1016/j.aim.2014.07.004},
       URL = {https://doi.org/10.1016/j.aim.2014.07.004},
}

@article{J,
  title={Hilbert polynomials and powers of ideals},
  author={Herzog, J{\"u}rgen and Puthenpurakal, Tony J and Verma, Jugal K},
  journal={Mathematical Proceedings of the Cambridge Philosophical Society},
  volume={145},
  number={3},
  pages={623--642},
  year={2008},
  publisher={Cambridge University Press}
}

@article {HHT07,
    AUTHOR = {Herzog, J\"urgen and Hibi, Takayuki and Trung, Ng\^o{} Vi\^et},
     TITLE = {Symbolic powers of monomial ideals and vertex cover algebras},
   JOURNAL = {Adv. Math.},
  FJOURNAL = {Advances in Mathematics},
    VOLUME = {210},
      YEAR = {2007},
    NUMBER = {1},
     PAGES = {304--322},
      ISSN = {0001-8708,1090-2082},
   MRCLASS = {13A30 (13H10)},
  MRNUMBER = {2298826},
MRREVIEWER = {Irena\ Swanson},
       DOI = {10.1016/j.aim.2006.06.007},
       URL = {https://doi.org/10.1016/j.aim.2006.06.007},
}

@article {Ree61a,
    AUTHOR = {Rees, David},
     TITLE = {A note on analytically unramified local rings},
   JOURNAL = {J. London Math. Soc.},
  FJOURNAL = {The Journal of the London Mathematical Society},
    VOLUME = {36},
      YEAR = {1961},
     PAGES = {24--28},
      ISSN = {0024-6107,1469-7750},
   MRCLASS = {13.95 (16.00)},
  MRNUMBER = {126465},
MRREVIEWER = {H.\ T.\ Muhly},
       DOI = {10.1112/jlms/s1-36.1.24},
       URL = {https://doi.org/10.1112/jlms/s1-36.1.24},
}

@article {Ree61b,
    AUTHOR = {Rees, David},
     TITLE = {{$a$}-transforms of local rings and a theorem on
              multiplicities of ideals},
   JOURNAL = {Proc. Cambridge Philos. Soc.},
  FJOURNAL = {Proceedings of the Cambridge Philosophical Society},
    VOLUME = {57},
      YEAR = {1961},
     PAGES = {8--17},
      ISSN = {0008-1981},
   MRCLASS = {16.00},
  MRNUMBER = {118750},
MRREVIEWER = {H.\ T.\ Muhly},
       DOI = {10.1017/s0305004100034800},
       URL = {https://doi.org/10.1017/s0305004100034800},
}

@article {KV10,
    AUTHOR = {Katz, Daniel and Validashti, Javid},
     TITLE = {Multiplicities and {R}ees valuations},
   JOURNAL = {Collect. Math.},
  FJOURNAL = {Universitat de Barcelona. Collectanea Mathematica},
    VOLUME = {61},
      YEAR = {2010},
    NUMBER = {1},
     PAGES = {1--24},
      ISSN = {0010-0757,2038-4815},
   MRCLASS = {13A30 (13H15)},
  MRNUMBER = {2604855},
MRREVIEWER = {Adela\ N.\ Vraciu},
       DOI = {10.1007/BF03191222},
       URL = {https://doi.org/10.1007/BF03191222},
}

@article {UV08,
    AUTHOR = {Ulrich, Bernd and Validashti, Javid},
     TITLE = {A criterion for integral dependence of modules},
   JOURNAL = {Math. Res. Lett.},
  FJOURNAL = {Mathematical Research Letters},
    VOLUME = {15},
      YEAR = {2008},
    NUMBER = {1},
     PAGES = {149--162},
      ISSN = {1073-2780},
   MRCLASS = {13C15 (13B21 13B22)},
  MRNUMBER = {2367181},
MRREVIEWER = {Aron\ Simis},
       DOI = {10.4310/MRL.2008.v15.n1.a13},
       URL = {https://doi.org/10.4310/MRL.2008.v15.n1.a13},
}

@article {UV11,
    AUTHOR = {Ulrich, Bernd and Validashti, Javid},
     TITLE = {Numerical criteria for integral dependence},
   JOURNAL = {Math. Proc. Cambridge Philos. Soc.},
  FJOURNAL = {Mathematical Proceedings of the Cambridge Philosophical
              Society},
    VOLUME = {151},
      YEAR = {2011},
    NUMBER = {1},
     PAGES = {95--102},
      ISSN = {0305-0041,1469-8064},
   MRCLASS = {13A30 (13A02)},
  MRNUMBER = {2801316},
MRREVIEWER = {Florian\ Enescu},
       DOI = {10.1017/S0305004111000144},
       URL = {https://doi.org/10.1017/S0305004111000144},
}

@article {CMM24,
    AUTHOR = {Cid-Ruiz, Yairon and Mohammadi, Fatemeh and Monin, Leonid},
     TITLE = {Multigraded algebras and multigraded linear series},
   JOURNAL = {J. Lond. Math. Soc. (2)},
  FJOURNAL = {Journal of the London Mathematical Society. Second Series},
    VOLUME = {109},
      YEAR = {2024},
    NUMBER = {3},
     PAGES = {Paper No. e12880, 38},
      ISSN = {0024-6107,1469-7750},
   MRCLASS = {13H15 (13A18 14M25 52B20)},
  MRNUMBER = {4754445},
MRREVIEWER = {Truong\ Le\ Hoang},
       DOI = {10.1112/jlms.12880},
       URL = {https://doi.org/10.1112/jlms.12880},
}

@article {Das21,
    AUTHOR = {Das, Suprajo},
     TITLE = {Epsilon multiplicity for graded algebras},
   JOURNAL = {J. Pure Appl. Algebra},
  FJOURNAL = {Journal of Pure and Applied Algebra},
    VOLUME = {225},
      YEAR = {2021},
    NUMBER = {10},
     PAGES = {Paper No. 106670, 21},
      ISSN = {0022-4049,1873-1376},
   MRCLASS = {13H15 (13A02 13D45 13H05 52B20)},
  MRNUMBER = {4207331},
MRREVIEWER = {Veronique\ Van Lierde},
       DOI = {10.1016/j.jpaa.2021.106670},
       URL = {https://doi.org/10.1016/j.jpaa.2021.106670},
}

@article {DDRV25,
    AUTHOR = {Das, Suprajo and Dubey, Saipriya and Roy, Sudeshna and Verma,
              Jugal K.},
     TITLE = {Computing epsilon multiplicities in graded algebras},
   JOURNAL = {J. Pure Appl. Algebra},
  FJOURNAL = {Journal of Pure and Applied Algebra},
    VOLUME = {229},
      YEAR = {2025},
    NUMBER = {11},
     PAGES = {Paper No. 108107, 28},
      ISSN = {0022-4049,1873-1376},
   MRCLASS = {13H15 (13A02 13A30 13D45 14C20)},
  MRNUMBER = {4970087},
       DOI = {10.1016/j.jpaa.2025.108107},
       URL = {https://doi.org/10.1016/j.jpaa.2025.108107},
}

@article {DRT25,
    AUTHOR = {Das, Suprajo and Roy, Sudeshna and Trivedi, Vijaylaxmi},
     TITLE = {Density functions for epsilon multiplicity and families of
              ideals},
   JOURNAL = {J. Lond. Math. Soc. (2)},
  FJOURNAL = {Journal of the London Mathematical Society. Second Series},
    VOLUME = {111},
      YEAR = {2025},
    NUMBER = {4},
     PAGES = {Paper No. e70155, 51},
      ISSN = {0024-6107,1469-7750},
   MRCLASS = {13H15 (13A30 13B22 14C17 14C20)},
  MRNUMBER = {4892394},
       DOI = {10.1112/jlms.70155},
       URL = {https://doi.org/10.1112/jlms.70155},
}

@article {DRT26,
    AUTHOR = {Das, Suprajo and Roy, Sudeshna and Trivedi, Vijaylaxmi},
     TITLE = {Numerical characterizations for integral dependence of graded
              ideals},
   JOURNAL = {Int. Math. Res. Not. IMRN},
  FJOURNAL = {International Mathematics Research Notices. IMRN},
      YEAR = {2026},
    NUMBER = {10},
     PAGES = {Paper No. rnag088, 22},
      ISSN = {1073-7928,1687-0247},
   MRCLASS = {13A30 (13D40 13E05 13H15)},
  MRNUMBER = {5069554},
       DOI = {10.1093/imrn/rnag088},
       URL = {https://doi.org/10.1093/imrn/rnag088},
}

@article {Tri18,
    AUTHOR = {Trivedi, Vijaylaxmi},
     TITLE = {Hilbert-{K}unz density function and {H}ilbert-{K}unz
              multiplicity},
   JOURNAL = {Trans. Amer. Math. Soc.},
  FJOURNAL = {Transactions of the American Mathematical Society},
    VOLUME = {370},
      YEAR = {2018},
    NUMBER = {12},
     PAGES = {8403--8428},
      ISSN = {0002-9947,1088-6850},
   MRCLASS = {13D40 (13H15 14H60 14N05)},
  MRNUMBER = {3864381},
MRREVIEWER = {Aryampilly\ V.\ Jayanthan},
       DOI = {10.1090/tran/7268},
       URL = {https://doi.org/10.1090/tran/7268},
}

@article {MT19,
    AUTHOR = {Mondal, Mandira and Trivedi, V.},
     TITLE = {Hilbert-{K}unz density function and asymptotic
              {H}ilbert-{K}unz multiplicity for projective toric varieties},
   JOURNAL = {J. Algebra},
  FJOURNAL = {Journal of Algebra},
    VOLUME = {520},
      YEAR = {2019},
     PAGES = {479--516},
      ISSN = {0021-8693,1090-266X},
   MRCLASS = {13D40 (13H15 14M25 52B20 52C22)},
  MRNUMBER = {3885209},
MRREVIEWER = {Christos\ Tatakis},
       DOI = {10.1016/j.jalgebra.2018.10.038},
       URL = {https://doi.org/10.1016/j.jalgebra.2018.10.038},
}

@article {MT20,
    AUTHOR = {Mondal, Mandira and Trivedi, Vijaylaxmi},
     TITLE = {Density function for the second coefficient of the
              {H}ilbert-{K}unz function on projective toric varieties},
   JOURNAL = {J. Algebraic Combin.},
  FJOURNAL = {Journal of Algebraic Combinatorics. An International Journal},
    VOLUME = {51},
      YEAR = {2020},
    NUMBER = {3},
     PAGES = {317--351},
      ISSN = {0925-9899,1572-9192},
   MRCLASS = {14M25 (13D40 13H15 52B20)},
  MRNUMBER = {4096329},
MRREVIEWER = {Margherita\ Barile},
       DOI = {10.1007/s10801-019-00877-8},
       URL = {https://doi.org/10.1007/s10801-019-00877-8},
}

@article {Tri20,
    AUTHOR = {Trivedi, Vijaylaxmi},
     TITLE = {Nondiscreteness of {$F$}-thresholds},
   JOURNAL = {Math. Res. Lett.},
  FJOURNAL = {Mathematical Research Letters},
    VOLUME = {27},
      YEAR = {2020},
    NUMBER = {6},
     PAGES = {1885--1895},
      ISSN = {1073-2780,1945-001X},
   MRCLASS = {13A35 (14G17)},
  MRNUMBER = {4216609},
MRREVIEWER = {Geoffrey\ D.\ Dietz},
       DOI = {10.4310/MRL.2020.v27.n6.a13},
       URL = {https://doi.org/10.4310/MRL.2020.v27.n6.a13},
}

@article {TW21,
    AUTHOR = {Trivedi, Vijaylaxmi and Watanabe, Kei-Ichi},
     TITLE = {Hilbert-{K}unz density functions and {$F$}-thresholds},
   JOURNAL = {J. Algebra},
  FJOURNAL = {Journal of Algebra},
    VOLUME = {567},
      YEAR = {2021},
     PAGES = {533--563},
      ISSN = {0021-8693,1090-266X},
   MRCLASS = {13A35},
  MRNUMBER = {4163076},
MRREVIEWER = {Alessandro\ De Stefani},
       DOI = {10.1016/j.jalgebra.2020.09.025},
       URL = {https://doi.org/10.1016/j.jalgebra.2020.09.025},
}

@article {TW23,
    AUTHOR = {Trivedi, Vijaylaxmi},
     TITLE = {The {H}ilbert-{K}unz density functions of quadric
              hypersurfaces},
   JOURNAL = {Adv. Math.},
  FJOURNAL = {Advances in Mathematics},
    VOLUME = {430},
      YEAR = {2023},
     PAGES = {Paper No. 109207, 63},
      ISSN = {0001-8708,1090-2082},
   MRCLASS = {13D40 (14J60)},
  MRNUMBER = {4617943},
MRREVIEWER = {Tony\ J.\ Puthenpurakal},
       DOI = {10.1016/j.aim.2023.109207},
       URL = {https://doi.org/10.1016/j.aim.2023.109207},
}

@article {DST15,
    AUTHOR = {Dumnicki, Marcin and Szpond, Justyna and Tutaj-Gasi\'nska,
              Halszka},
     TITLE = {Asymptotic {H}ilbert polynomials and limiting shapes},
   JOURNAL = {J. Pure Appl. Algebra},
  FJOURNAL = {Journal of Pure and Applied Algebra},
    VOLUME = {219},
      YEAR = {2015},
    NUMBER = {10},
     PAGES = {4446--4457},
      ISSN = {0022-4049,1873-1376},
   MRCLASS = {13F20 (13D40 14N20)},
  MRNUMBER = {3346500},
MRREVIEWER = {Dan\ Yan},
       DOI = {10.1016/j.jpaa.2015.02.026},
       URL = {https://doi.org/10.1016/j.jpaa.2015.02.026},
}

@article {May14,
    AUTHOR = {Mayes, Sarah},
     TITLE = {The asymptotic behaviour of symbolic generic initial systems
              of generic points},
   JOURNAL = {J. Pure Appl. Algebra},
  FJOURNAL = {Journal of Pure and Applied Algebra},
    VOLUME = {218},
      YEAR = {2014},
    NUMBER = {3},
     PAGES = {381--390},
      ISSN = {0022-4049,1873-1376},
   MRCLASS = {14N25 (13C40 13D10 13F20)},
  MRNUMBER = {3124204},
MRREVIEWER = {Ralf\ Fr\"oberg},
       DOI = {10.1016/j.jpaa.2013.06.002},
       URL = {https://doi.org/10.1016/j.jpaa.2013.06.002},
}

@book {Har77,
    AUTHOR = {Hartshorne, Robin},
     TITLE = {Algebraic geometry},
    SERIES = {Graduate Texts in Mathematics},
    VOLUME = {No. 52},
 PUBLISHER = {Springer-Verlag, New York-Heidelberg},
      YEAR = {1977},
     PAGES = {xvi+496},
      ISBN = {0-387-90244-9},
   MRCLASS = {14-01},
  MRNUMBER = {463157},
MRREVIEWER = {Robert\ Speiser},
}

@book {Har67,
    AUTHOR = {Hartshorne, Robin},
     TITLE = {Local cohomology},
    SERIES = {Lecture Notes in Mathematics},
    VOLUME = {No. 41},
      NOTE = {A seminar given by A. Grothendieck, Harvard University, Fall,
              1961},
 PUBLISHER = {Springer-Verlag, Berlin-New York},
      YEAR = {1967},
     PAGES = {vi+106},
   MRCLASS = {14.55 (18.00)},
  MRNUMBER = {224620},
MRREVIEWER = {F.\ Oort},
}

@article {LM09,
    AUTHOR = {Lazarsfeld, Robert and Musta\c{t}\u{a}, Mircea},
     TITLE = {Convex bodies associated to linear series},
   JOURNAL = {Ann. Sci. \'Ec. Norm. Sup\'er. (4)},
  FJOURNAL = {Annales Scientifiques de l'\'Ecole Normale Sup\'erieure.
              Quatri\`eme S\'erie},
    VOLUME = {42},
      YEAR = {2009},
    NUMBER = {5},
     PAGES = {783--835},
      ISSN = {0012-9593,1873-2151},
   MRCLASS = {14C20 (14E05)},
  MRNUMBER = {2571958},
MRREVIEWER = {Zach\ Teitler},
       DOI = {10.24033/asens.2109},
       URL = {https://doi.org/10.24033/asens.2109},
}

@article {Oko96,
    AUTHOR = {Okounkov, Andrei},
     TITLE = {Brunn-{M}inkowski inequality for multiplicities},
   JOURNAL = {Invent. Math.},
  FJOURNAL = {Inventiones Mathematicae},
    VOLUME = {125},
      YEAR = {1996},
    NUMBER = {3},
     PAGES = {405--411},
      ISSN = {0020-9910,1432-1297},
   MRCLASS = {58F05 (14L30 52B11 58F06)},
  MRNUMBER = {1400312},
       DOI = {10.1007/s002220050081},
       URL = {https://doi.org/10.1007/s002220050081},
}

@article {KK12,
    AUTHOR = {Kaveh, Kiumars and Khovanskii, A. G.},
     TITLE = {Newton-{O}kounkov bodies, semigroups of integral points,
              graded algebras and intersection theory},
   JOURNAL = {Ann. of Math. (2)},
  FJOURNAL = {Annals of Mathematics. Second Series},
    VOLUME = {176},
      YEAR = {2012},
    NUMBER = {2},
     PAGES = {925--978},
      ISSN = {0003-486X,1939-8980},
   MRCLASS = {52C07 (14M25 20M14 52B20)},
  MRNUMBER = {2950767},
MRREVIEWER = {Alexander\ A.\ Borisov},
       DOI = {10.4007/annals.2012.176.2.5},
       URL = {https://doi.org/10.4007/annals.2012.176.2.5},
}

@incollection {Gre98,
    AUTHOR = {Green, Mark L.},
     TITLE = {Generic initial ideals},
 BOOKTITLE = {Six lectures on commutative algebra ({B}ellaterra, 1996)},
    SERIES = {Progr. Math.},
    VOLUME = {166},
     PAGES = {119--186},
 PUBLISHER = {Birkh\"auser, Basel},
      YEAR = {1998},
      ISBN = {3-7643-5951-X},
   MRCLASS = {13F20 (13P10)},
  MRNUMBER = {1648665},
MRREVIEWER = {J.\ M.\ Landsberg},
}

@article {FKL16,
    AUTHOR = {Fulger, Mihai and Koll\'ar, J\'anos and Lehmann, Brian},
     TITLE = {Volume and {H}ilbert function of {$\mathbb{R}$}-divisors},
   JOURNAL = {Michigan Math. J.},
  FJOURNAL = {Michigan Mathematical Journal},
    VOLUME = {65},
      YEAR = {2016},
    NUMBER = {2},
     PAGES = {371--387},
      ISSN = {0026-2285,1945-2365},
   MRCLASS = {14C20},
  MRNUMBER = {3510912},
MRREVIEWER = {Montserrat\ Teixidor i Bigas},
       DOI = {10.1307/mmj/1465329018},
       URL = {https://doi.org/10.1307/mmj/1465329018},
}

@article {BKS04,
    AUTHOR = {Bauer, Thomas and K\"uronya, Alex and Szemberg, Tomasz},
     TITLE = {Zariski chambers, volumes, and stable base loci},
   JOURNAL = {J. Reine Angew. Math.},
  FJOURNAL = {Journal f\"ur die Reine und Angewandte Mathematik. [Crelle's
              Journal]},
    VOLUME = {576},
      YEAR = {2004},
     PAGES = {209--233},
      ISSN = {0075-4102,1435-5345},
   MRCLASS = {14C20 (14J26 14J28)},
  MRNUMBER = {2099205},
MRREVIEWER = {Flaminio\ Flamini},
       DOI = {10.1515/crll.2004.090},
       URL = {https://doi.org/10.1515/crll.2004.090},
}

@article {KLM13,
    AUTHOR = {K\"uronya, Alex and Lozovanu, Victor and Maclean, Catriona},
     TITLE = {Volume functions of linear series},
   JOURNAL = {Math. Ann.},
  FJOURNAL = {Mathematische Annalen},
    VOLUME = {356},
      YEAR = {2013},
    NUMBER = {2},
     PAGES = {635--652},
      ISSN = {0025-5831,1432-1807},
   MRCLASS = {14C20},
  MRNUMBER = {3048610},
MRREVIEWER = {Halszka\ Tutaj-Gasi\'nska},
       DOI = {10.1007/s00208-012-0859-0},
       URL = {https://doi.org/10.1007/s00208-012-0859-0},
}

@book {Gru07,
    AUTHOR = {Gruber, Peter M.},
     TITLE = {Convex and Discrete Geometry},
    SERIES = {Grundlehren der mathematischen Wissenschaften},
    VOLUME = {336},
 PUBLISHER = {Springer, Berlin},
      YEAR = {2007},
     PAGES = {xiv+578},
      ISBN = {978-3-540-71132-2},
   MRCLASS = {52-02 (11-02 46B20 49-02 49J53 90C25)},
  MRNUMBER = {2335496},
MRREVIEWER = {Aleksandr\ Koldobsky},
}

@article {HHT02,
    AUTHOR = {Herzog, J\"urgen and L\^e{} Tu\^an Hoa and Ng\^o{} Vi\^et
              Trung},
     TITLE = {Asymptotic linear bounds for the {C}astelnuovo-{M}umford
              regularity},
   JOURNAL = {Trans. Amer. Math. Soc.},
  FJOURNAL = {Transactions of the American Mathematical Society},
    VOLUME = {354},
      YEAR = {2002},
    NUMBER = {5},
     PAGES = {1793--1809},
      ISSN = {0002-9947,1088-6850},
   MRCLASS = {13D45},
  MRNUMBER = {1881017},
MRREVIEWER = {P.\ Schenzel},
       DOI = {10.1090/S0002-9947-02-02932-X},
       URL = {https://doi.org/10.1090/S0002-9947-02-02932-X},
}

@article {KV89,
    AUTHOR = {Katz, Daniel and Verma, Jugal K.},
     TITLE = {Extended {R}ees algebras and mixed multiplicities},
   JOURNAL = {Math. Z.},
  FJOURNAL = {Mathematische Zeitschrift},
    VOLUME = {202},
      YEAR = {1989},
    NUMBER = {1},
     PAGES = {111--128},
      ISSN = {0025-5874,1432-1823},
   MRCLASS = {13H15 (13H10)},
  MRNUMBER = {1007742},
MRREVIEWER = {Eduard\ Bod'a},
       DOI = {10.1007/BF01180686},
       URL = {https://doi.org/10.1007/BF01180686},
}

@article {Tru01,
    AUTHOR = {Trung, Ng\^o{} Vi\^et},
     TITLE = {Positivity of mixed multiplicities},
   JOURNAL = {Math. Ann.},
  FJOURNAL = {Mathematische Annalen},
    VOLUME = {319},
      YEAR = {2001},
    NUMBER = {1},
     PAGES = {33--63},
      ISSN = {0025-5831,1432-1807},
   MRCLASS = {13H15 (13A02 13A30 13D40)},
  MRNUMBER = {1812818},
MRREVIEWER = {Catalin\ Ciuperca},
       DOI = {10.1007/PL00004429},
       URL = {https://doi.org/10.1007/PL00004429},
}

@article {Bha57,
    AUTHOR = {Bhattacharya, P. B.},
     TITLE = {The {H}ilbert function of two ideals},
   JOURNAL = {Proc. Cambridge Philos. Soc.},
  FJOURNAL = {Proceedings of the Cambridge Philosophical Society},
    VOLUME = {53},
      YEAR = {1957},
     PAGES = {568--575},
      ISSN = {0008-1981},
   MRCLASS = {09.3X},
  MRNUMBER = {89835},
MRREVIEWER = {H.\ T.\ Muhly},
}

@article {BCG+16,
    AUTHOR = {Bocci, Cristiano and Cooper, Susan and Guardo, Elena and
              Harbourne, Brian and Janssen, Mike and Nagel, Uwe and
              Seceleanu, Alexandra and Van Tuyl, Adam and Vu, Thanh},
     TITLE = {The {W}aldschmidt constant for squarefree monomial ideals},
   JOURNAL = {J. Algebraic Combin.},
  FJOURNAL = {Journal of Algebraic Combinatorics. An International Journal},
    VOLUME = {44},
      YEAR = {2016},
    NUMBER = {4},
     PAGES = {875--904},
      ISSN = {0925-9899,1572-9192},
   MRCLASS = {13F20 (13A02 13F55 14N05 14N20)},
  MRNUMBER = {3566223},
MRREVIEWER = {Christopher\ A.\ Francisco},
       DOI = {10.1007/s10801-016-0693-7},
       URL = {https://doi.org/10.1007/s10801-016-0693-7},
}

@article {Cut86,
    AUTHOR = {Cutkosky, Steven Dale},
     TITLE = {Zariski decomposition of divisors on algebraic varieties},
   JOURNAL = {Duke Math. J.},
  FJOURNAL = {Duke Mathematical Journal},
    VOLUME = {53},
      YEAR = {1986},
    NUMBER = {1},
     PAGES = {149--156},
      ISSN = {0012-7094,1547-7398},
   MRCLASS = {14C20 (14J40)},
  MRNUMBER = {835801},
MRREVIEWER = {Takao\ Fujita},
       DOI = {10.1215/S0012-7094-86-05309-3},
       URL = {https://doi.org/10.1215/S0012-7094-86-05309-3},
}

@article {JL23,
    AUTHOR = {Jiang, Chen and Li, Zhiyuan},
     TITLE = {Algebraic reverse {K}hovanskii-{T}eissier inequality via
              {O}kounkov bodies},
   JOURNAL = {Math. Z.},
  FJOURNAL = {Mathematische Zeitschrift},
    VOLUME = {305},
      YEAR = {2023},
    NUMBER = {2},
     PAGES = {Paper No. 26, 14},
      ISSN = {0025-5874,1432-1823},
   MRCLASS = {14C20 (14C17 14M25)},
  MRNUMBER = {4645768},
MRREVIEWER = {Pietro\ Sabatino},
       DOI = {10.1007/s00209-023-03349-9},
       URL = {https://doi.org/10.1007/s00209-023-03349-9},
}

@article {LX17,
    AUTHOR = {Lehmann, Brian and Xiao, Jian},
     TITLE = {Correspondences between convex geometry and complex geometry},
   JOURNAL = {\'Epijournal G\'eom. Alg\'ebrique},
  FJOURNAL = {\'Epijournal de G\'eom\'etrie Alg\'ebrique. EPIGA},
    VOLUME = {1},
      YEAR = {2017},
     PAGES = {Art. 6, 29},
      ISSN = {2491-6765},
   MRCLASS = {14C20 (32Q15 52A39)},
  MRNUMBER = {3743109},
MRREVIEWER = {Eugenii\ Shustin},
       DOI = {10.46298/epiga.2017.volume1.2038},
       URL = {https://doi.org/10.46298/epiga.2017.volume1.2038},
}

@article {Huh12,
    AUTHOR = {Huh, June},
     TITLE = {Milnor numbers of projective hypersurfaces and the chromatic
              polynomial of graphs},
   JOURNAL = {J. Amer. Math. Soc.},
  FJOURNAL = {Journal of the American Mathematical Society},
    VOLUME = {25},
      YEAR = {2012},
    NUMBER = {3},
     PAGES = {907--927},
      ISSN = {0894-0347,1088-6834},
   MRCLASS = {14B05 (05B35 14C17)},
  MRNUMBER = {2904577},
MRREVIEWER = {Paolo\ Aluffi},
       DOI = {10.1090/S0894-0347-2012-00731-0},
       URL = {https://doi.org/10.1090/S0894-0347-2012-00731-0},
}

@article {TV07,
    AUTHOR = {Trung, Ng\^o{} Vi\^et and Verma, Jugal K.},
     TITLE = {Mixed multiplicities of ideals versus mixed volumes of
              polytopes},
   JOURNAL = {Trans. Amer. Math. Soc.},
  FJOURNAL = {Transactions of the American Mathematical Society},
    VOLUME = {359},
      YEAR = {2007},
    NUMBER = {10},
     PAGES = {4711--4727},
      ISSN = {0002-9947,1088-6850},
   MRCLASS = {13H15 (05E99 13A30 13D40 52B20)},
  MRNUMBER = {2320648},
MRREVIEWER = {Mihai\ Cipu},
       DOI = {10.1090/S0002-9947-07-04054-8},
       URL = {https://doi.org/10.1090/S0002-9947-07-04054-8},
}

@article {HX24,
    AUTHOR = {Hu, Jiajun and Xiao, Jian},
     TITLE = {Intersection theoretic inequalities via {L}orentzian
              polynomials},
   JOURNAL = {Math. Ann.},
  FJOURNAL = {Mathematische Annalen},
    VOLUME = {390},
      YEAR = {2024},
    NUMBER = {2},
     PAGES = {2859--2896},
      ISSN = {0025-5831,1432-1807},
   MRCLASS = {14C30 (05E05 32Q15 52A40)},
  MRNUMBER = {4801842},
       DOI = {10.1007/s00208-024-02822-y},
       URL = {https://doi.org/10.1007/s00208-024-02822-y},
}

@article {GMRV23,
    AUTHOR = {Goel, Kriti and Mukundan, Vivek and Roy, Sudeshna and Verma,
              Jugal K.},
     TITLE = {Algorithms for computing mixed multiplicities, mixed volumes,
              and sectional {M}ilnor numbers},
   JOURNAL = {J. Softw. Algebra Geom.},
  FJOURNAL = {Journal of Software for Algebra and Geometry},
    VOLUME = {13},
      YEAR = {2023},
    NUMBER = {1},
     PAGES = {1--12},
      ISSN = {1948-7916},
   MRCLASS = {13-04 (13A30 13H15)},
  MRNUMBER = {4633631},
       DOI = {10.2140/jsag.2023.13.1},
       URL = {https://doi.org/10.2140/jsag.2023.13.1},
}

@article {CM22,
    AUTHOR = {Cid-Ruiz, Yairon and Monta\~no, Jonathan},
     TITLE = {Mixed multiplicities of graded families of ideals},
   JOURNAL = {J. Algebra},
  FJOURNAL = {Journal of Algebra},
    VOLUME = {590},
      YEAR = {2022},
     PAGES = {394--412},
      ISSN = {0021-8693,1090-266X},
   MRCLASS = {13H15},
  MRNUMBER = {4332034},
MRREVIEWER = {Alessandro\ De Stefani},
       DOI = {10.1016/j.jalgebra.2021.10.010},
       URL = {https://doi.org/10.1016/j.jalgebra.2021.10.010},
}

@article {CEL01,
    AUTHOR = {Cutkosky, Steven Dale and Ein, Lawrence and Lazarsfeld,
              Robert},
     TITLE = {Positivity and complexity of ideal sheaves},
   JOURNAL = {Math. Ann.},
  FJOURNAL = {Mathematische Annalen},
    VOLUME = {321},
      YEAR = {2001},
    NUMBER = {2},
     PAGES = {213--234},
      ISSN = {0025-5831,1432-1807},
   MRCLASS = {14F05 (13F20)},
  MRNUMBER = {1866486},
MRREVIEWER = {Thomas\ Bauer},
       DOI = {10.1007/s002080100220},
       URL = {https://doi.org/10.1007/s002080100220},
}

@article {CHT99,
    AUTHOR = {Cutkosky, Steven Dale and Herzog, J\"urgen and Trung, Ng\^o{}
              Vi\^et},
     TITLE = {Asymptotic behaviour of the {C}astelnuovo-{M}umford
              regularity},
   JOURNAL = {Compositio Math.},
  FJOURNAL = {Compositio Mathematica},
    VOLUME = {118},
      YEAR = {1999},
    NUMBER = {3},
     PAGES = {243--261},
      ISSN = {0010-437X,1570-5846},
   MRCLASS = {13D45 (13A30 13C99 14F17)},
  MRNUMBER = {1711319},
MRREVIEWER = {Vincenzo\ Di Gennaro},
       DOI = {10.1023/A:1001559912258},
       URL = {https://doi.org/10.1023/A:1001559912258},
}

@article {TW05,
    AUTHOR = {Trung, Ng\^o{} Vi\^et and Wang, Hsin-Ju},
     TITLE = {On the asymptotic linearity of {C}astelnuovo-{M}umford
              regularity},
   JOURNAL = {J. Pure Appl. Algebra},
  FJOURNAL = {Journal of Pure and Applied Algebra},
    VOLUME = {201},
      YEAR = {2005},
    NUMBER = {1-3},
     PAGES = {42--48},
      ISSN = {0022-4049,1873-1376},
   MRCLASS = {13D45},
  MRNUMBER = {2158746},
MRREVIEWER = {Isabel\ Bermejo},
       DOI = {10.1016/j.jpaa.2004.12.043},
       URL = {https://doi.org/10.1016/j.jpaa.2004.12.043},
}

@article {Kod00,
    AUTHOR = {Kodiyalam, Vijay},
     TITLE = {Asymptotic behaviour of {C}astelnuovo-{M}umford regularity},
   JOURNAL = {Proc. Amer. Math. Soc.},
  FJOURNAL = {Proceedings of the American Mathematical Society},
    VOLUME = {128},
      YEAR = {2000},
    NUMBER = {2},
     PAGES = {407--411},
      ISSN = {0002-9939,1088-6826},
   MRCLASS = {13D45 (14B15)},
  MRNUMBER = {1621961},
MRREVIEWER = {P.\ Schenzel},
       DOI = {10.1090/S0002-9939-99-05020-0},
       URL = {https://doi.org/10.1090/S0002-9939-99-05020-0},
}

@article {Kle66,
    AUTHOR = {Kleiman, Steven L.},
     TITLE = {Toward a numerical theory of ampleness},
   JOURNAL = {Ann. of Math. (2)},
  FJOURNAL = {Annals of Mathematics. Second Series},
    VOLUME = {84},
      YEAR = {1966},
     PAGES = {293--344},
      ISSN = {0003-486X},
   MRCLASS = {14.55 (14.10)},
  MRNUMBER = {206009},
       DOI = {10.2307/1970447},
       URL = {https://doi.org/10.2307/1970447},
}

@incollection {Tei73,
    AUTHOR = {Teissier, Bernard},
     TITLE = {Cycles \'evanescents, sections planes et conditions de
              {W}hitney},
 BOOKTITLE = {Singularit\'es \`a{} {C}arg\`ese ({R}encontre {S}ingularit\'es
              {G}\'eom. {A}nal., {I}nst. \'Etudes {S}ci., {C}arg\`ese,
              1972)},
    SERIES = {Ast\'erisque},
    VOLUME = {Nos. 7 et 8},
     PAGES = {285--362},
 PUBLISHER = {Soc. Math. France, Paris},
      YEAR = {1973},
   MRCLASS = {32C40},
  MRNUMBER = {374482},
MRREVIEWER = {J.\ A.\ Morrow},
}
\end{document}